\pdfoutput=1

\documentclass{amsart}

\calclayout

\usepackage{hyperref}
\usepackage[lite]{amsrefs}
\usepackage{amsmath,amssymb,latexsym}

\usepackage{tikz-cd}
\usetikzlibrary{calc}
\usepackage{ytableau}

\def\paragraph{\subsection}
\makeatletter
\def\@sect#1#2#3#4#5#6[#7]#8{%
  \edef\@toclevel{\ifnum#2=\@m 0\else\number#2\fi}%
  \ifnum #2>\c@secnumdepth \let\@secnumber\@empty
  \else \@xp\let\@xp\@secnumber\csname the#1\endcsname\fi
  \@tempskipa #5\relax
  \ifnum #2>\c@secnumdepth
    \let\@svsec\@empty
  \else
    \refstepcounter{#1}%
    \edef\@secnumpunct{%
      \ifdim\@tempskipa>\z@ 
        \@ifnotempty{#8}{.\@nx\enspace}%
      \else
        \@ifempty{#8}{.}{.\@nx\enspace}%
      \fi
    }%
      \ifnum #2=\tw@ \def\@secnumfont{\bfseries}\fi{}%
    \protected@edef\@svsec{%
      \ifnum#2<\@m
        \@ifundefined{#1name}{}{%
          \ignorespaces\csname #1name\endcsname\space
        }%
      \fi
      \@seccntformat{#1}%
    }%
  \fi
  \ifdim \@tempskipa>\z@ 
    \begingroup #6\relax
    \@hangfrom{\hskip #3\relax\@svsec}{\interlinepenalty\@M #8\par}%
    \endgroup
    \ifnum#2>\@m \else \@tocwrite{#1}{#8}\fi
  \else
  \def\@svsechd{#6\hskip #3\@svsec
    \@ifnotempty{#8}{\ignorespaces#8\unskip
       \@addpunct.}%
    \ifnum#2>\@m \else \@tocwrite{#1}{#8}\fi
  }%
  \fi
  \global\@nobreaktrue
  \@xsect{#5}}
\makeatother

\newcommand{\hatzero}{\hat{0}}
\newcommand{\longequiv}{\Longleftrightarrow}
\newcommand{\mbbN}{\mathbb{N}}
\newcommand{\mbbO}{\mathbb{O}}
\newcommand{\mbbP}{\mathbb{P}}
\newcommand{\mbbQ}{\mathbb{Q}}
\newcommand{\mbbSY}{\mathbb{SY}}

\newcommand{\mbbY}{\mathbb{Y}}
\renewcommand{\color}{\mathnormal{color}}
\newcommand{\nnull}{\mathnormal{null}}

\newcommand{\horizline}[2]{
    \draw ($(\tikzcdmatrixname-#1-#2.south)!0.5!(\tikzcdmatrixname-\the\numexpr #1+1\relax-#2.north)$)
    coordinate (y)
    ($(\tikzcdmatrixname-#1-\the\numexpr #2-1\relax.east)!0.5!(\tikzcdmatrixname-#1-#2.west)$)
    coordinate (xl)
    ($(\tikzcdmatrixname-#1-#2.east)!0.5!(\tikzcdmatrixname-#1-\the\numexpr #2+1\relax.west)$)
    coordinate (xu)
    (xl|-y) -- (xu|-y);
}
\newcommand{\horizlinex}[3]{
    \draw ($(\tikzcdmatrixname-#1-#2.south)!0.5!(\tikzcdmatrixname-\the\numexpr #1+1\relax-#2.north)$)
    coordinate (y)
    ($(\tikzcdmatrixname-#1-\the\numexpr #2-1\relax.east)!0.5!(\tikzcdmatrixname-#1-#2.west)$)
    coordinate (xl)
    ($(\tikzcdmatrixname-#1-#3.east)!0.5!(\tikzcdmatrixname-#1-\the\numexpr #3+1\relax.west)$)
    coordinate (xu)
    (xl|-y) -- (xu|-y);
}
\newcommand{\vertline}[2]{
    \draw ($(\tikzcdmatrixname-#1-#2.east)!0.5!(\tikzcdmatrixname-#1-\the\numexpr #2+1\relax.west)$)
    coordinate (x)
    ($(\tikzcdmatrixname-\the\numexpr #1-1\relax-#2.south)!0.5!(\tikzcdmatrixname-#1-#2.north)$)
    coordinate (yl)
    ($(\tikzcdmatrixname-#1-#2.south)!0.5!(\tikzcdmatrixname-\the\numexpr #1+1\relax-#2.north)$)
    coordinate (yu)
    (x|-yl) -- (x|-yu);
}
\newcommand{\vertlinex}[3]{
    \draw ($(\tikzcdmatrixname-#1-#3.east)!0.5!(\tikzcdmatrixname-#1-\the\numexpr #3+1\relax.west)$)
    coordinate (x)
    ($(\tikzcdmatrixname-\the\numexpr #1-1\relax-#3.south)!0.5!(\tikzcdmatrixname-#1-#3.north)$)
    coordinate (yl)
    ($(\tikzcdmatrixname-#2-#3.south)!0.5!(\tikzcdmatrixname-\the\numexpr #2+1\relax-#3.north)$)
    coordinate (yu)
    (x|-yl) -- (x|-yu);
}

\newcommand{\growthtableauxh}[3]{\hbox to \linewidth{\hfill\vtop{\vspace{0pt}\hbox{#1}}\hfill
\vtop{\vspace{\baselineskip}\hbox{#2}\vskip 1in\hbox{#3}}\hfill}}
\newcommand{\growthtableauxv}[3]{#1\vskip \baselineskip\hbox to \linewidth
{\hfill\hbox{#2}\hfill\hbox{#3}\hfill}}

\newtheorem{theorem}{Theorem}[section]
\newtheorem{proposition}[theorem]{Proposition}

\begin{document}

\title[Graphical representation of insertion algorithms]{On the
  combinatorics of tableaux ---
  Graphical representation of insertion algorithms}
\author{Dale R. Worley}
\email{worley@alum.mit.edu}
\thanks{The author would like to credit Linux, Emacs, \TeX, \LaTeX,
  AMS-\LaTeX, Ti\textit{k}Z, \texttt{tikzcd}, and \texttt{ytableau}
  for typesetting this paper.}
\thanks{The author would like to thank arXiv, Sci-Hub, Internet
  Archive, and Google Scholar for providing a world-class research
  library.
  Calvin Mooers predicted this situation in \cite{Mooers1959}:
  \begin{quotation}
    In my discussion I shall bypass treating such useful and imminent
    tasks as the use of machines to store, transfer, and emit texts, so
    that at the time that you need to refer to a paper, even in an obscure
    journal, you can have a copy in hand within, say, twenty-four hours.
    [\ldots]  Neither shall I consider the application of machines to the
    integration of national and international library systems so that at
    any first-rate library, you will have at your command the catalogs of
    the major collections of the world.  These are all coming---but it
    should be noted with respect to them that the problems of human
    cooperation ranging from person-to-person to nation-to-nation
    cooperation are more serious than some of the machine and technical
    problems involved.
  \end{quotation}
 \hskip 0pt
}
\date{Feb 17, 2025} 

\begin{abstract}
Many algorithms for inserting elements into tableaux are known,
starting with the Robinson-Schensted algorithm.
Much of those
processes can be incorporated into the general framework of Fomin's
``growth diagrams''.
Even for single types of tableaux, there are various alternative
insertion algorithms and, due to the varying ways they are described,
the relationships
between the algorithms can be obscure.
The distinguishing features of many algorithms can be codified into
graphic ``insertion diagrams'' which make
important aspects of the algorithms immediately apparent.
We use insertion diagrams to build a graphic catalog or picture book
of many of the tableau insertion algorithms in the literature.
\end{abstract}

\maketitle

\tableofcontents
\listoffigures

\section{Introduction}

Many algorithms are known for inserting elements into Young tableaux,
starting with the Robinson-Schensted algorithm \cite{Schen1961}.
The algorithm has been extended in various ways.  The most significant
extension is Knuth's version, which
allows multiple equal elements to be inserted into a tableau.  However
we will not be pursuing that extension in this paper.

The type of tableau involved can be changed to shifted Young tableaux
(which we will examine in detail),
or the more complex ``oscillating tableaux'' (which we will not examine).
These algorithms can
can be incorporated into the general framework of ``growth diagrams''
\ref{sec:growth-diagrams}, created by Fomin \cites{Fom1994a,Fom1995a}.

In all of these situations, multiple algorithms can be constructed
which serve the same or very similar enumerative purposes, and often
no one algorithm seems to be ``natural''.
Due to their complexity and the varying ways they are
described, the relationships between the algorithms can be obscure.
For many algorithms, the distinguishing features can be codified into
graphic ``insertion diagrams'' \ref{sec:insert-diagram} which make
important aspects of the algorithms immediately apparent.
We use insertion diagrams to build a graphic catalog or picture book
of many of the tableau insertion algorithms in the literature.

The concept of a ``generic'' insertion algorithm, an algorithm that by the
choice of parameters (such as our insertion diagrams) generates a
family of insertion algorithms, was
pioneered by \cite{McLar1986},\footnote{\cite{GarMcLar1987} provides
most of the parts of \cite{McLar1986} dealing with unshifted tableaux, and is
more easily available.} a fact I discovered after this paper
was well underway.\footnote{\cite{McLar1986} also anticipates parts of
\cite{Fom1995a}.}

The power of insertion diagrams as a method of describing insertion
algorithms is shown by two algorithms we introduce here, ``jitter''
insertion~\ref{sec:jitter-insertion} and double-circle
insertion~\ref{sec:double-circle-insertion}, which we define not because they
are valuable, but because insertion diagrams make easy defining algorithms
with specific properties.

We introduce the extension of ``biweighting''~\ref{sec:biweighted},
factoring the weight
function that is applied to the edges of the ``descending'' graph in Fomin's dual
graded graph construction~\ref{sec:weight} into two weight functions,
one applied to the edges of the ``ascending'' graph and one to the
edges of the ``descending''
graph.  This is a small change, but it allows insertion algorithms that
apply ``circling'' to the $P$ tableau to be described with growth
processes.

We include a discussion of McLarnan's underappreciated ``shifted
column insertion'' algorithm.~\ref{sec:shifted-column}
It is biweighted and nearly inversion self-dual.

\section{Background}

The machinery underlying tableau insertion algorithms is defined in
many works (e.g.~\cites{Fom1994a,Fom1995a}, \cite{Haim1989a},
\cite{Roby1991a}).  This exposition will specialize the machinery for
the situations we will be considering; applying known generalizations
to this exposition is in most cases straightforward.  A whirlwind
summary follows.

\paragraph{Basics}

We use $\mbbP$ for $\{1, 2, 3, \ldots\}$ and $\mbbN$ for
$\{0, 1, 2, \ldots\}$.
We use $[n]$ for $\{1, 2, 3, \ldots, n\}$.

We use $\uplus$ to denote generally joining two structures that are assumed to
be disjoint.  One use is the union of two disjoint sets.  Other uses
are for the appropriate ``direct sum'' of two structures of the same
type.  In some cases, we abuse the notation by implicitly labeling the
elements of the two structures with 1 or 2 (resp.)\ to make the
structures disjoint before taking their direct sum.

If a set $Y$ is the set $X$ plus one additional element, we use
$Y - X$ to denote that element.  Hence $(X \uplus \{y\}) - X = y$.

In a poset, we use $p \gtrdot q$ to mean that $p$ covers $q$, that is,
$p > q$ and there is no $r$ such that $p > r > q$.  We use $p \lessdot q$
to mean that $p$ is covered by $q$, that is, $q \gtrdot p$.

A poset $P$ is \textit{finitary} if every principal order ideal of $P$
is finite.\cite{Stan2012a}*{sec.~3.4}
A poset $P$ has \textit{finite covers} if for every $p \in P$,
$\{q \in P: p \gtrdot q\}$ and $\{q \in P: p \lessdot q\}$ are both
finite.  A finitary poset has finite covers downward but may not have
finite covers upward.

A finitary poset is locally finite.  Intuitively, a finitary poset is
``finite downward'' but may be ``infinite upward''.  Intuitively, a poset with
finite covers is ``finite sideways''.

\begin{proposition}
  If a poset $P$ has finite covers, the following are equivalent:
  \begin{enumerate}
    \item $P$ is finitary.
    \item $P$ has no infinite descending chain (there does not exist
	$p_1, p_2, p_3, ... \in P$ such that
	$p_1 > p_2 > p_3 > ...$), and $P$ has no
      bounded infinite ascending chain (there does not exist $p_1,
      p_2, p_3, ..., q \in P$ such that
	$p_1 < p_2 < p_3 < ... < q$).
  \end{enumerate}
\end{proposition}
\begin{proof}

  (1) $\Rightarrow$ (2): Trivial.

  (2) $\Rightarrow$ (1): 
Suppose $P$ is not finitary, and so there is a $p \in P$ whose
principal ideal $I$ is infinite.
Consider the tree $T$ of saturated descending chains from $p$ in $P$.
For any $q \in I$, there is a finite saturated descending chain from
$p$ to $q$, because $P$ is locally finite.
Thus, for any $q \in I$, there is a distinct path in $T$, and
since $I$ is infinite, $T$ is infinite.
Applying the tree version of K\H{o}nig's infinity lemma to this tree,
either $T$ has a vertex of infinite degree or an infinite path.
Since $P$ has finite covers, $T$ has no vertex of infinite
degree.
Thus, $T$ has an infinite path, which is an infinite descending
chain from $p$.
\end{proof}

\paragraph{The poset of points}
For every instantiation of the general theory of tableau insertion algorithms, 
we start with a poset $B$ that is finitary and has finite covers.
The elements of $B$ are called \textit{points}, \textit{boxes}, or
sometimes \textit{squares}; the latter
terms because at a later stage boxes will be mapped to values, which
will be represented graphically by placing values in the boxes.
(We reserve the term ``cell'' for another definition later.)

For the ``archetype'' instantiation of the theory, which is the theory of
the Robinson-Schensted algorithm for unshifted tableaux \cite{Schen1961},
$B$ is the quadrant of the plane,
$\mbbQ = \mbbP^2$, with the order relation
\begin{equation}
(i,j) \le (k,l) \longequiv i \le k \textrm{ and } j \le l \textrm{\quad for all }
	(i,j), (k,l) \in \mbbQ
\end{equation}
We display the quadrant in ``English format'', as if the elements
were the indexes of the elements of a matrix, and sometimes with the boundaries
shown explicitly.  The locations of the points are displayed as in
figure~\ref{fig:quadrant-nos}.
\begin{figure}
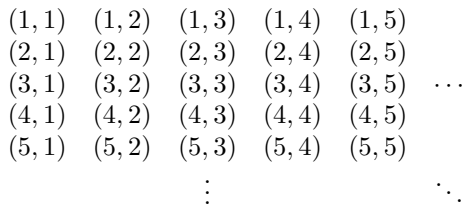

\[
\begin{array}{cccccc}
  (1,1) & (1,2) & (1,3) & (1,4) & (1,5) & \\
  (2,1) & (2,2) & (2,3) & (2,4) & (2,5) & \\
  (3,1) & (3,2) & (3,3) & (3,4) & (3,5) & \cdots \\
  (4,1) & (4,2) & (4,3) & (4,4) & (4,5) & \\
  (5,1) & (5,2) & (5,3) & (5,4) & (5,5) & \\
        &       & \vdots &      &       & \ddots
\end{array}
\]
\caption{The points of the quadrant $\mbbQ = \mbbP^2$}
\label{fig:quadrant-nos}
\end{figure}
Usually, we display the points without giving their coordinates, as in
figure~\ref{fig:quadrant-plain}.
\begin{figure}
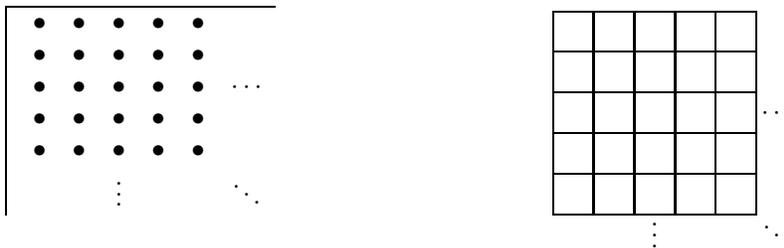

\hbox to \linewidth{\hfill
\raisebox{-2.5em}{$\begin{array}{@{}lcccccc}
\hline
\vline & \bullet & \bullet & \bullet & \bullet & \bullet & \\
\vline & \bullet & \bullet & \bullet & \bullet & \bullet & \\
\vline & \bullet & \bullet & \bullet & \bullet & \bullet & \cdots \\
\vline & \bullet & \bullet & \bullet & \bullet & \bullet & \\
\vline & \bullet & \bullet & \bullet & \bullet & \bullet & \\
\vline &         &         & \vdots  &         &         & \ddots
\end{array}$}
\hfill
\begin{ytableau}
\relax & & & & & \none \\
& & & & & \none \\
& & & & & \none[\cdots] \\
& & & & & \none \\
& & & & & \none \\
\noalign{\vskip 0.2em}
\none & \none & \none[\vdots] & \none & \none & \none[\ddots]
\end{ytableau}
\hfill}
\caption{The points of the quadrant $\mbbQ$ without coordinates}
\label{fig:quadrant-plain}
\end{figure}

\paragraph{The distributive lattice of diagrams}
A finite order ideal (lower set) of $B$ is called a \textit{diagram}
or \textit{shape}.

In the
archetype instantiation, they are also called \textit{Young diagrams} or
\textit{Ferrers diagrams}.  Examples are shown in figure~\ref{fig:young-diagrams}.
\begin{figure}
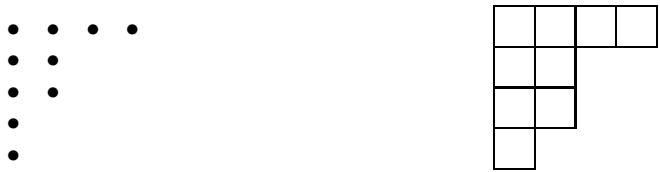

\hbox to \linewidth{\hfill
\raisebox{-2em}{$\begin{array}{cccccc}
\bullet & \bullet & \bullet & \bullet \\
\bullet & \bullet \\
\bullet & \bullet \\
\bullet \\
\bullet
\end{array}$}
\hfill
$\ydiagram{4,2,2,1}$
\hfill}
\caption{Young diagrams}
\label{fig:young-diagrams}
\end{figure}

The set of all finite order ideals of $B$, denoted $J_f(B)$
\cite{Stan2012a}*{sec.~3.4}\cite{Fom1994a}*{Def.~2.1.1}, is a
distributive lattice, and is the set of all diagrams.
$J_f(B)$ is finitary and because it is a lattice, it has
a minimum element, the empty order ideal of $B$
(which we denote by $\hatzero$ or $\emptyset$).
Because $B$ has finite covers, $J_f(B)$ has finite covers.
We set $V = J_f(B)$
since that will be the set of vertices of graphs that we will define
later.
$V$ is naturally graded by $\rho(x) = \lvert x \rvert$.

In the archetype instantiation, $V$ is the lattice of partitions of
integers, called \textit{Young's lattice} $\mbbY$, shown in
figure~\ref{fig:young-lattice}.
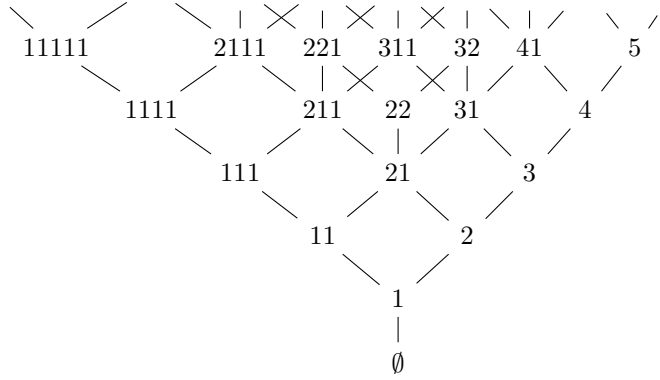
\begin{figure}
\begin{tikzcd}[column sep=0.5em,row sep=1em,every arrow/.append style={dash}]
\null \ar[dr, shorten <=2mm] && \null \ar[dl, shorten <=2mm] \ar[dr, shorten <=2mm] & \null \ar[d, shorten <=1mm] \ar[dr, shorten <=2mm]  & \null \ar[dl, shorten <=2mm] \ar[d, shorten <=1mm] \ar[dr, shorten <=2mm] & \null \ar[dl, shorten <=2mm] \ar[d, shorten <=1mm] \ar[dr, shorten <=2mm] & \null \ar[dl, shorten <=2mm] \ar[d, shorten <=1mm] \ar[dr, shorten <=2mm] & \null \ar[d, shorten <=1mm] & \null \ar[dl, shorten <=2mm] \ar[dr, shorten <=2mm] && \null \ar[dl, shorten <=2mm] \\
& 11111 \ar[dr] && 2111 \ar[dl] \ar[dr] & 221 \ar[d] \ar[dr] & 311 \ar[dl] \ar[dr] & 32 \ar[dl] \ar[d] & 41 \ar[dl] \ar[dr] && 5 \ar[dl] \\
&& 1111 \ar[dr] && 211 \ar[dl] \ar[dr] & 22 \ar[d] & 31 \ar[dl] \ar[dr] && 4 \ar[dl] \\
&&& 111 \ar[dr] && 21 \ar[dl] \ar[dr] && 3 \ar[dl] \\
&&&& 11 \ar[dr] && 2 \ar[dl] \\
&&&&& 1 \ar[d] \\
&&&&& \emptyset
\end{tikzcd}
\caption{Young's lattice $\mbbY$ of partitions}
\label{fig:young-lattice}
\end{figure}

\paragraph{The weight function}
\label{sec:weight}
We assume a \textit{weight} function $w$ from $B$ to
$\mbbP$
and a \textit{differential degree} $r \in \mbbP$.\footnote{The operations
applied to values of $w$ (and $r$) are very
limited and likely the target of $w$ could be extended from
$\mbbP$ to any
algebraic structure containing a homomorphic image of $\mbbP$,
with $r$ also being in the larger structure.  In our combinatorial
interpretation of $w$, this would cause additional complexities, but in
the algebraic approach of e.g.~\cite{Stan1988a}, it should be
straightforward.  See e.g.~the signed weights of \cite{Lam2008} and
Haiman's complex-valued weights described in \cite{Stan1990a}*{Prop.~3.3}.}

In order for our constructions to support the conventional counting
arguments of paths in the ascending and descending graphs, the weight
function must support the algebraic machinery of $U$ and $D$ operators
in \cite{Stan1988a}*{Def.~1.1} and \cite{Fom1994a}*{sec.~2.2}.
This requires
\begin{equation}\label{eq:weight}
\sum_{\substack{y \in V \\ y \lessdot x}} w(x-y) + r =
\sum_{\substack{z \in V \\ x \lessdot z}} w(z-x), \textrm{\quad for all } x \in V
\end{equation}
This is a condition relating the weights of the
elements of $B$ that are the differences between $x$ and the elements
of $V$ it
covers to the weights of the elements of $B$ that are the differences between
$x$ and the elements of $V$ that cover it.
That can be restated as a relationship between the weights of the
elements of $B$ that can be added to $x$ to make a larger order ideal
and the weights of the elements of $B$ that can be removed from $x$ to
make a smaller order ideal.
Because this condition relates several values of $w$ for each order
ideal of $B$, it is very restrictive and the space of allowed weight
functions for a given $B$ is usually small.

In the archetype instantiation, $w(p) = 1$ for all $p$ and $r = 1$, which
are the trivial values for $w$ and $r$.
The Hasse diagram of the quadrant and the Hasse diagram with the points labeled
with their weights are shown in figure~\ref{fig:hasse-quadrant}.
\begin{figure}
\hbox to \linewidth{\hfill
\begin{tikzcd}[column sep=0.5em,row sep=1em,every arrow/.append style={dash}]
\null \ar[dr, shorten <=2mm] && \null \ar[dl, shorten <=2mm] \ar[dr, shorten <=2mm] && \null \ar[dl, shorten <=2mm] \ar[dr, shorten <=2mm] && 
    \null \ar[dl, shorten <=2mm] \ar[dr, shorten <=2mm] &&
    \null \ar[dl, shorten <=2mm] \\
&\bullet \ar[dr] && \bullet \ar[dl] \ar[dr] && \bullet \ar[dl] \ar[dr] && 
    \bullet \ar[dl] \\
&& \bullet \ar[dr] && \bullet \ar[dl] \ar[dr] && \bullet \ar[dl] \\
&&& \bullet \ar[dr] && \bullet \ar[dl] \\
&&&& \bullet
\end{tikzcd}
\hfill
\begin{tikzcd}[column sep=0.5em,row sep=1em,every arrow/.append style={dash}]
\null \ar[dr, shorten <=2mm] && \null \ar[dl, shorten <=2mm] \ar[dr, shorten <=2mm] && \null \ar[dl, shorten <=2mm] \ar[dr, shorten <=2mm] && 
    \null \ar[dl, shorten <=2mm] \ar[dr, shorten <=2mm] &&
    \null \ar[dl, shorten <=2mm] \\
&\bullet_1 \ar[dr] && \bullet_1 \ar[dl] \ar[dr] && \bullet_1 \ar[dl] \ar[dr] && 
    \bullet_1 \ar[dl] \\
&& \bullet_1 \ar[dr] && \bullet_1 \ar[dl] \ar[dr] && \bullet_1 \ar[dl] \\
&&& \bullet_1 \ar[dr] && \bullet_1 \ar[dl] \\
&&&& \bullet_1
\end{tikzcd}
\hfill}
\caption[Hasse diagram of the quadrant $\mbbQ$]%
{(a) Hasse diagram of the quadrant $\mbbQ$, (b) with weights $w(p) = 1$}
\label{fig:hasse-quadrant}
\end{figure}
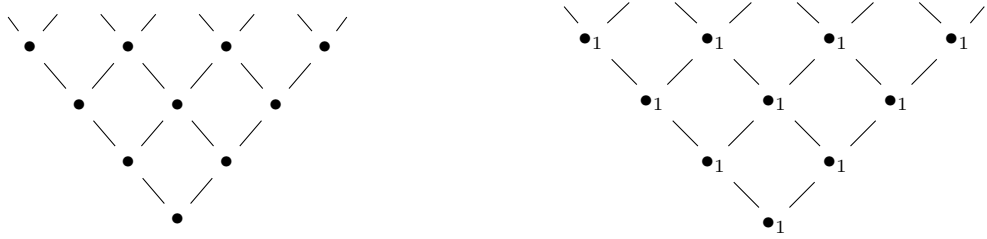

\paragraph{The ascending and descending graphs}
We define a directed graph $G_1$ on the vertex set $V$ by defining
edges for all ascending covering relationships:
\begin{equation}
  x \overset{G_1}{\longrightarrow} y \longequiv x \lessdot y \textrm{\quad for all } x, y \in V
\end{equation}
We define a second directed graph $G_2$ on the vertex set $V$ by
defining edges for all \textit{descending} covering relationships:
\begin{equation}
x \overset{G_2}{\longleftarrow} y \longequiv x \lessdot y \textrm{\quad for all } x, y \in V
\end{equation}
We conceptualize $m = w(y-x)$  as the \textit{multiplicity} of
edge $x \overset{G_2}{\longleftarrow} y$,
labeling each of the multiple edges with a ``color'' taken from
a set of $m$ elements, often $[m]$.
The color of an edge $e$ is denoted $\color(e)$.

Note that (at this stage of the exposition) edges of $G_1$ are not
multiple and do not have colors.

We abuse the notation by using $G_1$ and $G_2$ to denote both the
graphs as structures and their sets of edges.

The structure $\mathcal{G} = (V, \rho, G_1, G_2)$ forms a
\textit{weighted dual graded graph}
in the terminology of \cite{Fom1994a}, which we sometimes abbreviate to
\textit{w.d.g.g.}
If $r = 1$, the w.d.g.g.\ is a differential poset \cite{Stan1988a}.
The
particular structure and direction of the edges in $G_1$ and $G_2$
allow detailed enumeration of directed paths in the graph $G_1 \uplus G_2$
(on the vertex set $V$)
under various constraints.  These constraints can be crafted so that
the enumerations provide proof of various enumerative identities,
often via suitable bijections.  See e.g.~\cite{Fom1994a}*{sec.~1.5 and 2.9}.

In the archetype instantiation, $G_1$ and $G_2$ are the graph of covering
relations in the lattice of partitions $\mbbY$, with $G_1$ having the edges
directed toward the larger partition and $G_2$ having the edges
directed toward the smaller partition.
The edges of $G_2$ have multiplicity 1 and
so are just the reverses of the edges of $G_1$.
The two graphs can be drawn as in
figure~\ref{fig:young-graph-weighted}, with each edge being labeled
with its multiplicity in $G_2$.
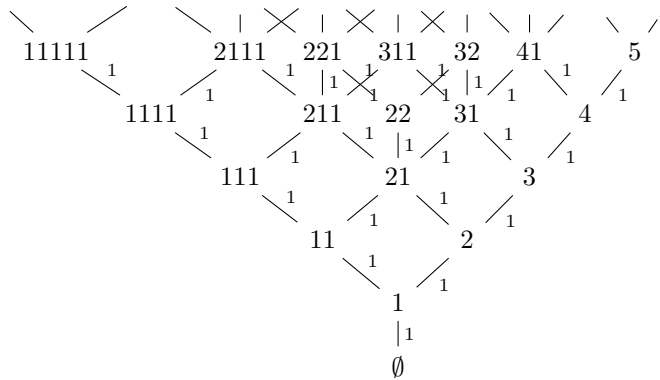
\begin{figure}
\begin{tikzcd}[column sep=0.5em,row sep=1em,every arrow/.append style={dash}]
\null \ar[dr, shorten <=2mm] && \null \ar[dl, shorten <=2mm] \ar[dr, shorten <=2mm] & \null \ar[d, shorten <=1mm] \ar[dr, shorten <=2mm]  & \null \ar[dl, shorten <=2mm] \ar[d, shorten <=1mm] \ar[dr, shorten <=2mm] & \null \ar[dl, shorten <=2mm] \ar[d, shorten <=1mm] \ar[dr, shorten <=2mm] & \null \ar[dl, shorten <=2mm] \ar[d, shorten <=1mm] \ar[dr, shorten <=2mm] & \null \ar[d, shorten <=1mm] & \null \ar[dl, shorten <=2mm] \ar[dr, shorten <=2mm] && \null \ar[dl, shorten <=2mm] \\
& 11111 \ar[dr, "1"] && 2111 \ar[dl, "1"] \ar[dr, "1"] & 221 \ar[d, "1"] \ar[dr, "1"] & 311 \ar[dl, "1"] \ar[dr, "1"] & 32 \ar[dl, "1"] \ar[d, "1"] & 41 \ar[dl, "1"] \ar[dr, "1"] && 5 \ar[dl, "1"] \\
&& 1111 \ar[dr, "1"] && 211 \ar[dl, "1"] \ar[dr, "1"] & 22 \ar[d, "1"] & 31 \ar[dl, "1"] \ar[dr, "1"] && 4 \ar[dl, "1"] \\
&&& 111 \ar[dr, "1"] && 21 \ar[dl, "1"] \ar[dr, "1"] && 3 \ar[dl, "1"] \\
&&&& 11 \ar[dr, "1"] && 2 \ar[dl, "1"] \\
&&&&& 1 \ar[d, "1"] \\
&&&&& \emptyset
\end{tikzcd}
\caption{The Young graph $\mbbY$ with edges labeled with their
  weights/multiplicities}
\label{fig:young-graph-weighted}
\end{figure}

\paragraph{Tableaux}
A mapping of the boxes of a diagram to another set, especially
$\mbbP$, is called a \textit{tableau} (with plural \textit{tableaux}).
If the mapping is a bijection into a set $[n]$, the
tableau is called \textit{standard}.

If a tableau's mapping is into $\mbbP \times \mbbP$ and the
second component of the value of any point $p$ in the diagram
is~$\in [w(p)]$, the tableau is called \textit{colored},
with the first component considered the ``element'' at $p$ and
the second component considered the ``color'' of the element.
If the first component of the values form a bijection into a set
$[n]$, the tableau is called a \textit{standard colored} tableau.

In the archetype instantiation,
tableaux are also called \textit{Young tableaux}, examples of which are shown in
figure~\ref{fig:young-tableaux}.
\begin{figure}
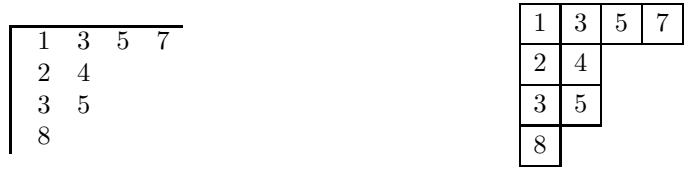

\vbox to \parskip{}
\hbox to \linewidth{\hfill
\raisebox{-2em}{$\begin{array}{@{}lcccccc}
\hline
\vline & 1 & 3 & 5 & 7 \\
\vline & 2 & 4 \\
\vline & 3 & 5 \\
\vline & 8
\end{array}$}
\hfill
\begin{ytableau}
1 & 3 & 5 & 7 \\
2 & 4 \\
3 & 5 \\
8
\end{ytableau}
\hfill}
\caption{Young tableaux}
\label{fig:young-tableaux}
\end{figure}

\paragraph{Growth diagrams}
\label{sec:growth-diagrams}
The next structure we define is a \textit{growth diagram}
or \textit{2-growth} \cite{Fom1995a}*{sec.~3.1}.  A skeleton growth
diagram is shown in figure~\ref{fig:growth}.
\begin{figure}
\begin{tikzcd}[sep=small]
n_{0,4} \ar[rr,"g^1_{1,4}"]\ar[dd,"g^2_{0,4}"] & &
  n_{1,4} \ar[rr,"g^1_{2,4}"]\ar[dd,"g^2_{1,4}"] & &
  n_{2,4} \ar[rr,"g^1_{3,4}"]\ar[dd,"g^2_{2,4}"] & &
  n_{3,4} \ar[rr,"g^1_{4,4}"]\ar[dd,"g^2_{3,4}"] & &
  n_{4,4} \ar[dd,"g^2_{4,4}"] \\
& \alpha_{1,4} & & \alpha_{2,4} & & \alpha_{3,4} & & \alpha_{4,4} \\
n_{0,3} \ar[rr,"g^1_{1,3}"]\ar[dd,"g^2_{0,3}"] & &
  n_{1,3} \ar[rr,"g^1_{2,3}"]\ar[dd,"g^2_{1,3}"] & &
  n_{2,3} \ar[rr,"g^1_{3,3}"]\ar[dd,"g^2_{2,3}"] & &
  n_{3,3} \ar[rr,"g^1_{4,3}"]\ar[dd,"g^2_{3,3}"] & &
  n_{4,3} \ar[dd,"g^2_{4,3}"] \\
& \alpha_{1,3} & & \alpha_{2,3} & & \alpha_{3,3} & & \alpha_{4,3} \\
n_{0,2} \ar[rr,"g^1_{1,2}"]\ar[dd,"g^2_{0,2}"] & &
  n_{1,2} \ar[rr,"g^1_{2,2}"]\ar[dd,"g^2_{1,2}"] & &
  n_{2,2} \ar[rr,"g^1_{3,2}"]\ar[dd,"g^2_{2,2}"] & &
  n_{3,2} \ar[rr,"g^1_{4,2}"]\ar[dd,"g^2_{3,2}"] & &
  n_{4,2} \ar[dd,"g^2_{4,2}"] \\
& \alpha_{1,2} & & \alpha_{2,2} & & \alpha_{3,2} & & \alpha_{4,2} \\
n_{0,1} \ar[rr,"g^1_{1,1}"]\ar[dd,"g^2_{0,1}"] & &
  n_{1,1} \ar[rr,"g^1_{2,1}"]\ar[dd,"g^2_{1,1}"] & &
  n_{2,1} \ar[rr,"g^1_{3,1}"]\ar[dd,"g^2_{2,1}"] & &
  n_{3,1} \ar[rr,"g^1_{4,1}"]\ar[dd,"g^2_{3,1}"] & &
  n_{4,1} \ar[dd,"g^2_{4,1}"] \\
& \alpha_{1,1} & & \alpha_{2,1} & & \alpha_{3,1} & & \alpha_{4,1} \\
n_{0,0} \ar[rr,"g^1_{1,0}"] & &
  n_{1,0} \ar[rr,"g^1_{2,0}"] & &
  n_{2,0} \ar[rr,"g^1_{3,0}"] & &
  n_{3,0} \ar[rr,"g^1_{4,0}"] & &
  n_{4,0}
\end{tikzcd}
\caption{A skeleton growth diagram}
\label{fig:growth}
\end{figure}
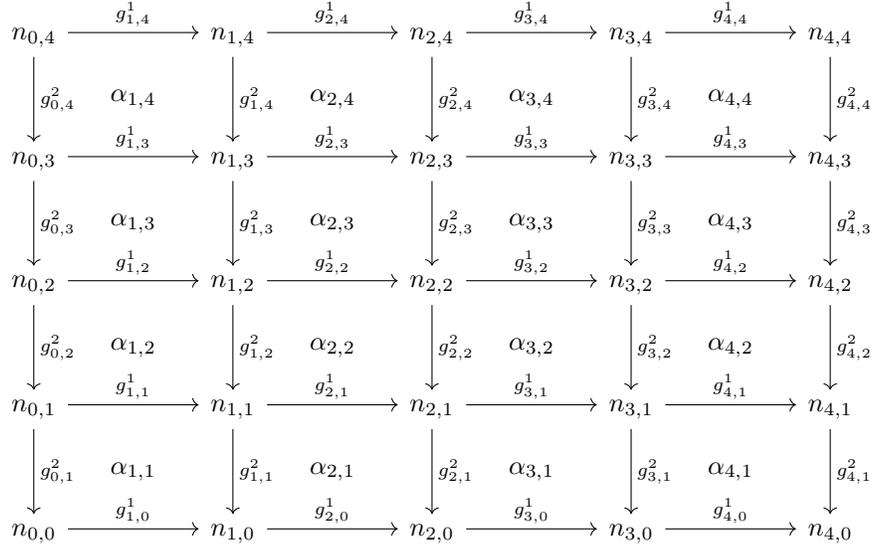

Growth diagrams are displayed using the subscripts as Cartesian
coordinates, so they are shown in ``French format''.
Note that the $g^1_{i,j}$ and $g^2_{i,j}$ have subscripts that match
the higher-numbered node that they connect (to the east or north,
resp.), and the $\alpha_{i,j}$
have subscripts that match the highest-numbered node of the four
surrounding them, the one to the northeast.
Each group of four nodes together with the four edges between them and
the central $\alpha_{i,j}$ is called a ``cell''.  A cell extracted
from its growth diagram is shown in figure~\ref{fig:cell}.
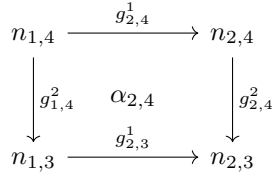
\begin{figure}
\begin{tikzcd}[sep=small]
  n_{1,4} \ar[rr,"g^1_{2,4}"]\ar[dd,"g^2_{1,4}"] & &
  n_{2,4} \ar[dd,"g^2_{2,4}"] \\
& \alpha_{2,4} \\
  n_{1,3} \ar[rr,"g^1_{2,3}"] & &
  n_{2,3}
\end{tikzcd}
\caption[A cell extracted from the growth diagram]%
{A cell extracted from the growth diagram
in figure~\protect\ref{fig:growth}}
\label{fig:cell}
\end{figure}

For a growth diagram of size $n \times m$:
\begin{enumerate}
\item The diagram has a set of \textit{nodes},
  $\{(i,j): i,j \in \mbbN, 0 \le i \le n, 0 \le j \le m\}$.
\item Each node is mapped to $n_{i,j}$, an element of $V$.  In a
  display, the value $n_{i,j}$ will usually be written at the position of the node.
\item Each node with $i = 0$ or $j = 0$ is mapped to $\hatzero \in V$.
\item Each horizontal pair of nodes $(i-1,j)$ and $(i,j)$ is connected
  by an edge $g^1_{i,j}$.  If $n_{i-1,j} = n_{i,j}$, then $g^1_{i,j}$
  is \textit{degenerate}, contains no further information, and is
  displayed as an unlabeled arrow.
  Otherwise, $g^1_{i,j}$ is mapped to an edge in
  $G_1$ from the value $n_{i-1,j}$ to the value $n_{i,j}$.  Since the
  edge is implied by the values $n_{i-1,j}$ and $n_{i,j}$, it is
  usually displayed as an unlabeled arrow.
\item Each vertical pair of nodes $(i,j-1)$ and $(i,j)$ is connected
  by an edge $g^2_{i,j}$.  If $n_{i,j-1} = n_{i,j}$, then $g^2_{i,j}$
  is \textit{degenerate}, contains no further information, and is usually
  displayed as an unlabeled arrow.  Otherwise,
  $g^2_{i,j}$ is mapped to an edge in
  $G_2$ from the value $n_{i,j}$ to the value $n_{i,j-1}$.  Since the
  edge is constrained to be one of a set of multiple edges from 
  $n_{i,j}$ to $n_{i,j-1}$, it is usually displayed as an arrow
  labeled by the color of the edge.  When there is only one color for
  an edge between
  that pair of elements of $V$, the color will often be omitted.
\item Each cell is mapped to a value $\alpha_{i,j}$ in the set
  $\mathcal{A} = \{0, 1, \ldots, r\} = \{0\} \uplus [r]$.
  (This $r$ is the differential degree,
  the $r$ in the constraint equation~\ref{eq:weight} above.)
  Usually, the value 0 is considered ``null'', whereas the values
  in $\mathcal{A}$ that are $> 0$ are considered ``colors''.
  (Note this is a different definition of ``color'' than is applied to the
  edges of $G_2$.)
  When
  $\alpha_{i,j} = 0$, the center of the cell is usually left blank.
  In the common case $r = 1$, the (sole) color 1 is shown as
  ``$X$''.
\end{enumerate}

Note that the arrows for non-degenerate $g^1_{i,j}$ and $g^2_{i,j}$ are shown
pointing between the node values in the direction that the edge and
nodes are related in $G_1$ or $G_2$ (resp.).

In the archetype instantiation, the $G_2$ arrows are unlabeled, and
$r = 1$, so all nonzero $\alpha_{i,j}$ are shown as
``$X$''.

\paragraph{Generalized permutations}
A \textit{generalized permutation} \cite{Roby1991a}*{Def.~2.3.6}
or \textit{diagonal set} \cite{Fom1994a}*{Def.~2.6.1} is an
assignment of values to the $\alpha_{i,j}$ of a growth diagram such
that no two non-zero elements have the same first subscript or the
same second subscript.
(Remember the values $\alpha_{i,j}$ are~$\in \mathcal{A} = \{0\} \uplus [r]$.)
In all growth diagrams we consider, the $\alpha_{i,j}$ form a
generalized permutation.

A generalized permutation can be represented compactly by listing for
each $j \ge 1$ the value $i$ (if any) for which $\alpha_{i,j} \ne 0$
and (if $r > 1$) the value of that $\alpha_{i,j}$.
For example,
the generalized permutation that assigns 1
to $\alpha_{1,1}$, $\alpha_{2,3}$, and $\alpha_{3,2}$; assigns 2
to  $\alpha_{4,5}$;
and 0 to all other $\alpha_{i,j}$ can be represented
$( 1^1, 3^1, 2^1, \hskip 1ex, 4^2 )$.
In terms of an insertion algorithm, this is the insertion of 1, 3, and
then 2 during steps 1, 2, and 3, inserting nothing during step 4, then
inserting 4 during step 5.  In addition, the first three insertions
have color 1 and the last insertion has color 2.

In the archetype instantiation, if the dimensions of the growth diagram are
equal ($n = m$), a generalized permutation containing $n$ non-zero
$\alpha_{i,j}$ is an $n \times n$ permutation matrix.
The compact representation of the generalized permutation can omit the
colors because they are all 1, resulting in the typical one-line
representation of a permutation of $[n]$.

\paragraph{R-correspondence}
Different authors use similar but not identical terminology for the
growth process.  We take our terminology from \cite{Roby1991a} but with
some differences.
In this discussion, we mostly consider relations between components of
a single cell, and in this context,
the conventional labels of the components of a single cell are
shown in figure~\ref{fig:cell-convention} \cite{Fom1995a}*{Lem.~3.7.12 and 4.2.1}.
\begin{figure}
\hbox to \linewidth{\hfill
\begin{tikzcd}[sep=small]
  y \ar[rr,"b_1"]\ar[dd,"a_2"] & & z \ar[dd,"b_2"] \\
& \alpha \\
  t \ar[rr,"a_1"] & & x
\end{tikzcd}
\hfill
$\vcenter{
  \hbox{$\alpha \textit{ is $0$ or} \in [r]$}
  \hbox{$a_1 \textit{ is degenerate or} = (t, x) \in G_1$}
  \hbox{$a_2 \textit{ is degenerate or} = (y, t, c_a) \in G_2$}
  \hbox{$b_1 \textit{ is degenerate or} = (y, z) \in G_1$}
  \hbox{$b_2 \textit{ is degenerate or} = (z, x, c_b) \in G_2$}
}$
\hfill}
\caption{Conventional labels of the components of a cell}
\label{fig:cell-convention}
\end{figure}
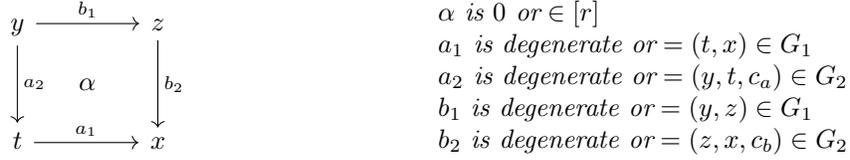

We assume an \textit{R-correspondence} $\psi$ that is an enumerative
expression of the constraint equation~\ref{eq:weight}:
For any $x \in V$, $\psi_x$ is a bijection from
\[ \{(x, t, c) \in G_2: t \lessdot x \textrm{ and } c \in [w(x-t)] \}
\uplus [r] \] to
\[ \{(z, x, c) \in G_2: z \gtrdot x \textrm{ and } c \in [w(z-x)] \}. \]
That is, $\psi_x$ maps the (multiple) edges in $G_2$ from $x$ downward
to some $t$
to a subset of the (multiple) edges in $G_2$ downward from some $z$ to $x$,
and maps the colors used in generalized
permutations ($[r]$) to the remaining edges in $G_2$ downward to $x$.

Given the R-correspondence $\psi$, we define a
function $\Psi$ that, given the west and south edge values of a cell
$a_1$ and $a_2$
and the $\alpha$ value of the cell, determines the
cell's north and east edge values $b_1$ and $b_2$
of the cell.\cite{Roby1991a}*{sec.~2.6}\cite{Fom1995a}*{sec.~3.4.4 and
  Lem.~4.2.1}
The function $\Psi$ clearly can be equivalently defined as a mapping
from the cell's southwest, northwest, and southeast node
values $x$, $y$, and $t$ and west color value $c_a$ to
the northeast node value $z$ and east color value $c_b$:
\begin{equation}\label{eq:Psi}
(z, c_b) = \Psi(t, x, y, c_a, \alpha) = \left\{
\begin{array}{lll}
(t, \nnull)      & \text{if $x = y = t$ and $\alpha = 0$} \\
\psi_x(\alpha)   & \text{if $x = y = t$ and $\alpha \ne 0$} & \text{(case ``$X$'')} \\
(x, c_a)         & \text{if $x \ne t$ and $y = t$} \\
(y, \nnull)      & \text{if $y \ne t$ and $x = t$} \\
\psi_x(x, t, c_a) & \text{if $x = y \ne t$} & \text{(case ``$\ast$'')} \\
(x \vee y, c_a)   & \text{if $x \ne y$, $x \ne t$, and $y \ne t$} & \text{(case ``$\vee$'')}
\end{array}
\right.
\end{equation}
Given the constraints which we wish growth diagrams to obey, all of
the variation allowed in $\Psi$ is contained within $\psi$.

In the archetype instantiation, there is only one color (viz.\ 1), and we
define $\psi$ as
\begin{equation}
\left\{
\begin{array}{ll}
\psi_x(1) & = x \textit{ with a box added to the
  1\textsuperscript{st} row} \\
\psi_x(x, t, 1) & = (x \textit{ with a box added to the
$k+1$\textsuperscript{st} row}, x, 1) \\
& \qquad\qquad\text{when $x$ is $t$ with a box added to the
  $k$\textsuperscript{th} row}
\end{array}
\right.
\end{equation}

\paragraph{The growth process}
Our primary object of study is the \textit{growth process}, which
operates on growth diagrams:
Given the function $\Psi$, the initial node values $\hatzero$ on the
west and south
sides, and the $\alpha_{i,j}$, the remaining values $g^1_{i,j}$, $g^2_{i,j}$,
and $n_{i,j}$ can be computed in
a wave from the southwest to the northeast of the growth diagram.

The values $g^1_{1,m}, g^1_{2,m}, \ldots, g^1_{n,m}$ (the north edge of the
diagram) form a chain of edges in
$G_1$ from $\hatzero$ to the vertex $n_{n,m}$, which is a diagram.
We can form a standard tableau which encodes this chain by
mapping
each point $n_{i,m} - n_{i-1,m}$ in $n_{n,m}$ (the point added by
$g^1_{i,m}$) to the value $i$.  This standard tableau is called the
\textit{$P$ tableau}.

The values $g^2_{n,m}, g^2_{n,m-1}, \ldots, g^2_{n,1}$ (the east edge of
the diagram) form a chain of edges in
$G_2$ from $n_{n,m}$ to $\hatzero$.
We can form a standard colored tableau which encodes this chain
by mapping
each point $n_{n,j} - n_{n,j-1}$ to the value
$(j, \color(g^2_{n,j}))$.  This standard colored tableau is called the
\textit{$Q$ tableau}.

Intuitively, the northward direction on the growth diagram can be seen
as time moving forward, successive insertions done by the insertion
algorithm.
(Hence the term ``growth''.)
The value inserted at time step $j$ (if any)
is the $i$ for which $\alpha_{i,j} \ne 0$, and its color is that
non-zero value.
The shape of the tableaux after the insertion at time $j$ is
$n_{n,j}$ at the east end of the row.

Intuitively,
the westward direction on the growth diagram is ``restriction of
magnitude'':  For all tableau insertion algorithms, the movements and
positioning of a value $i$ in $P$ is not affected by the insertion of
any value
$j > i$, either before or after the insertion of the $i$.  Truncating
the growth diagram at column $i_{max}$ shows the shape of $P$
as if none of the values $> i_{max}$ were inserted, or equivalently, if
all boxes with values $> i_{max}$ in $P$ are deleted.

In the archetype instantiation, the generalized permutation $(2, 3, 4, 1)$
that assigns 1
to $\alpha_{2,1}$, $\alpha_{3,2}$, $\alpha_{4,3}$, and
$\alpha_{1,4}$,\footnote{Remember that in a growth diagram, the
coordinates are assigned in the Cartesian manner.}
and 0 to all other $\alpha_{i,j}$ generates the growth process shown
in figure~\ref{fig:growth-unshifted}.
\begin{figure}
\begin{tikzcd}[ampersand replacement=\&,sep=small]
\emptyset \ar[rr]\ar[dd] \&\&
  1 \ar[rr]\ar[dd] \&\&
  \ast 11 \ar[rr]\ar[dd] \&\&
  \vee 21 \ar[rr]\ar[dd] \&\&
  \vee 31 \ar[dd] \\
\& X \&\&  \&\&  \&\&  \& \\
\emptyset \ar[rr]\ar[dd] \&\&
  \emptyset \ar[rr]\ar[dd] \&\&
  1 \ar[rr]\ar[dd] \&\&
  2 \ar[rr]\ar[dd] \&\&
  3 \ar[dd] \\
\&  \&\&  \&\&  \&\& X \& \\
\emptyset \ar[rr]\ar[dd] \&\&
  \emptyset \ar[rr]\ar[dd] \&\&
  1 \ar[rr]\ar[dd] \&\&
  2 \ar[rr]\ar[dd] \&\&
  2 \ar[dd] \\
\&  \&\&  \&\& X \&\&  \& \\
\emptyset \ar[rr]\ar[dd] \&\&
  \emptyset \ar[rr]\ar[dd] \&\&
  1 \ar[rr]\ar[dd] \&\&
  1 \ar[rr]\ar[dd] \&\&
  1 \ar[dd] \\
\&  \&\& X \&\&  \&\&  \& \\
\emptyset \ar[rr] \&\&
  \emptyset \ar[rr] \&\&
  \emptyset \ar[rr] \&\&
  \emptyset \ar[rr] \&\&
  \emptyset 
\end{tikzcd}
\caption[Example growth process for the permutation $(2, 3, 4, 1)$]%
{Example growth process for the permutation $(2, 3, 4, 1)$.\\
Each $n$ value is a Young diagram, which is a partition.  Each
partition is represented by listing its parts, with the empty
partition represented by $\emptyset$.\\
The cells where $z = \psi_x(\alpha)$ are marked by ``$X$''.\\
The nodes $z$ where $z = \psi_x(x, t, c_a)$ are marked by ``$\ast$''.\\
The nodes $z$ where $z = x \vee y$ are marked by ``$\vee$''.\\
These markings correspond with the marked cases of equation~\protect\ref{eq:Psi}.}
\label{fig:growth-unshifted}
\end{figure}
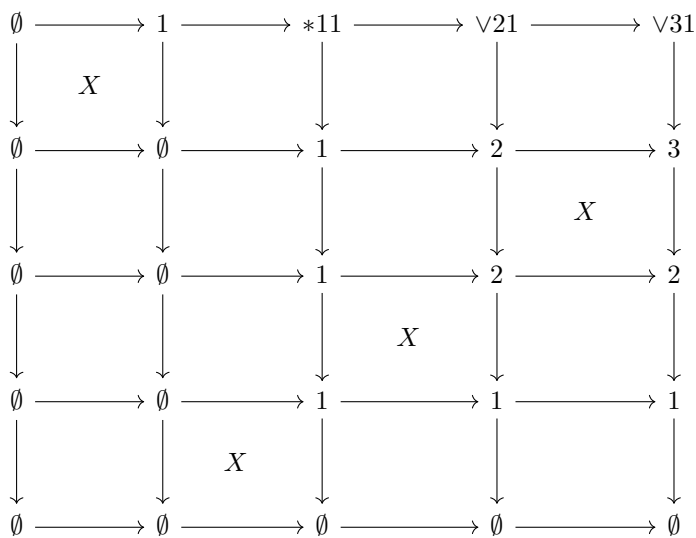
This growth generates the $P$ and $Q$ tableaux shown in figure~\ref{fig:PQ31}.
\begin{figure}
\hbox to \linewidth{\hfill
\begin{ytableau}
1 & 3 & 4 \\
2
\end{ytableau}
\hfill
\begin{ytableau}
1 & 2 & 3 \\
4
\end{ytableau}
\hfill}
\caption[$P$ and $Q$ tableaux generated by the example growth process]%
{$P$ and $Q$ tableaux generated by the example growth process
  of figure~\protect\ref{fig:growth-unshifted}}
\label{fig:PQ31}
\end{figure}

\paragraph{Relationship of R-correspondences to insertion algorithms}
The three marked cases of equation~\ref{eq:Psi} implement the features of
an insertion algorithm \cite{Fom1995a}*{sec.~4.2 and 4.4}:
\begin{enumerate}
\item[] case ``$X$'' -- This case shows where in the diagram a value will be
  inserted, if the value is larger than all values already in the tableau.
  (Remember that generalized permutations contain no duplicated values.)
  $t = x = y$ is the tableau before insertion, and $z$ is the tableau
  after the insertion.
  Additionally, the $b_2$ edge shows \textit{how} the value is
  inserted, for algorithms in which insertions can be done in more
  than one way.  This information is propagated to cells to the east,
  that is, to further steps in the overall insertion of that one
  element, and ultimately becomes the color of the corresponding
  element of the $Q$ tableau.
  Thus, the $b_2$ edge and edges eastward carry the state information
  during the insertion of one number.
\item[] case ``$\ast$'' -- This case shows what happens when a value is
  inserted but that insertion conflicts with the previous insertion of
  a larger value.  In general, the new, smaller value will be inserted in the same
  place it would have been inserted if the larger value was not
  present, and the larger value is displaced, \textit{bumped}, into another
  place.
  The $a_2$ edge shows where the smaller value is inserted, and
  the $b_2$ edge shows the new location of the previous, larger value.
\item[] case ``$\vee$'' -- This case is when a value is inserted and a
  larger value has been inserted previously, but the two insertions do
  not conflict.
\end{enumerate}

In the archetype instantiation, consider the example growth given in
figure~\ref{fig:growth-unshifted}.  The generalized permutation is
$(2, 3, 4, 1)$.  Because $r=1$ there are is only one color in the
generalized permutation, and the $G_2$ edges have only one color,
which is not shown; i.e.~all insertions are ``done in the same way''.

The southern row corresponds to
the first insertion of the Robinson-Schensted insertion, when 2 is
inserted into an empty tableau.  The westernmost cell has all nodes
$\emptyset$, because when only elements $\le 1$ are considered, no
change happens.  The next cell east shows 1 being inserted into
$\emptyset$ to form the tableau 1.  The remaining cells to the east
show that successively allowing consideration of larger elements
causes no change, as at this time, no larger elements have been inserted.
The tableau after the insertion has the shape given as northeast node
of the last cell, 1.  The $P$ tableau at this stage in time, which is
1, is constructed by
the nodes at the northern side of the row.  The $Q$ tableau, which is
1, is constructed by the nodes from the southeast corner to the
northeast node of the last cell in the row.

A more complicated case is the northern row, where 1 is inserted into
the tableau \begin{ytableau} 2 & 3 & 4 \end{ytableau} of shape 3.
The westernmost cell in the row shows the
insertion of 1, with elements $> 1$ being ignored, giving the shape
1.  The next cell shows how the insertion of 1 interacts with the
previous insertion of 2; the 2 is bumped into the second row.
That this is a bump is shown by the
nodes of the cell having $x = y \gtrdot t$.  The R-correspondence sets
$z$ to 11, meaning that the bumped element is moved to the next row.
The third cell shows the interaction of the insertion of 1 with the
insertion of 3.  In this case, the 3 is unaffected.  This is shown by
the nodes of the cell having $x \ne y$.  The $z$ is constructed by $x
\vee y$.  The fourth and easternmost node of the row is similar to the
third.

A more complex example is
the first instantiation for shifted tableaux \ref{sec:sagan},
in which the $G_2$ edges have
non-trivial labels, shown as ``$B$'', ``$R$'', and ``$-$'' in
figure~\ref{fig:growth-sagan}.  Since the $G_2$ edges are north-south
edges between nodes, during a growth, the information in an edge
propagates eastward, that is it is state information during the
insertion process, generated during
the initial insertion of an element into the tableau, possibly
modified by bumping operations, and finally recorded in the $Q$
tableau.  But that state is not affected by the state of earlier (more southern)
insertions and does not affect the state of later (more northern) insertions.

\section{Duality}

\paragraph{Inversion duality}
Since the values $\alpha_{i,j}$ form a permutation matrix, it is well known
that the inverse of that permutation is
$\alpha^{-1}_{i,j} = \alpha_{j,i}$.
Geometrically, this is transposing the matrix of
$\alpha$ values around the southwest-northeast axis.
Thus, since $(4, 1, 2, 3)$ is the inverse of $(2, 3, 4, 1)$,
the $\alpha$ values in the growth process shown in
figure~\ref{fig:growth-unshifted-inverse} are the transpose of the
$\alpha$ values in the growth process shown in
figure~\ref{fig:growth-unshifted}: $\alpha^\prime_{i,j} = \alpha_{i,j}$.

If $w(p) = 1$ for all $p$, and it is in the archetype instantiation,
there are effectively no labels on the $G_2$ edges in the growth
diagram, and applying $\Psi$ to construct the growth process to the
inverse permutation constructs the same node values as in the growth
process for the original permutation, but in the transposed locations:
$n^\prime_{i,j} = n_{j,i}$.
\begin{figure}
\growthtableauxh{%
\begin{tikzcd}[ampersand replacement=\&,sep=small]
\emptyset \ar[rr]\ar[dd] \&\&
  1 \ar[rr]\ar[dd] \&\&
  2 \ar[rr]\ar[dd] \&\&
  3 \ar[rr]\ar[dd] \&\&
  31 \ar[dd] \\
\&  \&\&  \&\& X \&\&  \& \\
\emptyset \ar[rr]\ar[dd] \&\&
  1 \ar[rr]\ar[dd] \&\&
  2 \ar[rr]\ar[dd] \&\&
  2 \ar[rr]\ar[dd] \&\&
  21 \ar[dd] \\
\&  \&\& X \&\&  \&\&  \& \\
\emptyset \ar[rr]\ar[dd] \&\&
  1 \ar[rr]\ar[dd] \&\&
  1 \ar[rr]\ar[dd] \&\&
  1 \ar[rr]\ar[dd] \&\&
  11 \ar[dd] \\
\& X \&\&  \&\&  \&\&  \& \\
\emptyset \ar[rr]\ar[dd] \&\&
  \emptyset \ar[rr]\ar[dd] \&\&
  \emptyset \ar[rr]\ar[dd] \&\&
  \emptyset \ar[rr]\ar[dd] \&\&
  1 \ar[dd] \\
\&  \&\&  \&\&  \&\& X \& \\
\emptyset \ar[rr] \&\&
  \emptyset \ar[rr] \&\&
  \emptyset \ar[rr] \&\&
  \emptyset \ar[rr] \&\&
  \emptyset 
\end{tikzcd}
}{%
\begin{ytableau}
1 & 2 & 3 \\
4
\end{ytableau}
}{%
\begin{ytableau}
1 & 3 & 4 \\
2
\end{ytableau}
}
\caption[Example growth process for the permutation $(4, 1, 2, 3)$]%
{Example growth process for the permutation $(4, 1, 2, 3)$,
  which is the inverse of $(2, 3, 4, 1)$ in
  figure~\protect\ref{fig:growth-unshifted}:
  (a) growth diagram, (b) $P$ tableau, and (c) $Q$ tableau}
\label{fig:growth-unshifted-inverse}
\end{figure}
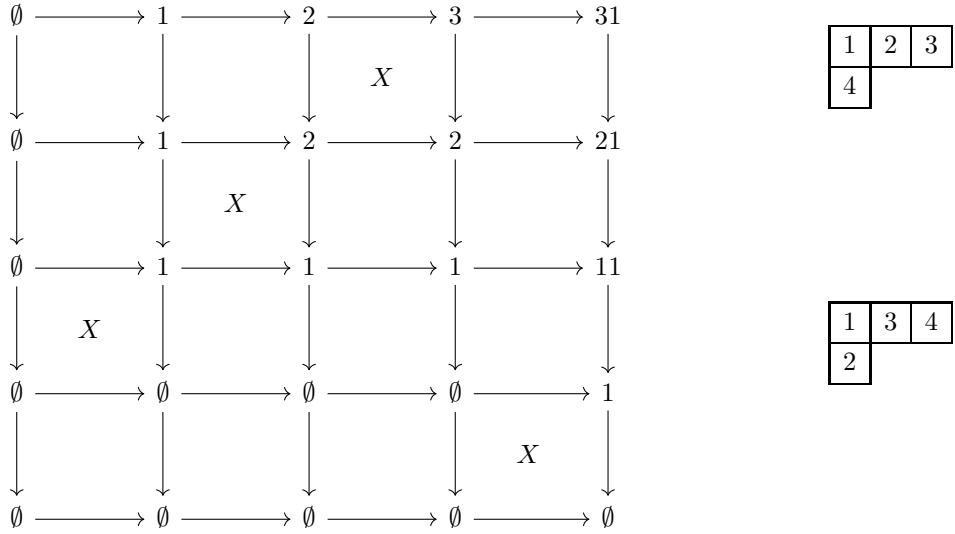

The result is that the $P$ and $Q$ tableaux for the inverse
permutation are the $Q$ and $P$ (resp.) tableaux for the original
permutation.

Note that inversion duality does not transpose the shapes that are the
$n_{i,j}$ values, nor the shapes of the $P$ or $Q$ tableaux.

\paragraph{Transpose duality}
A completely distinct (but confusingly similar) duality for
instantiations using unshifted tableaux is ``transpose duality'',
which is conceptually the application of the transpose automorphism to
$\mbbQ$ and $V$ and making the parallel changes to $\Psi$ and $\psi$:
\begin{equation}
\left\{
\begin{array}{lll}
\psi^T_x(\alpha) & = (z^T, x, c) & \textrm{\quad where } (z, x^T, c) = \psi_{x^T}(\alpha) \\
\psi^T_x(x, t, c) & = (z^T, x, c) & \textrm{\quad where } (z, x^T, c) = \psi_{x^T}(x^T, t^T, c)
\end{array}
\right.
\end{equation}
\begin{equation}
  \Psi^T(t, x, y, c_a, \alpha) = \Psi(t^T, x^T, y^T, c_a, \alpha)
\end{equation}

No insertion algorithm is self-dual under transpose duality.  But if
the transpose-dual of an algorithm is applied to a permutation, the
resulting growth diagram's node values are the
transposes all of the node values of the growth diagram generated by
the original algorithm, and the resulting $P$ and $Q$ tableaux are the
transposes of the $P$ and $Q$ tableaux generated by the original
algorithm.
The canonical example of transpose duality is the relationship between ``row insertion''
\ref{sec:row-insertion} and ``column insertion''
\ref{sec:col-insertion} for unshifted tableax.

There is also an extended sense of transpose duality where in addition
to transposig the shapes, a
bijection is applied to the colors of the elements of the
generalized permutation, and a bijection is applied to the $G_2$
edges and the colors of the elements of $Q$:
\begin{equation}
\left\{
\begin{array}{lll}
\psi^T_x(\alpha) & = (z^T, x, g(c)) & \textrm{\quad where } (z, x^T, c) = \psi_{x^T}(f(\alpha)) \\
\psi^T_x(x, t, c) & = (z^T, x, g(c)) & \textrm{\quad where } (z, x^T, c) = \psi_{x^T}(x^T, t^T, g^{-1}(c))
\end{array}
\right.
\end{equation}

\section{The insertion diagram for an insertion algorithm}
\label{sec:insert-diagram}

We next show a convenient way of displaying an R-correspondence $\psi$
that specifies $\psi_x$ for each shape $x$.

Fundamental to defining $\psi_x$ are the shapes $x^-$ covered
by $x$ in $V$, the shapes $x^+$ that cover $x$, and the set of colors
of non-zero elements
of generalized permutations $\alpha$ (which is $[r]$).

In the archetype instantiation for unshifted tableaux,
we can choose as an example $x$ the shape $(5, 3, 3, 1)$, shown
in figure~\ref{fig:tableau53311}(a).
\begin{figure}
\hfill
\begin{ytableau}
\relax & & & & \\
 & & \\
 & & \\
\relax \\
\relax
\end{ytableau}
\hspace{1em}
\hfill
\begin{ytableau}
\relax & & & & - & \none[+] \\
 & & & \none[+] \\
 & & - \\
\relax & \none[+] \\
- \\
\none[+]
\end{ytableau}
\hfill
$\begin{tikzcd}[baseline=(\tikzcdmatrixname-2-1.south), cramped, sep=small, execute at end picture={
    \horizlinex{1}{2}{7}
    \vertline{2}{6}
    \horizlinex{2}{5}{6}
    \vertlinex{3}{4}{4}
    \horizlinex{4}{3}{4}
    \vertlinex{5}{6}{2}
    \horizline{6}{2}
    \vertlinex{2}{7}{1}
}]
\null & \null  & \null  & \null  & \null & \null & \null & \null \\
\null & \null  & \null  & \null & \null & - & + \\
\null & \null  & \null  & \null & + & \null & \null \\
\null & \null  & \null  & - & \null \\
\null & \null  & + & \null &   &   \\
\null & - & \null &   &   &   \\
\null & +  &   &   &   &   \\
\null & \null &   &   &   & 
\end{tikzcd}$
\hfill
\caption[The shape $(5, 3, 3, 1)$ with its maximal and cominimal points]%
{(a) The shape $(5, 3, 3, 1)$, (b) its maximal and cominimal
  points, (c) the same shown more abstractly}
\label{fig:tableau53311}
\end{figure}
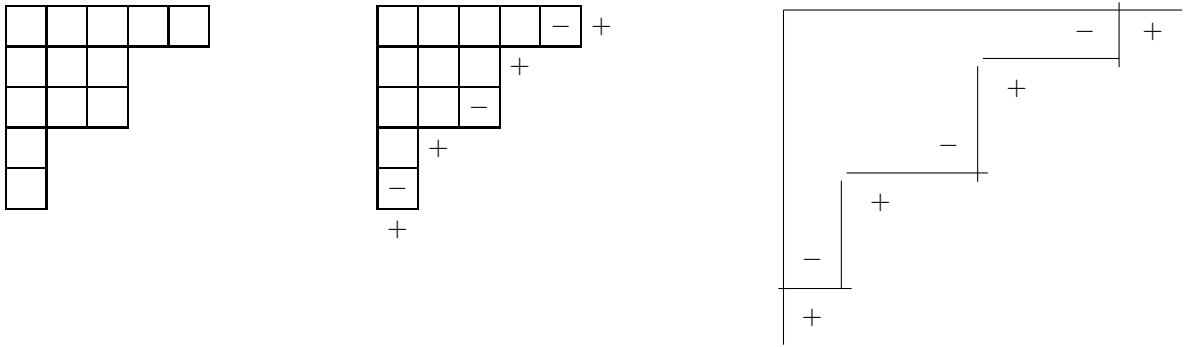
We label the points which can be removed from $x$ to form the
shapes $x^- \lessdot x$ with ``$-$'' and label the points that can be
added to $x$ to form the shapes $x^+ \gtrdot x$ with ``$+$'', as in
figure~\ref{fig:tableau53311}(b).
The points $x^-$ that can be removed from $x$ are maximal in $x$; we also
call them \textit{deletion points} of $x$.
The points $x^+$
that can be added to $x$ are \textit{cominimal}, the minimal points in
the complement of $x$ (with respect to $B$);
we also call them \textit{insertion points} of $x$.

For unshifted tableaux, the diagrams for all shapes $x$ have a similar pattern of
maximal and cominimal points, viz.\ progressing from northeast to
southwest, they alternate, with a ``$+$'' at both the northeast
and southwest ends.  For each ``$-$'', there is a ``$+$'' in
the next row south of it and a ``$+$'' in the next column east of it.
This generic pattern is shown in the more abstracted
figure~\ref{fig:tableau53311}(c).

For the archetype instantiation,
we can create an \textit{insertion diagram} for the Robinson-Schensted
insertion algorithm showing the generic $\psi$ as in
figure~\ref{fig:psi-RS}.  It encodes $\psi_x$ via:
\begin{figure}
\begin{tikzcd}[cramped, sep=small, execute at end picture={
    \horizline{1}{6}
    \vertline{2}{5}
    \horizline{2}{5}
    \vertline{3}{4}
    \horizline{3}{4}
    \vertline{4}{3}
    \horizline{4}{3}
    \vertline{5}{2}
    \horizline{5}{2}
    \vertline{6}{1}
}]
\null & \null  & \null  & \null  & \null & \hskip 1em \ar[d, "\alpha=1" near start, bend left=50] & \null \\
\null & \null  & \null  & \null & - \ar[d, bend right=50] & + \\
\null & \null  & \null  & - \ar[d, bend right=50] & + &   \\
\null & \null  & - \ar[d, bend right=50] & + &   &   \\
\null & - \ar[d, bend right=50] & + &   &   &   \\
\null & + &   &   &   &   \\
\null &   &   &   &   &
\end{tikzcd}
\caption{Generic insertion diagram for Robinson-Schensted
  (row) insertion}
\label{fig:psi-RS}
\end{figure}
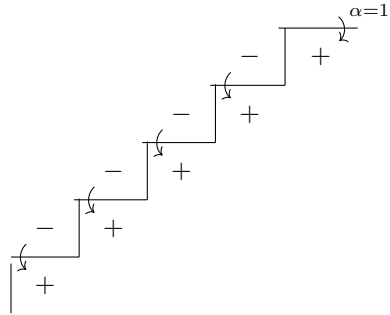
\begin{enumerate}
\item[] case ``$X$'' --
For the one $\alpha$ color $1$,
consider $\psi_x(1) = (x^+, x, c) \in G_2$ where $x^+ = x \uplus \{q\}$.
In the insertion diagram,
there is an unanchored arrow, labeled with ``$\alpha = 1$'', to the point
$q$ in the northeast.  If $w(q) > 1$,
the edge $x \overset{G_2}{\longleftarrow} x^+$ is multiple, we
would also label this edge with the color $c$ that selected that
edge in $G_2$, but for unshifted tableaux, $w(q) = 1$.
Arrows of this form specify the processing of cells where $\alpha \ne 0$ as
in figure~\ref{fig:ins-diag-cell}(a).
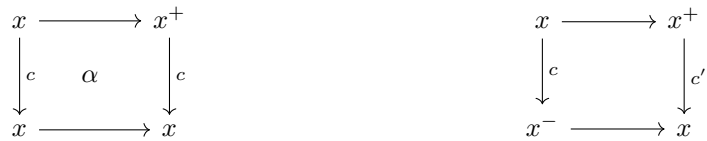
\begin{figure}
\hbox to \linewidth{\hfill
\begin{tikzcd}[sep=small]
  x \ar[rr]\ar[dd,"c"] & & x^+ \ar[dd,"c"] \\
& \alpha \\
  x \ar[rr] & & x
\end{tikzcd}
\hfill
\begin{tikzcd}[sep=small]
  x \ar[rr]\ar[dd,"c"] & & x^+ \ar[dd,"c^\prime"] \\
& \null \\
  x^- \ar[rr] & & x
\end{tikzcd}
\hfill}
\caption[Cells representing insertion and bumping]%
{Cells (a) representing insertion and (b) representing bumping}
\label{fig:ins-diag-cell}
\end{figure}
\item[] case ``$\ast$'' --
Now consider $(x, x^-, c) \in G_2$ where $x^- = x \setminus \{p\}$
and $c \in [w(p)]$.
The value $\psi_x(x, x^-, c) = (x^+, x, c^\prime) \in G_2$
for some $x^+ = x \uplus\{q\}$ and $c^\prime \in [w(q)]$.
In the insertion diagram, there is an arrow from $p$ (which is a
maximal point of $x$) to
$q$ (which is a cominimal point of $x$).
We label the arrow with the combination $c/c^\prime$ to show the
colors of the ``downward from $x$'' and ``downward to $x$'' edges,
respectively.  If either $w(p) = 1$ or $w(q) = 1$, we omit
$c$ or $c^\prime$ from the label (resp.) and for the archetype instantion, we
can omit the label entirely.
Arrows of this form specify the processing of cells where $\alpha = 0$ as
in figure~\ref{fig:ins-diag-cell}(b).
\end{enumerate}
In order for the arrows of an insertion diagram to define an
acceptable $\psi_x$, the following constraints must be met:
\begin{enumerate}
\item
There must be one ``$\alpha = \cdot$'' arrow to
some insertion point for each $\alpha$ color.
\item
For each deletion point $p$, there must be $w(p)$ arrows in the diagram
leading from it, with one labeled with each of the colors in $[w(p)]$.
\item
For each insertion point $q$, there must be $w(q)$ arrows
in the diagram leading to it, with one labeled with each of the colors
in $[w(q)]$.
\end{enumerate}

For an instantiation, there need not be a single generic pattern for
the insertion diagrams for all $x \in V$.  As \cite{McLar1986} emphasizes,
an insertion diagram
could be independently chosen for each $x$, and the assemblage of them
would define an R-correspondence $\psi$ for the
w.d.g.g.\ $\mathcal{G}$.  Since $V$ is countably infinite, there
are a continuum number of possible $\psi$.  Nonetheless, the more
interesting insertion algorithms have very systematic
R-correspondences which can be represented by one or a few generic
insertion diagrams.

Of course, insertion diagrams provide no more information than
specifying $\psi$ directly (e.g.~as in \cite{Fom1995a}*{sec.~4.2 and
4.5}), but they provide the information in a way that is much more intuitively
accessible, as will be shown in later sections that discuss
various instantiations of the theory.

As \cite{McLar1986} notes,
all R-correspondences for a particular lattice of diagrams
and weight function prove the same basic enumeration formulas.
However, a particular R-correspondence may have distinct enumerative
value if it transforms an interesting subset of permutations into
an interesting subset $P$/$Q$ pairs.

\section{Unshifted tableaux algorithms}

\paragraph{Robinson-Schensted insertion}
\label{sec:row-insertion}
The archetype instantiation of an insertion algorithm is
Robinson-Schensted insertion for unshifted tableaux
\cite{Schen1961}, also called
``row insertion'' and ``Schensted insertion'' in \cite{Haim1989a}*{sec.~2}.
Its generic insertion diagram is shown in shown in figure~\ref{fig:psi-RS}.

\paragraph{Column insertion}
\label{sec:col-insertion}
Similarly, the generic insertion diagram for the ``column
insertion'' variant of the Robinson-Schensted algorithm for unshifted
tableaux is shown in figure~\ref{fig:psi-RScol}.
\begin{figure}
\begin{tikzcd}[cramped, sep=small, execute at end picture={
    \horizline{1}{6}
    \vertline{2}{5}
    \horizline{2}{5}
    \vertline{3}{4}
    \horizline{3}{4}
    \vertline{4}{3}
    \horizline{4}{3}
    \vertline{5}{2}
    \horizline{5}{2}
    \vertline{6}{1}
}]
\null & \null  & \null  & \null  & \null & \null & \null \\
\null & \null  & \null  & \null & - \ar[r, bend right=50] & + \\
\null & \null  & \null  & - \ar[r, bend right=50] & + &   \\
\null & \null  & - \ar[r, bend right=50] & + &   &   \\
\null & - \ar[r, bend right=50] & + &   &   &   \\
\null & + &   &   &   &   \\
\null & \hskip 1em \ar[u, "\alpha=1"', bend right=50] &   &   &   &
\end{tikzcd}
\caption{Generic insertion diagram for column insertion into unshifted tableaux}
\label{fig:psi-RScol}
\end{figure}
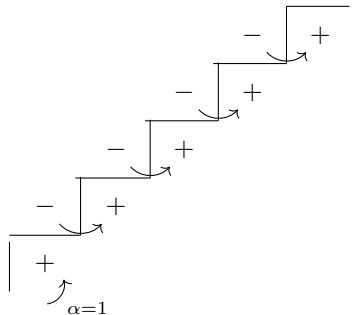
Note that in this paper, all sequences of inserted values are
generalized permutations and so have no repeated values (and so we do
not attach Knuth's name to the algorithms).
Thus, 
column insertion is the transpose-dual of row insertion, which can be
seen by the fact that its insertion diagram,
figure~\ref{fig:psi-RScol}, is the transpose of the insertion diagram
of row insertion, figure~\ref{fig:psi-RS}.

\paragraph{Haiman's left-right insertion}
\label{sec:left-right}
Given a particular lattice of diagrams $V$,
the weight equation~\ref{eq:weight} is linear on
the differential degree $r$ and the weight function $w$ together.
The archetype
instantiation has $r = 1$ and $w(p) = 1$, so another instantiation on
the lattice of Young diagrams is the ``double'' of it, $r = 2$ and $w(p) = 2$.

Haiman \cite{Haim1989a}*{sec.~4} constructed an insertion algorithm
for this doubled instantiation
that is a combination of row and column insertion, called ``left-right
insertion''.  Since $r=2$, there are two colors in the
generalized permutation, with $\alpha = 1$ called uncircled and
$\alpha = 2$ called circled.  Uncircled elements are inserted with row
insertion and circled elements are inserted with column insertion,
with the elements' circling recorded in the $Q$ tableau.  The
insertion diagram for left-right insertion is a straightforward
combination of the insertion diagrams for row and column insertion, as
shown in figure~\ref{fig:psi-left-right}, where the edges in $G_2$
have two labels, represented ``$U$'' and ``$C$'', and the values of
$\alpha$ are similarly represented by ``$U$'' and ``$C$''.
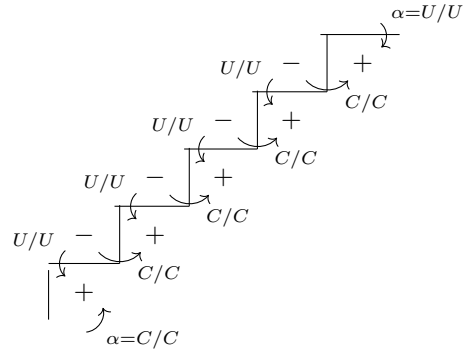
\begin{figure}
\begin{tikzcd}[cramped, sep=small, execute at end picture={
    \horizline{1}{6}
    \vertline{2}{5}
    \horizline{2}{5}
    \vertline{3}{4}
    \horizline{3}{4}
    \vertline{4}{3}
    \horizline{4}{3}
    \vertline{5}{2}
    \horizline{5}{2}
    \vertline{6}{1}
}]
\null & \null  & \null  & \null  & \null & \hskip 1em \ar[d, "\alpha=U/U" near start, bend left=50] & \null \\
\null & \null  & \null  & \null & - \ar[d, "U/U"' near start, bend right=50] \ar[r, "C/C"' near end, bend right=50] & + \\
\null & \null  & \null  & - \ar[d, "U/U"' near start, bend right=50] \ar[r, "C/C"' near end, bend right=50] & + &   \\
\null & \null  & - \ar[d, "U/U"' near start, bend right=50] \ar[r, "C/C"' near end, bend right=50] & + &   &   \\
\null & - \ar[d, "U/U"' near start, bend right=50] \ar[r, "C/C"' near end, bend right=50] & + &   &   &   \\
\null & + &   &   &   &   \\
\null & \hskip 1em \ar[u, "\alpha=C/C"', bend right=50] &   &   &   &
\end{tikzcd}
\caption{Generic insertion diagram for left-right insertion}
\label{fig:psi-left-right}
\end{figure}

An example of a left-right insertion as a growth diagram is
figure~\ref{fig:growth-left-right}.
Note that for any insertion/bump sequence, the colors of all $G_2$
edges in that row of the growth are the same, and the circling of
an element $i$ in the $Q$ tableau is the same as the circling of the
$i$-th element of the generalized permutation, which created it. 

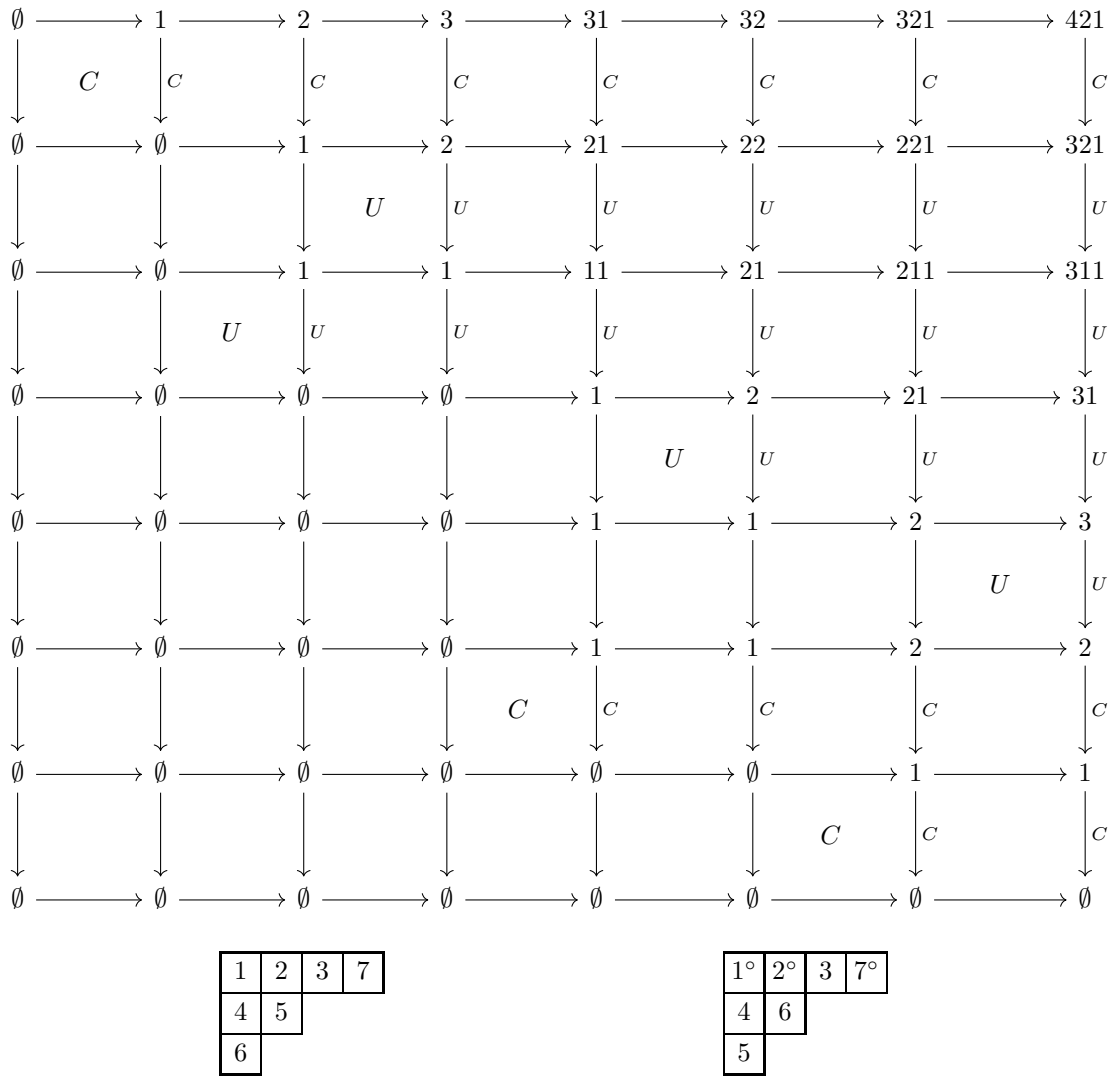
\begin{figure}
\growthtableauxv{%
\begin{tikzcd}[ampersand replacement=\&,sep=small]
\emptyset \ar[rr]\ar[dd] \&\&
  1 \ar[rr]\ar[dd,"C"] \&\&
  2 \ar[rr]\ar[dd,"C"] \&\&
  3 \ar[rr]\ar[dd,"C"] \&\&
  31 \ar[rr]\ar[dd,"C"] \&\&
  32 \ar[rr]\ar[dd,"C"] \&\&
  321 \ar[rr]\ar[dd,"C"] \&\&
  421 \ar[dd,"C"] \\
\& C \&\&  \&\&  \&\&  \&\&  \&\&  \&\&  \& \\
\emptyset \ar[rr]\ar[dd] \&\&
  \emptyset \ar[rr]\ar[dd] \&\&
  1 \ar[rr]\ar[dd] \&\&
  2 \ar[rr]\ar[dd,"U"] \&\&
  21 \ar[rr]\ar[dd,"U"] \&\&
  22 \ar[rr]\ar[dd,"U"] \&\&
  221 \ar[rr]\ar[dd,"U"] \&\&
  321 \ar[dd,"U"] \\
\&  \&\&  \&\& U \&\&  \&\&  \&\&  \&\&  \& \\
\emptyset \ar[rr]\ar[dd] \&\&
  \emptyset \ar[rr]\ar[dd] \&\&
  1 \ar[rr]\ar[dd,"U"] \&\&
  1 \ar[rr]\ar[dd,"U"] \&\&
  11 \ar[rr]\ar[dd,"U"] \&\&
  21 \ar[rr]\ar[dd,"U"] \&\&
  211 \ar[rr]\ar[dd,"U"] \&\&
  311 \ar[dd,"U"] \\
\&  \&\& U \&\&  \&\&  \&\&  \&\&  \&\&  \& \\
\emptyset \ar[rr]\ar[dd] \&\&
  \emptyset \ar[rr]\ar[dd] \&\&
  \emptyset \ar[rr]\ar[dd] \&\&
  \emptyset \ar[rr]\ar[dd] \&\&
  1 \ar[rr]\ar[dd] \&\&
  2 \ar[rr]\ar[dd,"U"] \&\&
  21 \ar[rr]\ar[dd,"U"] \&\&
  31 \ar[dd,"U"] \\
\&  \&\&  \&\&  \&\&  \&\& U \&\&  \&\&  \& \\
\emptyset \ar[rr]\ar[dd] \&\&
  \emptyset \ar[rr]\ar[dd] \&\&
  \emptyset \ar[rr]\ar[dd] \&\&
  \emptyset \ar[rr]\ar[dd] \&\&
  1 \ar[rr]\ar[dd] \&\&
  1 \ar[rr]\ar[dd] \&\&
  2 \ar[rr]\ar[dd] \&\&
  3 \ar[dd,"U"] \\
\&  \&\&  \&\&  \&\&  \&\&  \&\&  \&\& U \& \\
\emptyset \ar[rr]\ar[dd] \&\&
  \emptyset \ar[rr]\ar[dd] \&\&
  \emptyset \ar[rr]\ar[dd] \&\&
  \emptyset \ar[rr]\ar[dd] \&\&
  1 \ar[rr]\ar[dd,"C"] \&\&
  1 \ar[rr]\ar[dd,"C"] \&\&
  2 \ar[rr]\ar[dd,"C"] \&\&
  2 \ar[dd,"C"] \\
\&  \&\&  \&\&  \&\& C \&\&  \&\&  \&\&  \& \\
\emptyset \ar[rr]\ar[dd] \&\&
  \emptyset \ar[rr]\ar[dd] \&\&
  \emptyset \ar[rr]\ar[dd] \&\&
  \emptyset \ar[rr]\ar[dd] \&\&
  \emptyset \ar[rr]\ar[dd] \&\&
  \emptyset \ar[rr]\ar[dd] \&\&
  1 \ar[rr]\ar[dd,"C"] \&\&
  1 \ar[dd,"C"] \\
\&  \&\&  \&\&  \&\&  \&\&  \&\& C \&\&  \& \\
\emptyset \ar[rr] \&\&
  \emptyset \ar[rr] \&\&
  \emptyset \ar[rr] \&\&
  \emptyset \ar[rr] \&\&
  \emptyset \ar[rr] \&\&
  \emptyset \ar[rr] \&\&
  \emptyset \ar[rr] \&\&
  \emptyset 
\end{tikzcd}
}{%
\begin{ytableau}
1 & 2 & 3 & 7 \\
4 & 5 \\
6
\end{ytableau}
}{%
\begin{ytableau}
1^\circ & 2^\circ & 3 & 7^\circ \\
4 & 6 \\
5
\end{ytableau}
}
\caption[Growth for left-right insertion]%
{Growth for the permutation
$(6^\circ, 4^\circ, 7, 5, 2, 3, 1^\circ)$ for left-right insertion:
(a) growth diagram, (b) $P$ tableau, and (c) $Q$ tableau}
\label{fig:growth-left-right}
\end{figure}

Left-right insertion has transposition duality when the circling of
all elements of the generalized permutation is inverted.  This can be
seen by inspecting the insertion diagram
figure~\ref{fig:psi-left-right}:  transposing it and replacing all
``$U$'' with ``$C$'' and vice-versa leaves it unchanged.

\paragraph{McLarnan's fairy insertion}
\cite{McLar1986} and \cite{GarMcLar1987} introduce ``generic''
insertion algorithms, parameterized algorithms that generate a family
of different insertion algorithms.  Any such family makes it easy to
define many algorithms with baroque and unusual structures.
These aspects resemble the creation of unconventional pieces in
fairy chess so we call these algorithms ``fairy'' algorithms.

\cite{McLar1986}*{Fig.~1.4} and \cite{GarMcLar1987}*{Fig.~4}
illustrate a relatively simple fairy algorithm.
The generic insertion diagram for this algorithm is
figure~\ref{fig:psi-unshifted-mclarnan} and is not as visually simple
as for other algorithms;
the arrows from the deletion points to the insertion points, instead
of preserving the order of the two sets, reverse it,
reserving the northeastmost insertion point as the
``$\alpha=\cdot$'' point.  In McLarnan's words,
\begin{quote}
A slightly more complex bumping scheme might take the bottom removable
square to the top addible square, the second lowest removable square
to the second highest addible square, and so on. For this scheme as
for the row insertion bumping scheme, the lone square is always in the
first row.
\end{quote}
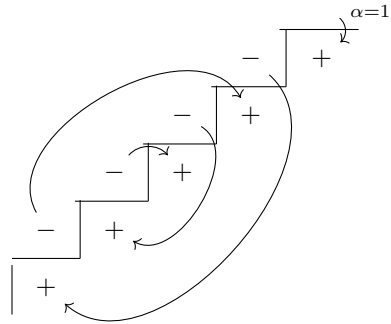
\begin{figure}
\begin{tikzcd}[cramped, sep=small, execute at end picture={
    \horizline{1}{6}
    \vertline{2}{5}
    \horizline{2}{5}
    \vertline{3}{4}
    \horizline{3}{4}
    \vertline{4}{3}
    \horizline{4}{3}
    \vertline{5}{2}
    \horizline{5}{2}
    \vertline{6}{1}
}]
\null & \null  & \null  & \null  & \null & \hskip 1em \ar[d, "\alpha=1" near start, bend left=50] & \null \\
\null & \null  & \null  & \null & - \ar[ddddlll, bend left=90] & + \\
\null & \null  & \null  & - \ar[ddl, bend left=90] & + &   \\
\null & \null  & - \ar[r, bend left=50] & + &   &   \\
\null & - \ar[uurrr, bend left=90] & + &   &   &   \\
\null & + &   &   &   &   \\
\null &   &   &   &   &
\end{tikzcd}
\caption{Generic insertion diagram for McLarnan's fairy insertion}
\label{fig:psi-unshifted-mclarnan}
\end{figure}
An example of the algorithm is
figure~\ref{fig:growth-unshifted-mclarnan}, taken from
\cite{GarMcLar1987}*{Fig.~7}.
\begin{figure}
\growthtableauxv{%
\begin{tikzcd}[ampersand replacement=\&,sep=small]
\emptyset \ar[rr]\ar[dd] \&\&
  1 \ar[rr]\ar[dd] \&\&
  11 \ar[rr]\ar[dd] \&\&
  21 \ar[rr]\ar[dd] \&\&
  211 \ar[rr]\ar[dd] \&\&
  2111 \ar[rr]\ar[dd] \&\&
  2211 \ar[rr]\ar[dd] \&\&
  3211 \ar[dd] \\
\&  \&\&  \&\& X \&\&  \&\&  \&\&  \&\&  \& \\
\emptyset \ar[rr]\ar[dd] \&\&
  1 \ar[rr]\ar[dd] \&\&
  11 \ar[rr]\ar[dd] \&\&
  11 \ar[rr]\ar[dd] \&\&
  111 \ar[rr]\ar[dd] \&\&
  211 \ar[rr]\ar[dd] \&\&
  221 \ar[rr]\ar[dd] \&\&
  321 \ar[dd] \\
\&  \&\&  \&\&  \&\&  \&\&  \&\&  \&\& X \& \\
\emptyset \ar[rr]\ar[dd] \&\&
  1 \ar[rr]\ar[dd] \&\&
  11 \ar[rr]\ar[dd] \&\&
  11 \ar[rr]\ar[dd] \&\&
  111 \ar[rr]\ar[dd] \&\&
  211 \ar[rr]\ar[dd] \&\&
  221 \ar[rr]\ar[dd] \&\&
  221 \ar[dd] \\
\& X \&\&  \&\&  \&\&  \&\&  \&\&  \&\&  \& \\
\emptyset \ar[rr]\ar[dd] \&\&
  \emptyset \ar[rr]\ar[dd] \&\&
  1 \ar[rr]\ar[dd] \&\&
  1 \ar[rr]\ar[dd] \&\&
  11 \ar[rr]\ar[dd] \&\&
  21 \ar[rr]\ar[dd] \&\&
  211 \ar[rr]\ar[dd] \&\&
  211 \ar[dd] \\
\&  \&\&  \&\&  \&\&  \&\& X \&\&  \&\&  \& \\
\emptyset \ar[rr]\ar[dd] \&\&
  \emptyset \ar[rr]\ar[dd] \&\&
  1 \ar[rr]\ar[dd] \&\&
  1 \ar[rr]\ar[dd] \&\&
  11 \ar[rr]\ar[dd] \&\&
  11 \ar[rr]\ar[dd] \&\&
  21 \ar[rr]\ar[dd] \&\&
  21 \ar[dd] \\
\&  \&\&  \&\&  \&\&  \&\&  \&\& X \&\&  \& \\
\emptyset \ar[rr]\ar[dd] \&\&
  \emptyset \ar[rr]\ar[dd] \&\&
  1 \ar[rr]\ar[dd] \&\&
  1 \ar[rr]\ar[dd] \&\&
  11 \ar[rr]\ar[dd] \&\&
  11 \ar[rr]\ar[dd] \&\&
  11 \ar[rr]\ar[dd] \&\&
  11 \ar[dd] \\
\&  \&\& X \&\&  \&\&  \&\&  \&\&  \&\&  \& \\
\emptyset \ar[rr]\ar[dd] \&\&
  \emptyset \ar[rr]\ar[dd] \&\&
  \emptyset \ar[rr]\ar[dd] \&\&
  \emptyset \ar[rr]\ar[dd] \&\&
  1 \ar[rr]\ar[dd] \&\&
  1 \ar[rr]\ar[dd] \&\&
  1 \ar[rr]\ar[dd] \&\&
  1 \ar[dd] \\
\&  \&\&  \&\&  \&\& X \&\&  \&\&  \&\&  \& \\
\emptyset \ar[rr] \&\&
  \emptyset \ar[rr] \&\&
  \emptyset \ar[rr] \&\&
  \emptyset \ar[rr] \&\&
  \emptyset \ar[rr] \&\&
  \emptyset \ar[rr] \&\&
  \emptyset \ar[rr] \&\&
  \emptyset 
\end{tikzcd}
}{%
\begin{ytableau}
1 & 3 & 7 \\
2 & 6 \\
4 \\
5
\end{ytableau}
}{%
\begin{ytableau}
1 & 3 & 6 \\
2 & 5 \\
4 \\
7
\end{ytableau}
}
\caption[Growth for McLarnan's fairy insertion]%
{Growth for the permutation
$(4, 2, 6, 5, 1, 7, 3)$ for McLarnan's fairy insertion:
(a) growth diagram, (b) $P$ tableau, and (c) $Q$ tableau}
\label{fig:growth-unshifted-mclarnan}
\end{figure}

\paragraph{Jitter fairy insertion}
\label{sec:jitter-insertion}
Other algorithms we catalog here leave the labels (circling)
unchanged for most bumping operations, occasionally changing labels
when a bump involves the northeast or southwest corner of the
tableau.  We now construct a fairy algorithm that changes labels for
\textit{all} bumps, and label it \textit{jitter insertion}.
The generic insertion diagram for this algorithm is
figure~\ref{fig:psi-jitter}.
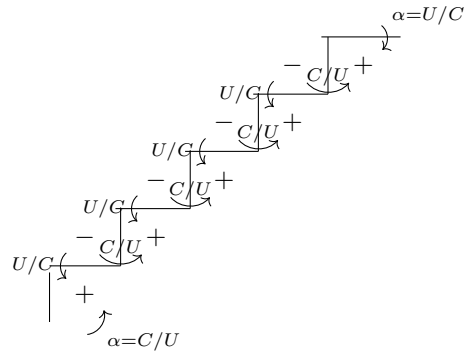
\begin{figure}
\begin{tikzcd}[cramped, sep=small, execute at end picture={
    \horizline{1}{6}
    \vertline{2}{5}
    \horizline{2}{5}
    \vertline{3}{4}
    \horizline{3}{4}
    \vertline{4}{3}
    \horizline{4}{3}
    \vertline{5}{2}
    \horizline{5}{2}
    \vertline{6}{1}
}]
\null & \null  & \null  & \null  & \null & \hskip 1em \ar[d, "\alpha=U/C" near start, bend left=50] & \null \\
\null & \null  & \null  & \null & - \ar[d, "U/C"', bend right=50] \ar[r, "C/U", bend right=50] & + \\
\null & \null  & \null  & - \ar[d, "U/C"', bend right=50] \ar[r, "C/U", bend right=50] & + &   \\
\null & \null  & - \ar[d, "U/C"', bend right=50] \ar[r, "C/U", bend right=50] & + &   &   \\
\null & - \ar[d, "U/C"', bend right=50] \ar[r, "C/U", bend right=50] & + &   &   &   \\
\null & + &   &   &   &   \\
\null & \hskip 1em \ar[u, "\alpha=C/U"', bend right=50] &   &   &   &
\end{tikzcd}
\caption{Generic insertion diagram for the ``jitter'' insertion}
\label{fig:psi-jitter}
\end{figure}
It resembles left-right insertion~\ref{sec:left-right}, except that
every time a number is inserted or bumped, its circling is reversed,
and so the next time the number is bumped, it will move in the
reverse direction in the tableau.
Examples of this algorithm are figures~\ref{fig:growth-jitter-1},
\ref{fig:growth-jitter-2},  \ref{fig:growth-jitter-3},
and \ref{fig:growth-jitter-4}.
\begin{figure}
\growthtableauxh{%
\begin{tikzcd}[ampersand replacement=\&,sep=small]
\emptyset \ar[rr]\ar[dd] \&\&
  1 \ar[rr]\ar[dd] \&\&
  2 \ar[rr]\ar[dd] \&\&
  3 \ar[rr]\ar[dd] \&\&
  4 \ar[dd,"C"] \\
\&  \&\&  \&\&  \&\& U \& \\
\emptyset \ar[rr]\ar[dd] \&\&
  1 \ar[rr]\ar[dd] \&\&
  2 \ar[rr]\ar[dd] \&\&
  3 \ar[rr]\ar[dd,"C"] \&\&
  3 \ar[dd,"C"] \\
\&  \&\&  \&\& U \&\&  \& \\
\emptyset \ar[rr]\ar[dd] \&\&
  1 \ar[rr]\ar[dd] \&\&
  2 \ar[rr]\ar[dd,"C"] \&\&
  2 \ar[rr]\ar[dd,"C"] \&\&
  2 \ar[dd,"C"] \\
\&  \&\& U \&\&  \&\&  \& \\
\emptyset \ar[rr]\ar[dd] \&\&
  1 \ar[rr]\ar[dd,"C"] \&\&
  1 \ar[rr]\ar[dd,"C"] \&\&
  1 \ar[rr]\ar[dd,"C"] \&\&
  1 \ar[dd,"C"] \\
\& U \&\&  \&\&  \&\&  \& \\
\emptyset \ar[rr] \&\&
  \emptyset \ar[rr] \&\&
  \emptyset \ar[rr] \&\&
  \emptyset \ar[rr] \&\&
  \emptyset 
\end{tikzcd}
}{%
\begin{ytableau}
1 & 2 & 3 & 4
\end{ytableau}
}{%
\begin{ytableau}
1^\circ & 2^\circ & 3^\circ & 4^\circ
\end{ytableau}
}
\caption[Growth for the permutation $(1, 2, 3, 4)$ for the ``jitter'' algorithm]%
{Growth for the permutation $(1, 2, 3, 4)$ for the ``jitter'' algorithm:
(a) growth diagram, (b) $P$ tableau, and (c) $Q$ tableau}
\label{fig:growth-jitter-1}
\end{figure}
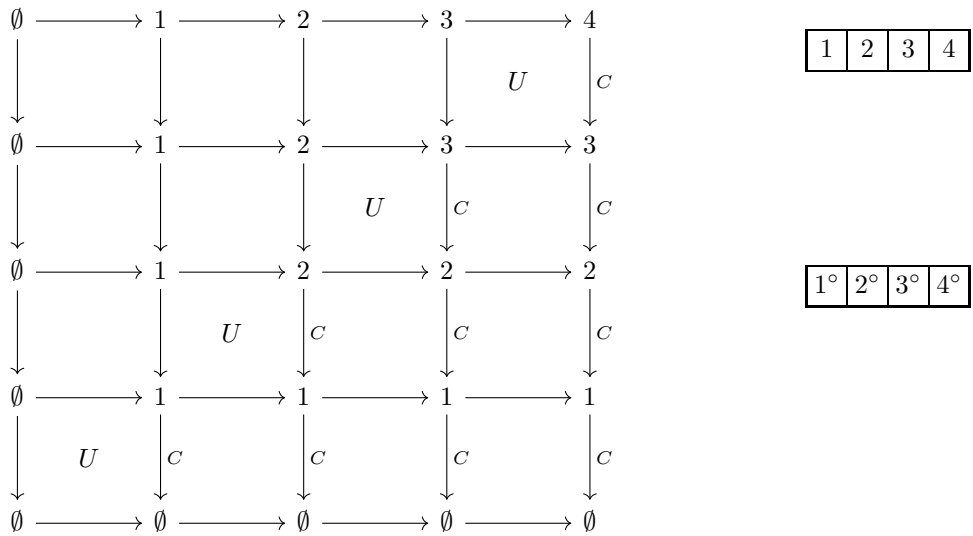
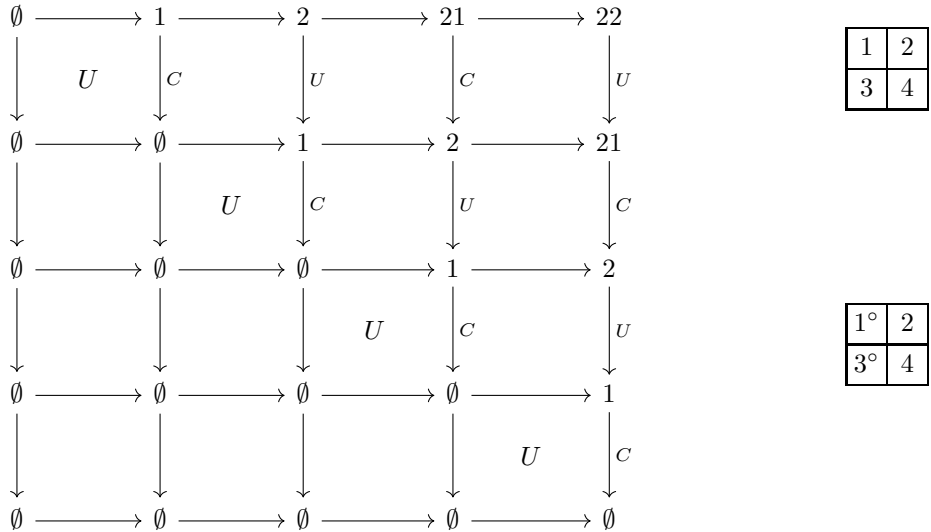
\begin{figure}
\growthtableauxh{%
\begin{tikzcd}[ampersand replacement=\&,sep=small]
\emptyset \ar[rr]\ar[dd] \&\&
  1 \ar[rr]\ar[dd,"C"] \&\&
  2 \ar[rr]\ar[dd,"U"] \&\&
  21 \ar[rr]\ar[dd,"C"] \&\&
  22 \ar[dd,"U"] \\
\& U \&\&  \&\&  \&\&  \& \\
\emptyset \ar[rr]\ar[dd] \&\&
  \emptyset \ar[rr]\ar[dd] \&\&
  1 \ar[rr]\ar[dd,"C"] \&\&
  2 \ar[rr]\ar[dd,"U"] \&\&
  21 \ar[dd,"C"] \\
\&  \&\& U \&\&  \&\&  \& \\
\emptyset \ar[rr]\ar[dd] \&\&
  \emptyset \ar[rr]\ar[dd] \&\&
  \emptyset \ar[rr]\ar[dd] \&\&
  1 \ar[rr]\ar[dd,"C"] \&\&
  2 \ar[dd,"U"] \\
\&  \&\&  \&\& U \&\&  \& \\
\emptyset \ar[rr]\ar[dd] \&\&
  \emptyset \ar[rr]\ar[dd] \&\&
  \emptyset \ar[rr]\ar[dd] \&\&
  \emptyset \ar[rr]\ar[dd] \&\&
  1 \ar[dd,"C"] \\
\&  \&\&  \&\&  \&\& U \& \\
\emptyset \ar[rr] \&\&
  \emptyset \ar[rr] \&\&
  \emptyset \ar[rr] \&\&
  \emptyset \ar[rr] \&\&
  \emptyset 
\end{tikzcd}
}{%
\begin{ytableau}
1 & 2 \\
3 & 4
\end{ytableau}
}{%
\begin{ytableau}
1^\circ & 2 \\
3^\circ & 4
\end{ytableau}
}
\caption[Growth for the permutation $(4, 3, 2, 1)$ for the ``jitter'' algorithm]%
{Growth for the permutation $(4, 3, 2, 1)$ for the ``jitter'' algorithm:
(a) growth diagram, (b) $P$ tableau, and (c) $Q$ tableau}
\label{fig:growth-jitter-2}
\end{figure}
\begin{figure}
\growthtableauxh{%
\begin{tikzcd}[ampersand replacement=\&,sep=small]
\emptyset \ar[rr]\ar[dd] \&\&
  1 \ar[rr]\ar[dd] \&\&
  11 \ar[rr]\ar[dd] \&\&
  111 \ar[rr]\ar[dd] \&\&
  1111 \ar[dd,"U"] \\
\&  \&\&  \&\&  \&\& C \& \\
\emptyset \ar[rr]\ar[dd] \&\&
  1 \ar[rr]\ar[dd] \&\&
  11 \ar[rr]\ar[dd] \&\&
  111 \ar[rr]\ar[dd,"U"] \&\&
  111 \ar[dd,"U"] \\
\&  \&\&  \&\& C \&\&  \& \\
\emptyset \ar[rr]\ar[dd] \&\&
  1 \ar[rr]\ar[dd] \&\&
  11 \ar[rr]\ar[dd,"U"] \&\&
  11 \ar[rr]\ar[dd,"U"] \&\&
  11 \ar[dd,"U"] \\
\&  \&\& C \&\&  \&\&  \& \\
\emptyset \ar[rr]\ar[dd] \&\&
  1 \ar[rr]\ar[dd,"U"] \&\&
  1 \ar[rr]\ar[dd,"U"] \&\&
  1 \ar[rr]\ar[dd,"U"] \&\&
  1 \ar[dd,"U"] \\
\& C \&\&  \&\&  \&\&  \& \\
\emptyset \ar[rr] \&\&
  \emptyset \ar[rr] \&\&
  \emptyset \ar[rr] \&\&
  \emptyset \ar[rr] \&\&
  \emptyset 
\end{tikzcd}
}{%
\begin{ytableau}
1 \\
2 \\
3 \\
4
\end{ytableau}
}{%
\begin{ytableau}
1 \\
2 \\
3 \\
4
\end{ytableau}
}
\caption[Growth for the permutation $(1^\circ, 2^\circ, 3^\circ, 4^\circ)$ for the ``jitter'' algorithm]%
{Growth for the permutation $(1^\circ, 2^\circ, 3^\circ, 4^\circ)$ for the ``jitter'' algorithm:
(a) growth diagram, (b) $P$ tableau, and (c) $Q$ tableau}
\label{fig:growth-jitter-3}
\end{figure}
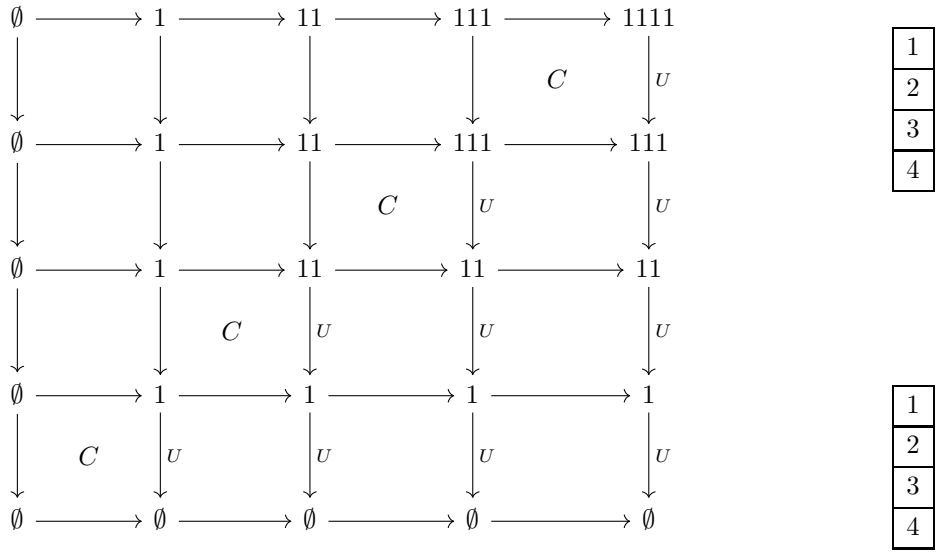
\begin{figure}
\growthtableauxh{%
\begin{tikzcd}[ampersand replacement=\&,sep=small]
\emptyset \ar[rr]\ar[dd] \&\&
  1 \ar[rr]\ar[dd,"U"] \&\&
  11 \ar[rr]\ar[dd,"C"] \&\&
  21 \ar[rr]\ar[dd,"U"] \&\&
  22 \ar[dd,"C"] \\
\& C \&\&  \&\&  \&\&  \& \\
\emptyset \ar[rr]\ar[dd] \&\&
  \emptyset \ar[rr]\ar[dd] \&\&
  1 \ar[rr]\ar[dd,"U"] \&\&
  11 \ar[rr]\ar[dd,"C"] \&\&
  21 \ar[dd,"U"] \\
\&  \&\& C \&\&  \&\&  \& \\
\emptyset \ar[rr]\ar[dd] \&\&
  \emptyset \ar[rr]\ar[dd] \&\&
  \emptyset \ar[rr]\ar[dd] \&\&
  1 \ar[rr]\ar[dd,"U"] \&\&
  11 \ar[dd,"C"] \\
\&  \&\&  \&\& C \&\&  \& \\
\emptyset \ar[rr]\ar[dd] \&\&
  \emptyset \ar[rr]\ar[dd] \&\&
  \emptyset \ar[rr]\ar[dd] \&\&
  \emptyset \ar[rr]\ar[dd] \&\&
  1 \ar[dd,"U"] \\
\&  \&\&  \&\&  \&\& C \& \\
\emptyset \ar[rr] \&\&
  \emptyset \ar[rr] \&\&
  \emptyset \ar[rr] \&\&
  \emptyset \ar[rr] \&\&
  \emptyset 
\end{tikzcd}
}{%
\begin{ytableau}
1 & 3 \\
2 & 4
\end{ytableau}
}{%
\begin{ytableau}
1 & 3 \\
2^\circ & 4^\circ
\end{ytableau}
}
\caption[Growth for the permutation $(4^\circ, 3^\circ, 2^\circ, 1^\circ)$ for the ``jitter'' algorithm]%
{Growth for the permutation $(4^\circ, 3^\circ, 2^\circ, 1^\circ)$ for the ``jitter'' algorithm:
(a) growth diagram, (b) $P$ tableau, and (c) $Q$ tableau}
\label{fig:growth-jitter-4}
\end{figure}
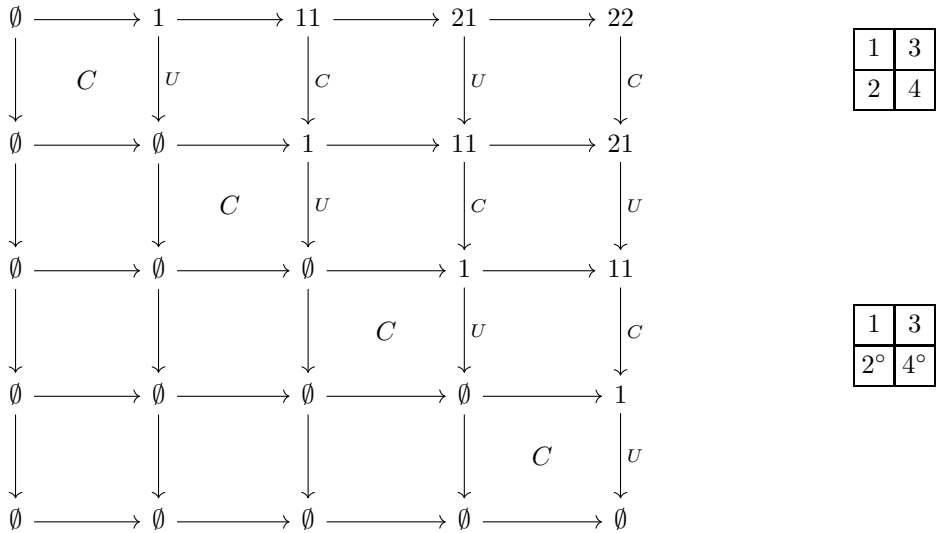

\section{Shifted tableaux algorithms}

Further instantiations of the theory use \textit{shifted tableaux}.
For these instantiations,
$B$ is the ``octant'' of the plane,
$\mbbO = \{(i, j) \textrm{ where } i, j \in \mbbP \textrm{ and } i > j\}$,
with the same order relation as the quadrant,
\begin{equation}
(i,j) \le (k,l) \longequiv i \le k \textrm{ and } j \le l \textrm{
    for all } (i,j), (k,l) \in \mbbO
\end{equation}
Like the quadrant, we display the octant in ``English format'',
with the locations of the points displayed as in
figure~\ref{fig:octant}.
\begin{figure}
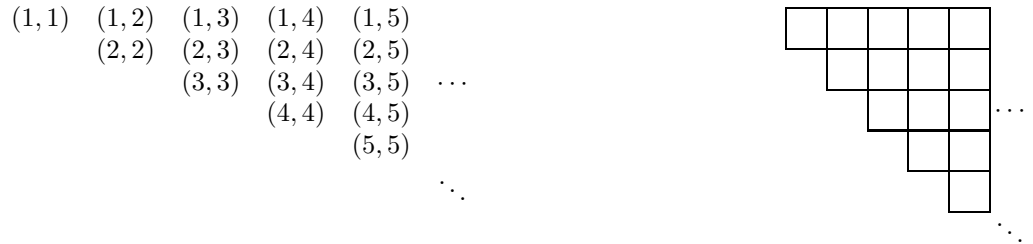

\hfill
\raisebox{-2.5em}{$\begin{array}{cccccc}
  (1,1) & (1,2) & (1,3) & (1,4) & (1,5) & \\
        & (2,2) & (2,3) & (2,4) & (2,5) & \\
        &       & (3,3) & (3,4) & (3,5) & \cdots \\
        &       &       & (4,4) & (4,5) & \\
        &       &       &       & (5,5) & \\
        &       &       &       &       & \ddots
\end{array}$}
\hfill
\begin{ytableau}
\relax & & & & & \none \\
\none & & & & & \none \\
\none & \none & & & & \none[\cdots] \\
\none & \none & \none & & & \none \\
\none & \none & \none & \none & & \none \\
\noalign{\vskip 0.2em}
\none & \none & \none & \none & \none & \none[\ddots]
\end{ytableau}
\hfill
\caption{The points of the octant $\mbbO$ with and without coordinates}
\label{fig:octant}
\end{figure}
The lattice of diagrams $V$ is the lattice of shifted Young diagrams,
$\mbbSY$.
Example \textit{shifted Young tableaux} are shown in
figure~\ref{fig:shifted-young-tableaux}.
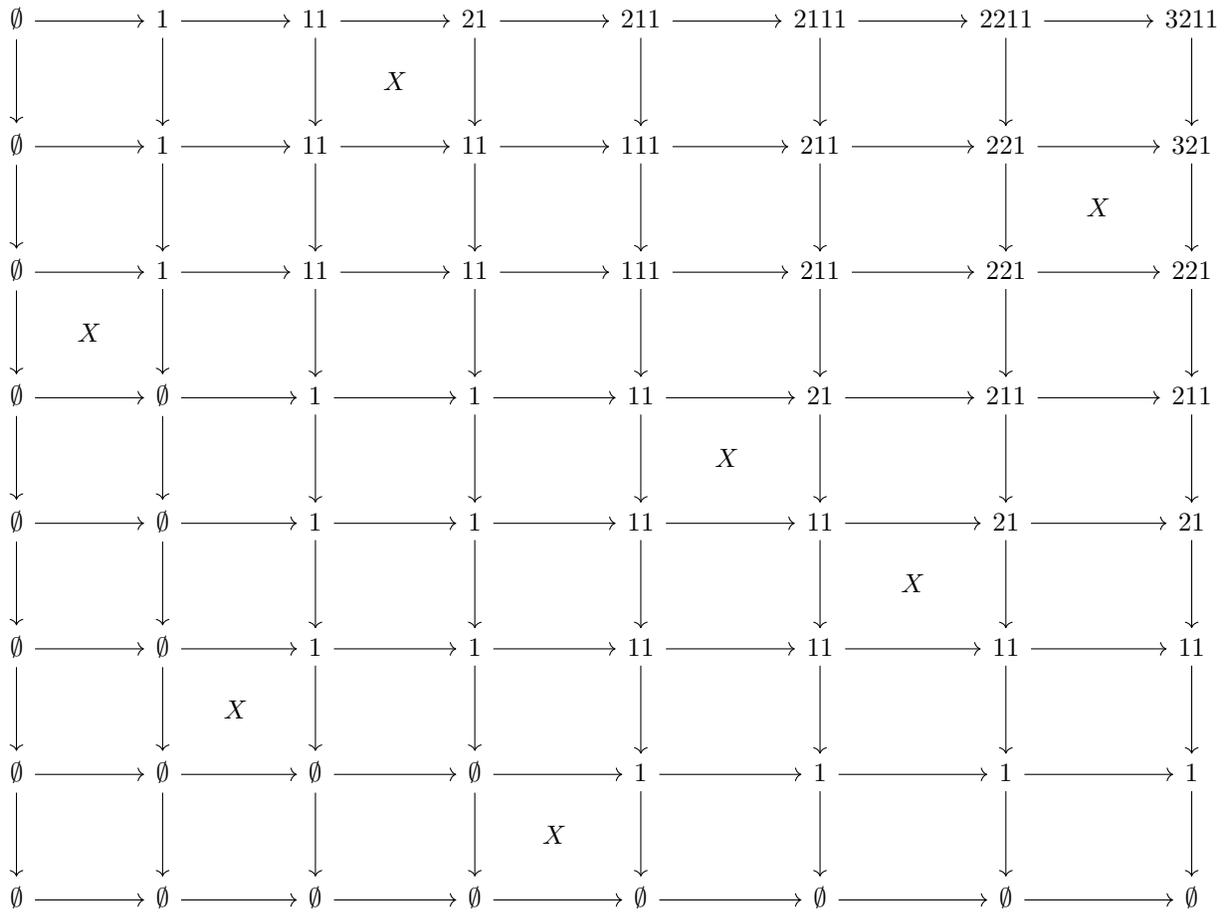
\begin{figure}
\hbox to \linewidth{\hfill
\begin{tikzcd}[baseline=(\tikzcdmatrixname-2-1.south), cramped, sep=small, execute at end picture={
    \horizlinex{1}{2}{7}
    \vertline{2}{1}
    \horizline{2}{2}
    \vertline{3}{2}
    \horizline{3}{3}
    \vertline{4}{3}
    \horizline{4}{4}
    \vertline{5}{4}
    \horizline{5}{5}
    \vertline{6}{5}
    \horizline{6}{6}
}]
\null & \null & \null  & \null  & \null  & \null & \hskip 1em & \null \\
\null & 1 & 2 & 4 & 5 & 7 & \null \\
\null & \null & 3 & 6 & 8 & \null & \\
\null & \null & \null & 7 & 9 & \null &   \\
\null & \null & \null & \null & 10 & \null &   \\
\null & \null & \null & \null & \null & \null & \null \\
\null & \null & \null & \null & \null & \null &
\end{tikzcd}
\hfill
\begin{ytableau}
1 & 2 & 4 & 5 & 7 \\
\none & 3 & 6 & 8 \\
\none & \none & 7 & 9 \\
\none & \none & \none & 10
\end{ytableau}
\hfill}
\caption{Shifted Young tableaux}
\label{fig:shifted-young-tableaux}
\end{figure}
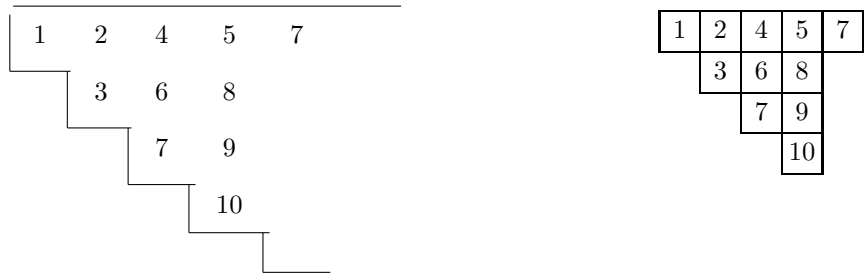

The only weight function on this $B$  compatible with differential
degree $r = 1$ is
\cite{Stan1990a}*{sec.~3}\cite{Fom1994a}*{Exam.~2.2.8}
\begin{equation}\label{eq:shifted-weight}
w(p) = \begin{cases}
  1 & \text{if $p$ is on the diagonal} \\
  2 & \text{otherwise}
  \end{cases}
\end{equation}
The Hasse diagram of the octant with the points labeled with their
weights is figure~\ref{fig:octant-weights}.
\begin{figure}
\begin{tikzcd}[column sep=0.5em,row sep=1em,every arrow/.append style={dash}]
\null \ar[dr, shorten <=2mm] && \null \ar[dl, shorten <=2mm] \ar[dr, shorten <=2mm] && \null \ar[dl, shorten <=2mm] \ar[dr, shorten <=2mm] \\
&\bullet_2 \ar[dr] && \bullet_2 \ar[dl] \ar[dr] && \bullet_1 \ar[dl] \\
&&\bullet_2 \ar[dr] && \bullet_2 \ar[dl] \ar[dr] \\
&&& \bullet_2 \ar[dr] && \bullet_1 \ar[dl] \\
&&&& \bullet_2 \ar[dr] \\
&&&&& \bullet_1
\end{tikzcd}
\caption[The Hasse diagram of the octant $\mbbO$]%
{The Hasse diagram of the octant $\mbbO$ with weights $w(p) =$ 1
  for diagonal elements and 2 for other elements}
\label{fig:octant-weights}
\end{figure}
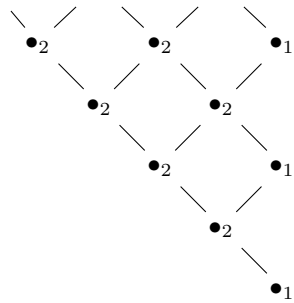
The dual graded graph with the edges labeled with their multiplicities
in $G_2$ is shown in figure~\ref{fig:strict-graph-weighted}.
\begin{figure}
\hbox to \linewidth{\hfill
\begin{tikzcd}[column sep=0.5em,row sep=1em,every arrow/.append style={dash}]
&\null \ar[dl, shorten <=2mm] \ar[dr, shorten <=2mm] & \null \ar[d, shorten <=1mm] & \null \ar[dl, shorten <=2mm] \ar[d, shorten <=1mm] & \null \ar[dl, shorten <=2mm] \ar[dr, shorten <=2mm] && \null \ar[dl, shorten <=2mm] \\
321 \ar[dr, "1"] && 42 \ar[dl, "2"] \ar[d, "2"] & 51 \ar[dl, "2"] \ar[dr, "1"] && 6 \ar[dl, "2"] \\
& 32 \ar[d, "2"] & 41 \ar[dl, "2"] \ar[dr, "1"] && 5 \ar[dl, "2"] \\
& 31 \ar[dl, "2"] \ar[dr, "1"] && 4 \ar[dl, "2"] \\
21 \ar[dr, "1"] && 3 \ar[dl, "2"] \\
& 2 \ar[dl, "2"] \\
1 \ar[d, "1"] \\
\emptyset
\end{tikzcd}
\hfill
\begin{tikzcd}[column sep=0.5em,row sep=1em,every arrow/.append style={dash}]
&\null \ar[dl, equal, shorten <=2mm] \ar[dr, shorten <=2mm] & \null \ar[d, equal, shorten <=1mm] & \null \ar[dl, equal, shorten <=2mm] \ar[d, equal, shorten <=1mm] & \null \ar[dl, equal, shorten <=2mm] \ar[dr, shorten <=2mm] && \null \ar[dl, equal, shorten <=2mm] \\
321 \ar[dr] && 42 \ar[dl, equal] \ar[d, equal] & 51 \ar[dl, equal] \ar[dr] && 6 \ar[dl, equal] \\
& 32 \ar[d, equal] & 41 \ar[dl, equal] \ar[dr] && 5 \ar[dl, equal] \\
& 31 \ar[dl, equal] \ar[dr] && 4 \ar[dl, equal] \\
21 \ar[dr] && 3 \ar[dl, equal] \\
& 2 \ar[dl, equal] \\
1 \ar[d] \\
\emptyset
\end{tikzcd}
\hfill}
\caption[The lattice of strict partitions $\mbbSY$]%
{The graph of strict partitions $\mbbSY$ (a) with edges labeled with
  their weights/multiplicities and (b) with edges shown multiply}
\label{fig:strict-graph-weighted}
\end{figure}
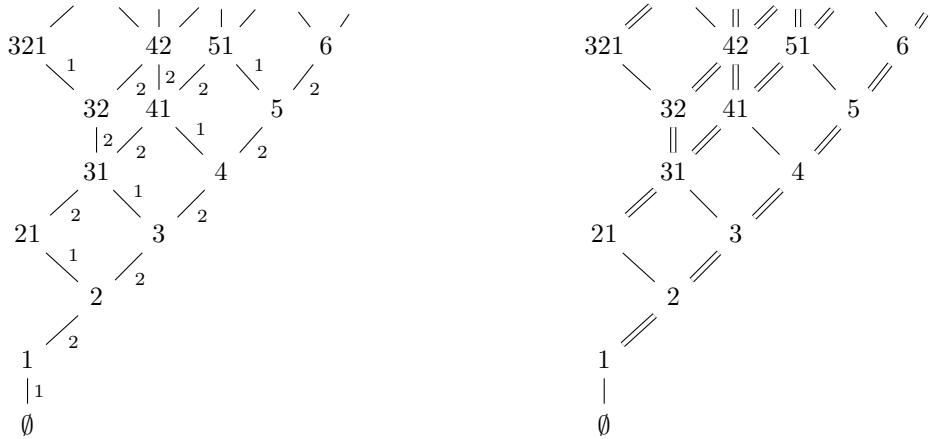

For colored shifted Young tableaux, the
the main-diagonal points have weight 1 and so their
values always have the same color.  The off-diagonal points
have weight 2 and so their values can have two colors.  In
\cite{Sag1979b} the values of point with color 2 are made members of a set of
distinguished values.  In other works dealing with shifted tableaux
the values with color 2 are considered
marked in some manner, usually as being \textit{circled}.
Here we show off-diagonal points with color 2 as
``circled'' with a superscript ``degree'' symbol:  $\cdot^\circ$
(resembling \cite{Haim1989a}*{sec.~2}).

\paragraph{Sagan's first algorithm for shifted tableaux}
\label{sec:sagan}
One insertion algorithm for shifted tableaux is Sagan's first algorithm
\cite{Sag1979b}.  
A thumbnail description is:
\begin{enumerate}
\item A number being inserted into the $P$ tableau is inserted
into the first row, with any (larger) number that it bumps being
inserted into the second row, etc. as for row insertion into unshifted
tableaux.
\item However, if a number is inserted into the main diagonal
cell of a row and bumps a number, that larger number begins
another sequence of row insertions, starting with the first row, but
now with the constraint that if a number would be inserted into a main
diagonal cell, it instead is returned to the first row to start a new
sequence of row insertions.
\end{enumerate}

This insertion algorithm can be summarized straightforwardly in the insertion
diagrams shown in figure~\ref{fig:psi-sagan1}.  They show
directly that the insertion algorithm is invertible and that
\ref{eq:weight} is satisfied for $r = 1$.
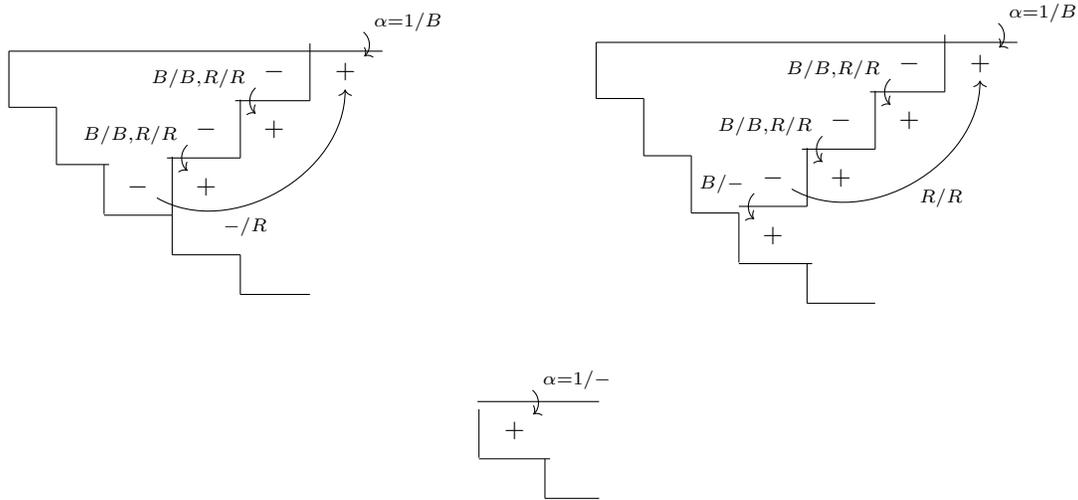
\begin{figure}
\hbox to \linewidth{\hfill
\begin{tikzcd}[cramped, sep=small, execute at end picture={
    \horizlinex{1}{2}{7}
    \vertline{2}{1}
    \horizline{2}{2}
    \vertline{3}{2}
    \horizline{3}{3}
    \vertline{4}{3}
    \horizline{4}{4}
    \vertline{5}{4}
    \horizline{5}{5}
    \vertline{6}{5}
    \horizline{6}{6}
    \vertline{2}{6}
    \horizline{2}{6}
    \vertline{3}{5}
    \horizline{3}{5}
    \vertline{4}{4}
}]
\null & \null & \null  & \null  & \null  & \null & \hskip 1em \ar[d, "\alpha=1/B" near start, bend left=50] & \null \\
\null & \null & \null  & \null  & \null & - \ar[d, bend right=50, "{B/B,R/R}"' near start] & + \\
\null & \null & \null  & \null  & - \ar[d, bend right=50, "{B/B,R/R}"' near start] & + &   \\
\null & \null & \null  & - \ar[uurrr, bend right=60, "{-/R}"' near start] & + & \null &   \\
\null & \null & \null & \null & \null & \null &   \\
\null & \null & \null & \null & \null & \null & \null \\
\null & \null & \null & \null & \null & \null &   \\
\null & \null & \null & \null & \null & \null &
\end{tikzcd}
\hfill
\begin{tikzcd}[cramped, sep=small, execute at end picture={
    \horizlinex{1}{2}{8}
    \vertline{2}{1}
    \horizline{2}{2}
    \vertline{3}{2}
    \horizline{3}{3}
    \vertline{4}{3}
    \horizline{4}{4}
    \vertline{5}{4}
    \horizline{5}{5}
    \vertline{6}{5}
    \horizline{6}{6}
    \vertline{2}{7}
    \horizline{2}{7}
    \vertline{3}{6}
    \horizline{3}{6}
    \vertline{4}{5}
    \horizline{4}{5}
}]
\null & \null & \null & \null  & \null  & \null  & \null & \hskip 1em \ar[d, "\alpha=1/B" near start, bend left=50] & \null \\
\null & \null & \null & \null  & \null  & \null & - \ar[d, bend right=50, "{B/B,R/R}"' near start] & + \\
\null & \null & \null & \null  & \null  & - \ar[d, bend right=50, "{B/B,R/R}"' near start] & + &   \\
\null & \null & \null & \null  & - \ar[d, bend right=50, "{B/-}"' near start] \ar[uurrr, bend right=60, "{R/R}"'] & + & \null &   \\
\null & \null & \null & \null & + & \null & \null &   \\
\null & \null & \null & \null & \null & \null & \null & \null \\
\null & \null & \null & \null & \null & \null & \null &   \\
\null & \null & \null & \null & \null & \null & \null &
\end{tikzcd}
\hfill}
\begin{tikzcd}[cramped, sep=small, execute at end picture={
    \horizlinex{1}{2}{3}
    \vertline{2}{1}
    \horizline{2}{2}
    \vertline{3}{2}
    \horizline{3}{3}
}]
\null & \hskip 1em \ar[d, "\alpha=1/-" near start, bend left=50] & \null & \null \\
\null & + & \null & \null \\
\null & \null & \null & \null \\
\null & \null & \null & \null
\end{tikzcd}
\caption[Generic insertion diagrams for Sagan's first algorithm]%
{Generic insertion diagrams for Sagan's first
  algorithm: (a) last part of $x = 1$, (b) last part of $x > 1$, and
  (c) $x = \emptyset$, which is a special case of (b)}
\label{fig:psi-sagan1}
\end{figure}
Note these points about the insertion diagram:
\begin{enumerate}
\item
  There are two distinct generic insertion diagrams, because the
  pattern of insertion and deletion points differs depending on
  whether the shape has a (last) row with length 1 or not.
  (The shapes that do not are $\emptyset$ and shapes with a last row
  with length $> 1$.)  The third insertion diagram shows how the
  special case $\emptyset$ works out.
\item
  The points on the main diagonal have weight 1, and so all edges to
  such points in $G_2$ have the same color.  Following
  \cite{Fom1995a}*{sec.~4.5}, we call the color of those edges
  ``black'', which we abbreviate ``$-$'' in the insertion diagrams.
\item
  The points not on the main diagonal have weight 2.  Each such point
  has two edges in $G_2$, which again following
  \cite{Fom1995a}*{sec.~4.5}, we call the color of those edges
  ``blue'' and ``red'', which we abbreviate ``$B$'' and ``$R$'' in the
  insertion diagrams.
\item
  From some of the deletion points, the blue and red edges lead to the
  same insertion point, so a single arrow is shown with the two
  ``before/after'' color specifications combined.
\item
  It is a straightforward visual exercise to verify that these
  insertion diagrams satisfy the constraints required to specify a
  valid $\psi$.
\item
  The insertion diagram shows the general pattern of bumps generated
  by an insertion:  The ``$\alpha = \cdot$'' edge shows that the
  element is inserted into the first row as a blue insertion.
  Successive bumps are blue insertions into successive rows until a
  main diagonal element is bumped, at which point the bumped element
  is inserted into the first row as a \textit{red} insertion.  Bumps
  caused by red insertions remain red insertions, and never are inserted as
  main diagonal elements.
\end{enumerate}

The growth for the example permutation $(1,2,5,4,3)$ and the resulting
tableaux
are shown in figure~\ref{fig:growth-sagan}.
\begin{figure}
\growthtableauxh{%
\begin{tikzcd}[ampersand replacement=\&,sep=small]
\emptyset \ar[rr]\ar[dd] \&\&
  1 \ar[rr]\ar[dd] \&\&
  2 \ar[rr]\ar[dd] \&\&
  3 \ar[rr]\ar[dd,"B"] \&\&
  31 \ar[rr]\ar[dd,"-"] \&\&
  41 \ar[dd,"R"] \\
\& \&\& \&\& X \&\& \&\& \& \\
\emptyset \ar[rr]\ar[dd] \&\&
  1 \ar[rr]\ar[dd] \&\&
  2 \ar[rr]\ar[dd] \&\&
  2 \ar[rr]\ar[dd] \&\&
  3 \ar[rr]\ar[dd,"B"] \&\&
  31 \ar[dd,"-"] \\
\& \&\& \&\& \&\& X \&\& \& \\
\emptyset \ar[rr]\ar[dd] \&\&
  1 \ar[rr]\ar[dd] \&\&
  2 \ar[rr]\ar[dd] \&\&
  2 \ar[rr]\ar[dd] \&\&
  2 \ar[rr]\ar[dd] \&\&
  3 \ar[dd,"B"] \\
\& \&\& \&\& \&\& \&\& X \& \\
\emptyset \ar[rr]\ar[dd] \&\&
  1 \ar[rr]\ar[dd] \&\&
  2 \ar[rr]\ar[dd,"B"] \&\&
  2 \ar[rr]\ar[dd,"B"] \&\&
  2 \ar[rr]\ar[dd,"B"] \&\&
  2 \ar[dd,"B"] \\
\& \&\& X \&\& \&\& \&\& \& \\
\emptyset \ar[rr]\ar[dd] \&\&
  1 \ar[rr]\ar[dd,"-"] \&\&
  1 \ar[rr]\ar[dd,"-"] \&\&
  1 \ar[rr]\ar[dd,"-"] \&\&
  1 \ar[rr]\ar[dd,"-"] \&\&
  1 \ar[dd,"-"] \\
\& X \&\& \&\& \&\& \&\& \& \\
\emptyset \ar[rr] \&\&
  \emptyset \ar[rr] \&\&
  \emptyset \ar[rr] \&\&
  \emptyset \ar[rr] \&\&
  \emptyset \ar[rr] \&\&
  \emptyset
\end{tikzcd}
}{%
\begin{ytableau}
1 & 2 & 3 & 5 \\
\none & 4
\end{ytableau}
}{%
\begin{ytableau}
1 & 2 & 3 & 5^\circ \\
\none & 4
\end{ytableau}
}
\caption[Growth for the permutation $(1,2,5,4,3)$ for Sagan's first algorithm]%
{Growth for the permutation $(1,2,5,4,3)$ for Sagan's first algorithm:
(a) growth diagram, (b) $P$ tableau, and (c) $Q$ tableau}
\label{fig:growth-sagan}
\end{figure}
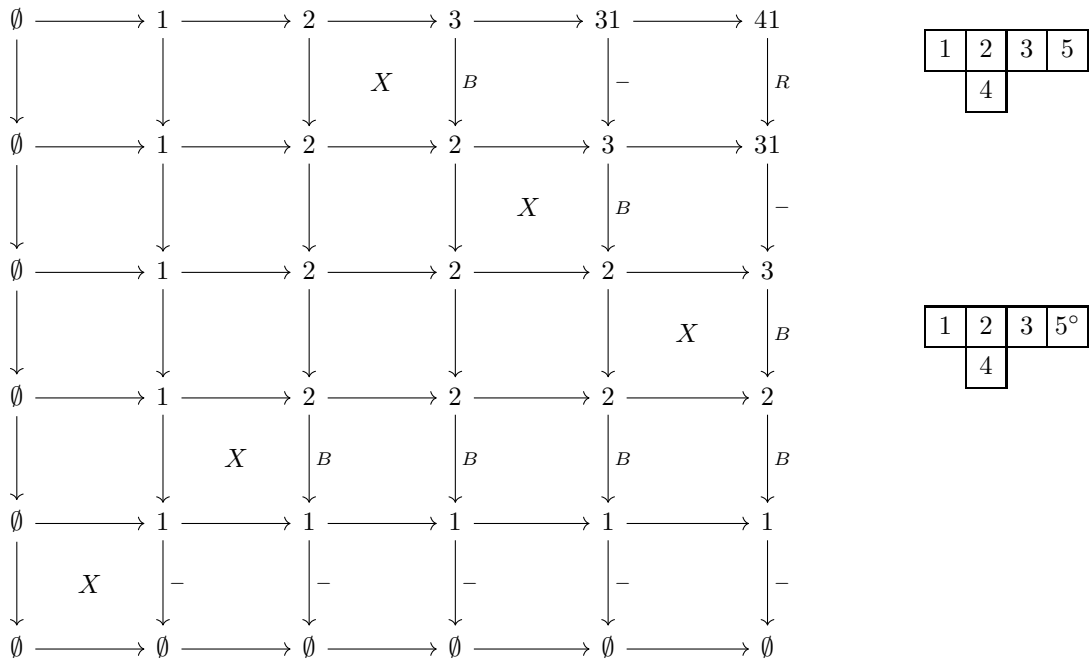

\paragraph{The second algorithm for shifted tableaux}
\label{sec:shifted}
A second insertion algorithm for shifted tableaux was discovered
independently\footnote{And nearly simultaneously:  A letter from Sagan
to Richard Stanley reporting his discovery crossed in the mail with a
letter from Stanley to Sagan reporting Worley's discovery.}
by Worley \cite{Wor1984a} and Sagan \cite{Sag1987a}.
A thumbnail description is:
\begin{enumerate}
\item A number being inserted into the $P$
is inserted
into the first row, with any (larger) number that it bumps being
inserted into the second row, etc. as for row insertion into unshifted
tableaux.
\item However, if a number is inserted into the main diagonal
cell of a row and bumps a number, that larger number begins
sequence of column insertions, starting with the first column
rightward of the diagonal cell, continuing rightward until an
insertion is done into an empty cell.
\end{enumerate}

This insertion algorithm can be summarized straightforwardly in the insertion
diagrams shown in figure~\ref{fig:psi-worley}.  The diagrams show
directly that the insertion algorithm is invertible and that
\ref{eq:weight} is satisfied for $r = 1$.
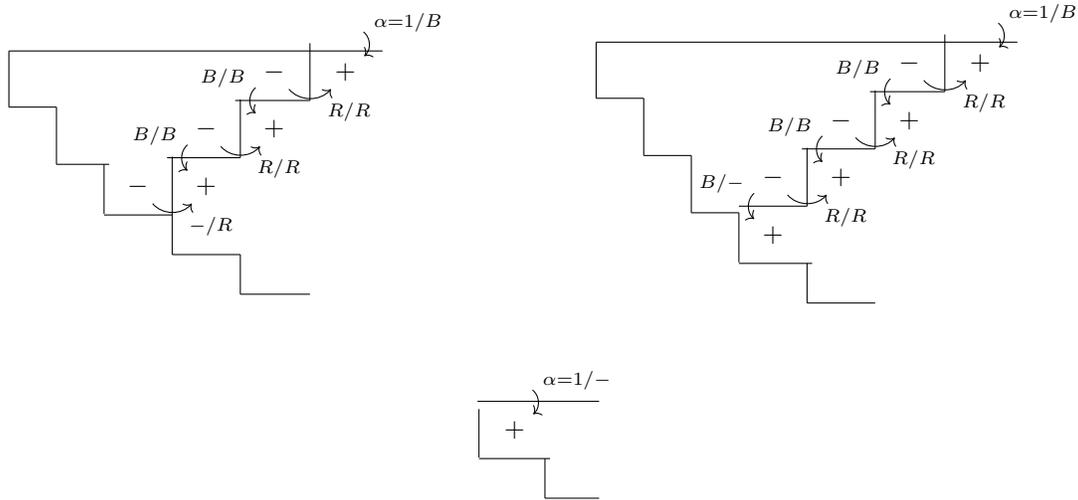
\begin{figure}
\hbox to \linewidth{\hfill
\begin{tikzcd}[cramped, sep=small, execute at end picture={
    \horizlinex{1}{2}{7}
    \vertline{2}{1}
    \horizline{2}{2}
    \vertline{3}{2}
    \horizline{3}{3}
    \vertline{4}{3}
    \horizline{4}{4}
    \vertline{5}{4}
    \horizline{5}{5}
    \vertline{6}{5}
    \horizline{6}{6}
    \vertline{2}{6}
    \horizline{2}{6}
    \vertline{3}{5}
    \horizline{3}{5}
    \vertline{4}{4}
}]
\null & \null & \null  & \null  & \null  & \null & \hskip 1em \ar[d, "\alpha=1/B" near start, bend left=50] & \null \\
\null & \null & \null  & \null  & \null & - \ar[d, bend right=50, "B/B"' near start] \ar[r, bend right=50, "R/R"' near end] & + \\
\null & \null & \null  & \null  & - \ar[d, bend right=50, "B/B"' near start] \ar[r, bend right=50, "R/R"' near end] & + &   \\
\null & \null & \null  & - \ar[r, bend right=50, "-/R"' near end] & + & \null &   \\
\null & \null & \null & \null & \null & \null &   \\
\null & \null & \null & \null & \null & \null & \null \\
\null & \null & \null & \null & \null & \null &   \\
\null & \null & \null & \null & \null & \null &
\end{tikzcd}
\hfill
\begin{tikzcd}[cramped, sep=small, execute at end picture={
    \horizlinex{1}{2}{8}
    \vertline{2}{1}
    \horizline{2}{2}
    \vertline{3}{2}
    \horizline{3}{3}
    \vertline{4}{3}
    \horizline{4}{4}
    \vertline{5}{4}
    \horizline{5}{5}
    \vertline{6}{5}
    \horizline{6}{6}
    \vertline{2}{7}
    \horizline{2}{7}
    \vertline{3}{6}
    \horizline{3}{6}
    \vertline{4}{5}
    \horizline{4}{5}
}]
\null & \null & \null & \null  & \null  & \null  & \null & \hskip 1em \ar[d, "\alpha=1/B" near start, bend left=50] & \null \\
\null & \null & \null & \null  & \null  & \null & - \ar[d, bend right=50, "B/B"' near start] \ar[r, bend right=50, "R/R"' near end] & + \\
\null & \null & \null & \null  & \null  & - \ar[d, bend right=50, "B/B"' near start] \ar[r, bend right=50, "R/R"' near end] & + &   \\
\null & \null & \null & \null  & - \ar[d, bend right=50, "{B/-}"' near start] \ar[r, bend right=50, "R/R"' near end] & + & \null &   \\
\null & \null & \null & \null & + & \null & \null &   \\
\null & \null & \null & \null & \null & \null & \null & \null \\
\null & \null & \null & \null & \null & \null & \null &   \\
\null & \null & \null & \null & \null & \null & \null &
\end{tikzcd}
\hfill}
\begin{tikzcd}[cramped, sep=small, execute at end picture={
    \horizlinex{1}{2}{3}
    \vertline{2}{1}
    \horizline{2}{2}
    \vertline{3}{2}
    \horizline{3}{3}
}]
\null & \hskip 1em \ar[d, "\alpha=1/-" near start, bend left=50] & \null & \null \\
\null & + & \null & \null \\
\null & \null & \null & \null \\
\null & \null & \null & \null
\end{tikzcd}
\caption[Generic insertion diagrams for the second shifted algorithm]%
{Generic insertion diagrams for the second shifted
  algorithm: (a) last part of $x = 1$, (b) last part of $x > 1$, and
  (c) $x = \emptyset$, which is a special case of (b)}
\label{fig:psi-worley}
\end{figure}

The growth for the example permutation $(1,2,5,4,3)$ and the resulting
tableaux
are shown in figure~\ref{fig:growth-worley}.
Note that the only difference between the growth diagram
figure~\ref{fig:growth-worley} and the growth diagram for that
permutation using the first algorithm
figure~\ref{fig:growth-sagan} is the northeasternmost node, which is
32 rather than 41.
\begin{figure}
\growthtableauxh{%
\begin{tikzcd}[ampersand replacement=\&,sep=small]
\emptyset \ar[rr]\ar[dd] \&\&
  1 \ar[rr]\ar[dd] \&\&
  2 \ar[rr]\ar[dd] \&\&
  3 \ar[rr]\ar[dd,"B"] \&\&
  31 \ar[rr]\ar[dd,"-"] \&\&
  32 \ar[dd,"R"] \\
\& \&\& \&\& X \&\& \&\& \& \\
\emptyset \ar[rr]\ar[dd] \&\&
  1 \ar[rr]\ar[dd] \&\&
  2 \ar[rr]\ar[dd] \&\&
  2 \ar[rr]\ar[dd] \&\&
  3 \ar[rr]\ar[dd,"B"] \&\&
  31 \ar[dd,"-"] \\
\& \&\& \&\& \&\& X \&\& \& \\
\emptyset \ar[rr]\ar[dd] \&\&
  1 \ar[rr]\ar[dd] \&\&
  2 \ar[rr]\ar[dd] \&\&
  2 \ar[rr]\ar[dd] \&\&
  2 \ar[rr]\ar[dd] \&\&
  3 \ar[dd,"B"] \\
\& \&\& \&\& \&\& \&\& X \& \\
\emptyset \ar[rr]\ar[dd] \&\&
  1 \ar[rr]\ar[dd] \&\&
  2 \ar[rr]\ar[dd,"B"] \&\&
  2 \ar[rr]\ar[dd,"B"] \&\&
  2 \ar[rr]\ar[dd,"B"] \&\&
  2 \ar[dd,"B"] \\
\& \&\& X \&\& \&\& \&\& \& \\
\emptyset \ar[rr]\ar[dd] \&\&
  1 \ar[rr]\ar[dd,"-"] \&\&
  1 \ar[rr]\ar[dd,"-"] \&\&
  1 \ar[rr]\ar[dd,"-"] \&\&
  1 \ar[rr]\ar[dd,"-"] \&\&
  1 \ar[dd,"-"] \\
\& X \&\& \&\& \&\& \&\& \& \\
\emptyset \ar[rr] \&\&
  \emptyset \ar[rr] \&\&
  \emptyset \ar[rr] \&\&
  \emptyset \ar[rr] \&\&
  \emptyset \ar[rr] \&\&
  \emptyset
\end{tikzcd}
}{%
\begin{ytableau}
1 & 2 & 3 \\
\none & 4 & 5
\end{ytableau}
}{%
\begin{ytableau}
1 & 2 & 3 \\
\none & 4 & 5^\circ
\end{ytableau}
}
\caption[Growth for the permutation $(1,2,5,4,3)$ for the second shifted
algorithm]%
{Growth for the permutation $(1,2,5,4,3)$ for the second shifted
algorithm:
(a) growth diagram, (b) $P$ tableau, and (c) $Q$ tableau}
\label{fig:growth-worley}
\end{figure}
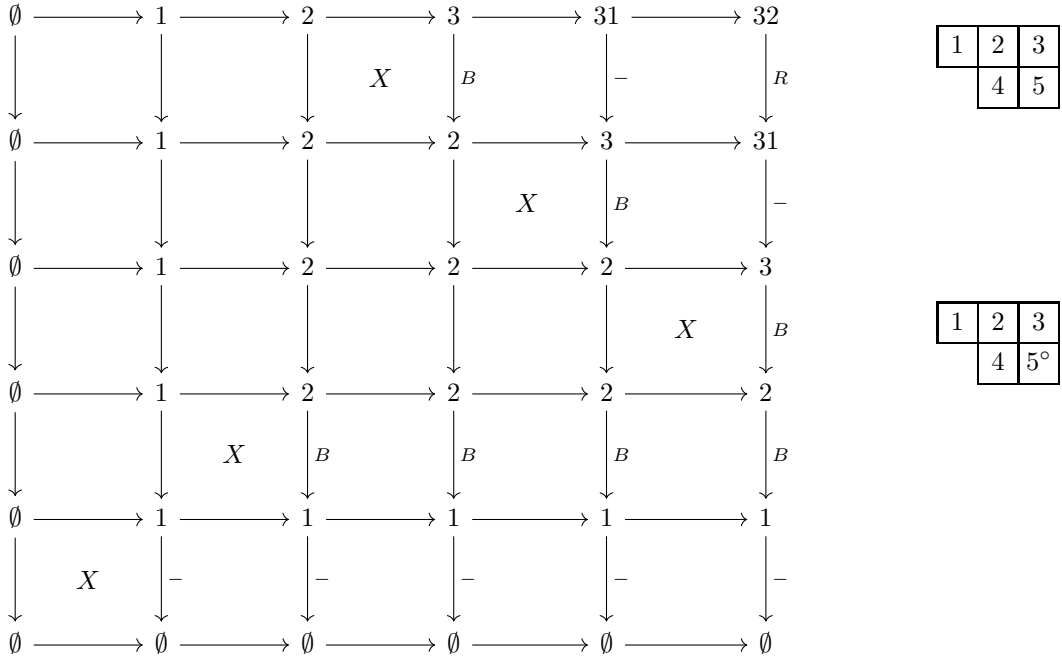

\section{Biweighted algorithms}
\label{sec:biweighted}

We can extend the weight function $w$ and the construction of the
ascending and descending graphs $G_1$ and $G_2$ by allowing multiple
edges in $G_1$ as well as $G_2$.
Thus, in ``biweighting'', we have two weight functions $w_1,
w_2: B \rightarrow \mbbP$, and, as before, a differential degree
$r \in \mbbP$.
Maintaining the properties of the $U$ and $D$ operators
in \cite{Fom1994a}*{sec.~2.2} so as to preserve the enumerative
properties of the machinery requires:
\begin{equation}\label{eq:biweight}
\sum_{\substack{y \in V \\ y \lessdot x}} w_1(x-y) w_2(x-y) + r =
\sum_{\substack{z \in V \\ x \lessdot z}} w_1(z-x) w_2(x-y)
\textrm{ for all } x \in V
\end{equation}
Thus, from any valid biweight functions, we can construct a valid
weight $w(p) = w_1(p) w_2(p)$
and any valid weight function can be factored into valid biweight functions.
Similarly, any valid weight function, paired with $w_1(p) = 1$ becomes
a valid biweight.

If we use the weight functions to specify the multiplicity of edges in
$G_1$ and $G_2$, then this factoring is more restrictive than we might
think at first.  For example, the weight function for shifted tableaux
\ref{eq:shifted-weight}, since it has values of 1 and 2, is difficult
to factor into integral $w_1$ and $w_2$ that are both nontrivial and
reasonably uniform over $B$.

The insertion diagrams for biweighted insertion algorithms are similar
to those for other algorithms.
The labels on the arrows now show
\textit{four} colors:  $a_1a_2/b_1b_2$ (listing the colors of the
the edges of a cell labeled as
shown in figure~\ref{fig:cell-convention}), with both the ``before'' and
``after'' parts giving the $G_1$ label before the $G_2$ label.  
See e.g.~figure~\ref{fig:psi-mixed}.

Extending coloring to $G_1$ allows us to extend both inversion
duality and transpose duality to include a systematic change to the
colors of the $G_1$ edges and the elements of $P$ in the
straightforward way.

\paragraph{Haiman's mixed insertion}
One instantiation of a biweight (that is not a weight) is Haiman's
``mixed insertion'' \cite{Haim1989a}*{sec.~3}, with $r=2$, $w_1(p)=2$,
and $w_2(p)=1$.  The elements of the generalized permutation have two
colors; $\alpha = 1$ is called uncircled and $\alpha = 2$ is called
circled.  The insertion process is \cite{Haim1989a}*{Def.~3.1}:
\begin{quote}
If $w_i$ is not circled, insert $w_i$ into the first row of $T_i$; if
it is circled,
into the first column. As each subsequent element $x$ of $T_i$ is bumped
by an insertion, insert $x$ into the row immediately below if it is
not circled, or into the column immediately to its right if it is circled.
Continue until an insertion takes place at the end of a row or
column, bumping no new element. As in ordinary Schensted insertion,
this process terminates because (by definition) each $x$ bumps
an element greater than itself. It is easy to see that the result of the
insertion process is again a tableau.
\end{quote}
The generic insertion diagram for mixed insertion is shown in
figure~\ref{fig:psi-mixed}.
All $G_2$ labels are the same and are denoted ``$-$''.
Comparing figure~\ref{fig:psi-mixed} with
figure~\ref{fig:psi-left-right} shows that they are inversion-duals of
each other.
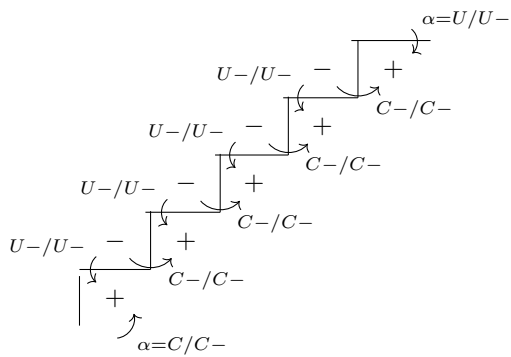
\begin{figure}
\begin{tikzcd}[cramped, sep=small, execute at end picture={
    \horizline{1}{6}
    \vertline{2}{5}
    \horizline{2}{5}
    \vertline{3}{4}
    \horizline{3}{4}
    \vertline{4}{3}
    \horizline{4}{3}
    \vertline{5}{2}
    \horizline{5}{2}
    \vertline{6}{1}
}]
\null & \null  & \null  & \null  & \null & \hskip 1em \ar[d, "\alpha=U/U-" near start, bend left=50] & \null \\
\null & \null  & \null  & \null & - \ar[d, "U-/U-"' near start, bend right=50] \ar[r, "C-/C-"' near end, bend right=50] & + \\
\null & \null  & \null  & - \ar[d, "U-/U-"' near start, bend right=50] \ar[r, "C-/C-"' near end, bend right=50] & + &   \\
\null & \null  & - \ar[d, "U-/U-"' near start, bend right=50] \ar[r, "C-/C-"' near end, bend right=50] & + &   &   \\
\null & - \ar[d, "U-/U-"' near start, bend right=50] \ar[r, "C-/C-"' near end, bend right=50] & + &   &   &   \\
\null & + &   &   &   &   \\
\null & \hskip 1em \ar[u, "\alpha=C/C-"', bend right=50] &   &   &   &
\end{tikzcd}
\caption{Generic insertion diagram for mixed insertion.}
\label{fig:psi-mixed}
\end{figure}

Because mixed insertion is the inversion-dual of left-right insertion,
and left-right insertion is transpose self-dual (with interchange of
``$C$'' and ``$U$''), mixed insertion is transpose self-dual in the
same way.

An example of mixed insertion as a growth diagram is
figure~\ref{fig:growth-mixed}.
The generalized permutation in the example is the inverse of the
generalized permutation in the left-right insertion in
figure~\ref{fig:growth-left-right}, so the
two growth diagrams are transposes of each other.
Note that since $w_2(p) = 1$, the labels on $G_2$ edges are not shown.
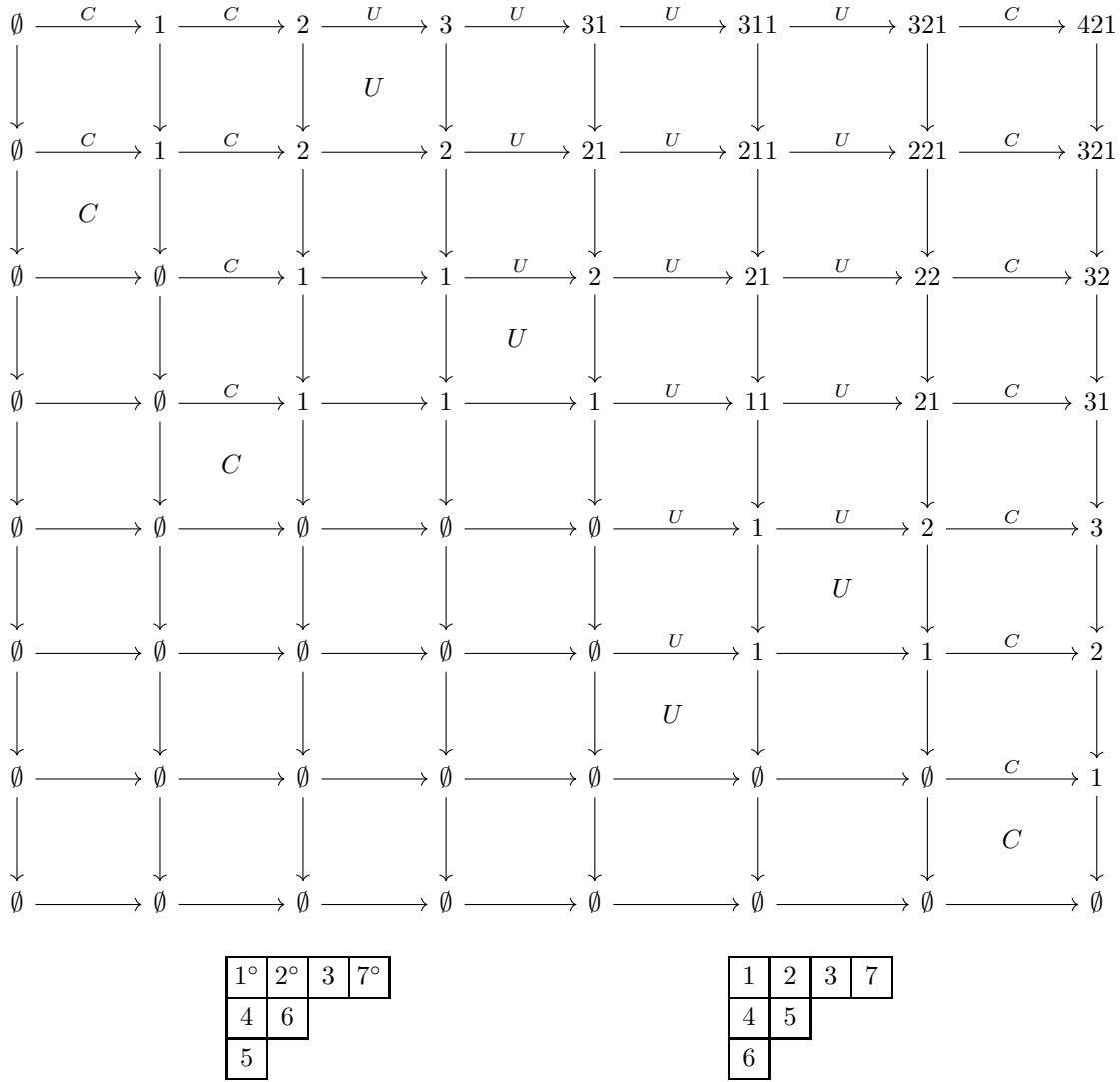
\begin{figure}
\growthtableauxv{%
\begin{tikzcd}[ampersand replacement=\&,sep=small]
\emptyset \ar[rr,"C"]\ar[dd] \&\&
  1 \ar[rr,"C"]\ar[dd] \&\&
  2 \ar[rr,"U"]\ar[dd] \&\&
  3 \ar[rr,"U"]\ar[dd] \&\&
  31 \ar[rr,"U"]\ar[dd] \&\&
  311 \ar[rr,"U"]\ar[dd] \&\&
  321 \ar[rr,"C"]\ar[dd] \&\&
  421 \ar[dd] \\
\&  \&\&  \&\& U \&\&  \&\&  \&\&  \&\&  \& \\
\emptyset \ar[rr,"C"]\ar[dd] \&\&
  1 \ar[rr,"C"]\ar[dd] \&\&
  2 \ar[rr]\ar[dd] \&\&
  2 \ar[rr,"U"]\ar[dd] \&\&
  21 \ar[rr,"U"]\ar[dd] \&\&
  211 \ar[rr,"U"]\ar[dd] \&\&
  221 \ar[rr,"C"]\ar[dd] \&\&
  321 \ar[dd] \\
\& C \&\&  \&\&  \&\&  \&\&  \&\&  \&\&  \& \\
\emptyset \ar[rr]\ar[dd] \&\&
  \emptyset \ar[rr,"C"]\ar[dd] \&\&
  1 \ar[rr]\ar[dd] \&\&
  1 \ar[rr,"U"]\ar[dd] \&\&
  2 \ar[rr,"U"]\ar[dd] \&\&
  21 \ar[rr,"U"]\ar[dd] \&\&
  22 \ar[rr,"C"]\ar[dd] \&\&
  32 \ar[dd] \\
\&  \&\&  \&\&  \&\& U \&\&  \&\&  \&\&  \& \\
\emptyset \ar[rr]\ar[dd] \&\&
  \emptyset \ar[rr,"C"]\ar[dd] \&\&
  1 \ar[rr]\ar[dd] \&\&
  1 \ar[rr]\ar[dd] \&\&
  1 \ar[rr,"U"]\ar[dd] \&\&
  11 \ar[rr,"U"]\ar[dd] \&\&
  21 \ar[rr,"C"]\ar[dd] \&\&
  31 \ar[dd] \\
\&  \&\& C \&\&  \&\&  \&\&  \&\&  \&\&  \& \\
\emptyset \ar[rr]\ar[dd] \&\&
  \emptyset \ar[rr]\ar[dd] \&\&
  \emptyset \ar[rr]\ar[dd] \&\&
  \emptyset \ar[rr]\ar[dd] \&\&
  \emptyset \ar[rr,"U"]\ar[dd] \&\&
  1 \ar[rr,"U"]\ar[dd] \&\&
  2 \ar[rr,"C"]\ar[dd] \&\&
  3 \ar[dd] \\
\&  \&\&  \&\&  \&\&  \&\&  \&\& U \&\&  \& \\
\emptyset \ar[rr]\ar[dd] \&\&
  \emptyset \ar[rr]\ar[dd] \&\&
  \emptyset \ar[rr]\ar[dd] \&\&
  \emptyset \ar[rr]\ar[dd] \&\&
  \emptyset \ar[rr,"U"]\ar[dd] \&\&
  1 \ar[rr]\ar[dd] \&\&
  1 \ar[rr,"C"]\ar[dd] \&\&
  2 \ar[dd] \\
\&  \&\&  \&\&  \&\&  \&\& U \&\&  \&\&  \& \\
\emptyset \ar[rr]\ar[dd] \&\&
  \emptyset \ar[rr]\ar[dd] \&\&
  \emptyset \ar[rr]\ar[dd] \&\&
  \emptyset \ar[rr]\ar[dd] \&\&
  \emptyset \ar[rr]\ar[dd] \&\&
  \emptyset \ar[rr]\ar[dd] \&\&
  \emptyset \ar[rr,"C"]\ar[dd] \&\&
  1 \ar[dd] \\
\&  \&\&  \&\&  \&\&  \&\&  \&\&  \&\& C \& \\
\emptyset \ar[rr] \&\&
  \emptyset \ar[rr] \&\&
  \emptyset \ar[rr] \&\&
  \emptyset \ar[rr] \&\&
  \emptyset \ar[rr] \&\&
  \emptyset \ar[rr] \&\&
  \emptyset \ar[rr] \&\&
  \emptyset 
\end{tikzcd}
}{%
\begin{ytableau}
1^\circ & 2^\circ & 3 & 7^\circ \\
4 & 6 \\
5
\end{ytableau}
}{%
\begin{ytableau}
1 & 2 & 3 & 7 \\
4 & 5 \\
6
\end{ytableau}
}
\caption[Growth for mixed insertion]%
{Growth for the permutation
$(7^\circ, 5, 6, 2^\circ, 4, 1^\circ, 3)$ for mixed insertion:
(a) growth diagram, (b) $P$ tableau, and (c) $Q$ tableau}
\label{fig:growth-mixed}
\end{figure}

\paragraph{Double-circle insertion}
\label{sec:double-circle-insertion}
Haiman remarks \cite{Haim1989a}*{sec.~4}
\begin{quote}
  This makes it possible to
define left mixed insertion, and by analogy with the definition of left-right
insertion below, to define \textit{left-right mixed insertion} (on words with two
types of circles!). In such a way it is possible to get a self-dual
[i.e.~inversion self-dual] insertion
theory that generalizes Schensted, left-right, and mixed insertion. Having
no application for it in this article, we do not develop this topic here. The
interested reader is invited to work out the details for him- or
herself.
\end{quote}
Here we develop this theory, not so much because we have
an application for it but because using insertion diagrams, it is
straightforward to decide how to define the algorithm, and using the
associated software (see sec.~\ref{sec:software}) it is
straightforward to execute the algorithm on permutations.  Thus, we
get another algorithm example with little extra work.

In order to have inversion self-duality, we require that transposition of
the growth diagram creates a valid growth diagram, possibly with a
systematic change of the labels on edges.  This requires that 
both the $G_1$ and $G_2$ graphs have two labels on edges.  Thus
$w_1(p) = 2$ and $w_2(p) = 2$, implying $r = 4$.  So in regard to the
biweighting, the theory is a ``quadrupling'' of the Robinson-Schensted
theory.
There are four non-zero $\alpha$ values
\begin{enumerate}
\item represented by $UU$ in cells and $\cdot$ in generalized permutations
\item represented by $CU$ in cells and $\cdot^\circ$ in generalized permutations
\item represented by $UC$ in cells and $\cdot^\bullet$ in generalized permutations
\item represented by $CC$ in cells and $\cdot^{\circ\bullet}$ in generalized permutations
\end{enumerate}

The labels on arrows in the generic insertion diagram will be
$a_1a_2/b_1b_2$, where each of the four elements can be ``$U$'' or
``$C$''.

In order to have the new algorithm have the transpose self-duality of
left-right and mixed
insertion, that the joint transposition of the $P$ and $Q$ tableaux can
be accomplished by changing the circling of the generalized permutation,
there must be the same number of arrows ``going northeast'' as ``going
southwest'' at each point in the diagram.  In this case, each deletion
point must have two
arrows to the insertion point in the next row and two to the insertion
point in the next column.

Similarly, each insertion point has four incoming arrows.
For most insertion points, two are incoming from deletion points
``from the northeast'' and
two from deletion points ``from the southeast''.
But the northeast-most insertion point has two incoming $\alpha=\cdot$
arrows and similarly for the southwest-most insertion point.
For each insertion point, one of its incoming arrows is labeled with
each of the ``after states'', $UU$, $UC$, $CU$, and $CC$.

For simplicity, and in parallel with all of the preceding unshifted
insertion algorithms, each arrow has the same ``before state'' and
``after state''.  E.g.~we have labels like $UC/UC$ but not $UC/CU$.
The result is that the arrows can be grouped into four chains, one for
each circling combination, labeled with that combination as both
``before state'' and ``after state''.  Each chain begins at either the
southwest or northeast extreme and passes northeastward or
southwestward (resp.) through all the insertion and deletion points.

All of this reduces the number of choices to be made to $2^4$:  For each of the
chains, $UU$, $UC$, $CU$, and $CC$, its chain runs either
southwestward or northeastward.  By symmetry, one choice is arbitrary,
so we choose the $UU$ chain to run southwestward.
If we choose $UC$ or $CU$ as the other chain runing southwestward,
then the other two chains run northeastward, and either the first or
second (resp.) circling does not affect the insertion process itself,
and just passes through to label the elements in one of the result
tableaux.

Thus, the only interesting choice is to have the $UU$ and $CC$ chains
run southwestward and the $UC$ and $CC$ chains run northeastward.
This generic insertion diagram is shown in figure~\ref{fig:psi-double}.
A typical growth diagram is shown in figure~\ref{fig:growth-double}.
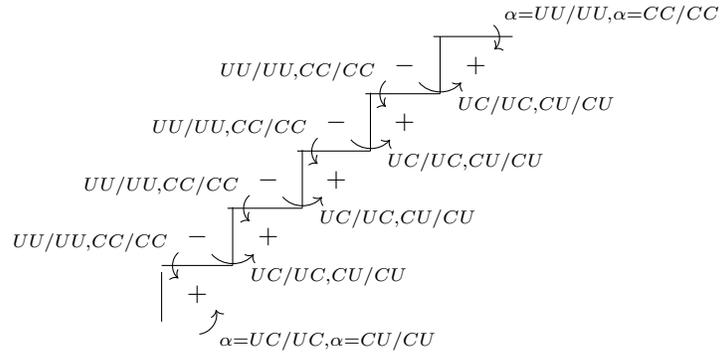
\begin{figure}
\begin{tikzcd}[cramped, sep=small, execute at end picture={
    \horizline{1}{6}
    \vertline{2}{5}
    \horizline{2}{5}
    \vertline{3}{4}
    \horizline{3}{4}
    \vertline{4}{3}
    \horizline{4}{3}
    \vertline{5}{2}
    \horizline{5}{2}
    \vertline{6}{1}
}]
\null & \null  & \null  & \null  & \null & \hskip 1em \ar[d, "{\alpha=UU/UU,\alpha=CC/CC}" near start, bend left=50] & \null \\
\null & \null  & \null  & \null & - \ar[d, "{UU/UU,CC/CC}"' near start, bend right=50] \ar[r, "{UC/UC,CU/CU}"' near end, bend right=50] & + \\
\null & \null  & \null  & - \ar[d, "{UU/UU,CC/CC}"' near start, bend right=50] \ar[r, "{UC/UC,CU/CU}"' near end, bend right=50] & + &   \\
\null & \null  & - \ar[d, "{UU/UU,CC/CC}"' near start, bend right=50] \ar[r, "{UC/UC,CU/CU}"' near end, bend right=50] & + &   &   \\
\null & - \ar[d, "{UU/UU,CC/CC}"' near start, bend right=50] \ar[r, "{UC/UC,CU/CU}"' near end, bend right=50] & + &   &   &   \\
\null & + &   &   &   &   \\
\null & \hskip 1em \ar[u, "{\alpha=UC/UC,\alpha=CU/CU}"', bend right=50] &   &   &   &
\end{tikzcd}
\caption{Generic insertion diagram for double-circle insertion}
\label{fig:psi-double}
\end{figure}
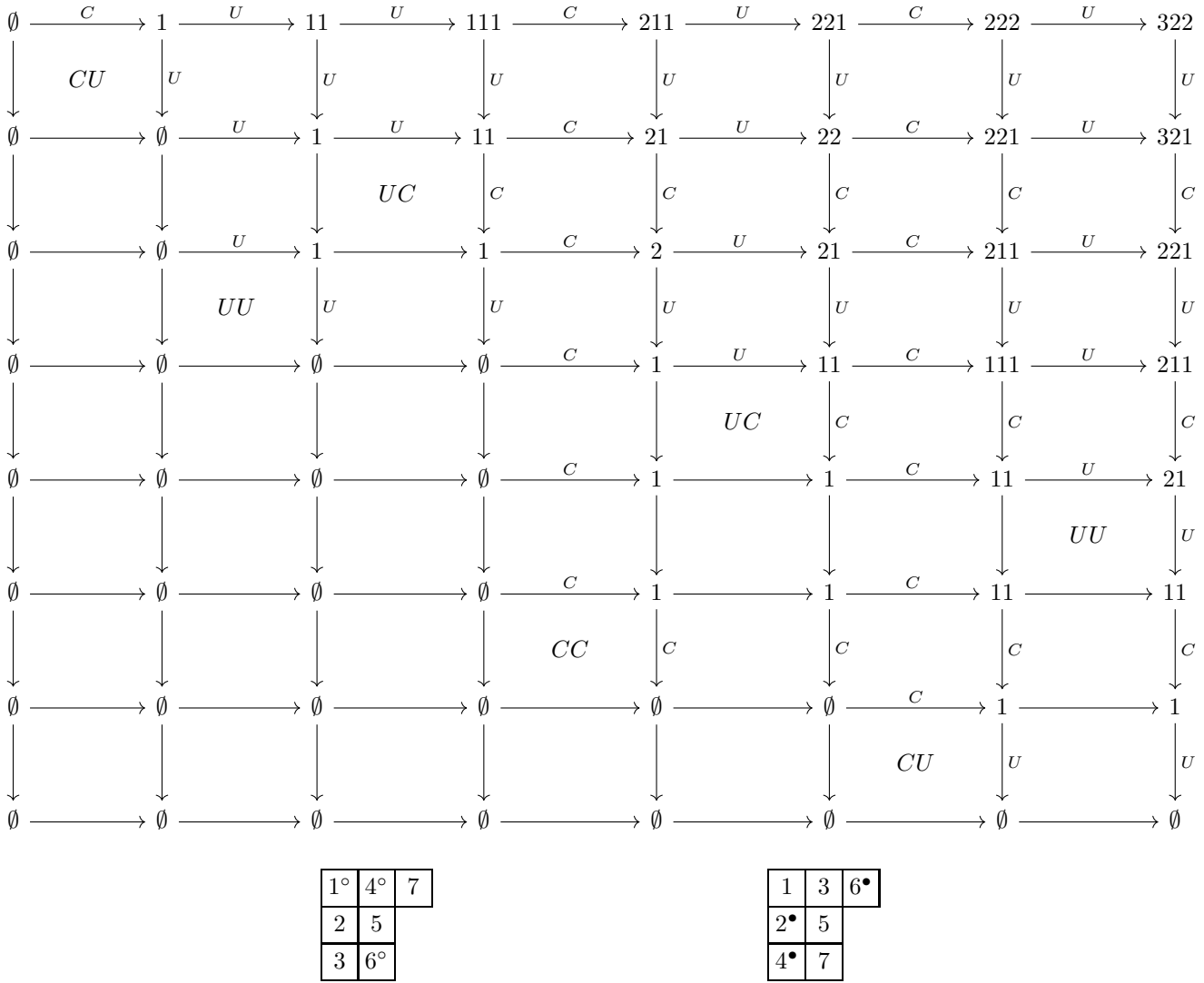
\begin{figure}
\growthtableauxv{
\begin{tikzcd}[ampersand replacement=\&,sep=small]
\emptyset \ar[rr,"C"]\ar[dd] \&\&
  1 \ar[rr,"U"]\ar[dd,"U"] \&\&
  11 \ar[rr,"U"]\ar[dd,"U"] \&\&
  111 \ar[rr,"C"]\ar[dd,"U"] \&\&
  211 \ar[rr,"U"]\ar[dd,"U"] \&\&
  221 \ar[rr,"C"]\ar[dd,"U"] \&\&
  222 \ar[rr,"U"]\ar[dd,"U"] \&\&
  322 \ar[dd,"U"] \\
\& CU \&\&  \&\&  \&\&  \&\&  \&\&  \&\&  \& \\
\emptyset \ar[rr]\ar[dd] \&\&
  \emptyset \ar[rr,"U"]\ar[dd] \&\&
  1 \ar[rr,"U"]\ar[dd] \&\&
  11 \ar[rr,"C"]\ar[dd,"C"] \&\&
  21 \ar[rr,"U"]\ar[dd,"C"] \&\&
  22 \ar[rr,"C"]\ar[dd,"C"] \&\&
  221 \ar[rr,"U"]\ar[dd,"C"] \&\&
  321 \ar[dd,"C"] \\
\&  \&\&  \&\& UC \&\&  \&\&  \&\&  \&\&  \& \\
\emptyset \ar[rr]\ar[dd] \&\&
  \emptyset \ar[rr,"U"]\ar[dd] \&\&
  1 \ar[rr]\ar[dd,"U"] \&\&
  1 \ar[rr,"C"]\ar[dd,"U"] \&\&
  2 \ar[rr,"U"]\ar[dd,"U"] \&\&
  21 \ar[rr,"C"]\ar[dd,"U"] \&\&
  211 \ar[rr,"U"]\ar[dd,"U"] \&\&
  221 \ar[dd,"U"] \\
\&  \&\& UU \&\&  \&\&  \&\&  \&\&  \&\&  \& \\
\emptyset \ar[rr]\ar[dd] \&\&
  \emptyset \ar[rr]\ar[dd] \&\&
  \emptyset \ar[rr]\ar[dd] \&\&
  \emptyset \ar[rr,"C"]\ar[dd] \&\&
  1 \ar[rr,"U"]\ar[dd] \&\&
  11 \ar[rr,"C"]\ar[dd,"C"] \&\&
  111 \ar[rr,"U"]\ar[dd,"C"] \&\&
  211 \ar[dd,"C"] \\
\&  \&\&  \&\&  \&\&  \&\& UC \&\&  \&\&  \& \\
\emptyset \ar[rr]\ar[dd] \&\&
  \emptyset \ar[rr]\ar[dd] \&\&
  \emptyset \ar[rr]\ar[dd] \&\&
  \emptyset \ar[rr,"C"]\ar[dd] \&\&
  1 \ar[rr]\ar[dd] \&\&
  1 \ar[rr,"C"]\ar[dd] \&\&
  11 \ar[rr,"U"]\ar[dd] \&\&
  21 \ar[dd,"U"] \\
\&  \&\&  \&\&  \&\&  \&\&  \&\&  \&\& UU \& \\
\emptyset \ar[rr]\ar[dd] \&\&
  \emptyset \ar[rr]\ar[dd] \&\&
  \emptyset \ar[rr]\ar[dd] \&\&
  \emptyset \ar[rr,"C"]\ar[dd] \&\&
  1 \ar[rr]\ar[dd,"C"] \&\&
  1 \ar[rr,"C"]\ar[dd,"C"] \&\&
  11 \ar[rr]\ar[dd,"C"] \&\&
  11 \ar[dd,"C"] \\
\&  \&\&  \&\&  \&\& CC \&\&  \&\&  \&\&  \& \\
\emptyset \ar[rr]\ar[dd] \&\&
  \emptyset \ar[rr]\ar[dd] \&\&
  \emptyset \ar[rr]\ar[dd] \&\&
  \emptyset \ar[rr]\ar[dd] \&\&
  \emptyset \ar[rr]\ar[dd] \&\&
  \emptyset \ar[rr,"C"]\ar[dd] \&\&
  1 \ar[rr]\ar[dd,"U"] \&\&
  1 \ar[dd,"U"] \\
\&  \&\&  \&\&  \&\&  \&\&  \&\& CU \&\&  \& \\
\emptyset \ar[rr] \&\&
  \emptyset \ar[rr] \&\&
  \emptyset \ar[rr] \&\&
  \emptyset \ar[rr] \&\&
  \emptyset \ar[rr] \&\&
  \emptyset \ar[rr] \&\&
  \emptyset \ar[rr] \&\&
  \emptyset 
\end{tikzcd}
}{
\begin{ytableau}
1^\circ & 4^\circ & 7 \\
2 & 5 \\
3 & 6^\circ
\end{ytableau}
}{
\begin{ytableau}
1 & 3 & 6^\bullet \\
2^\bullet & 5 \\
4^\bullet & 7
\end{ytableau}
}
\caption[Growth for double-circle insertion]%
{Growth for double-circle for the permutation
$(6^{\circ}, 4^{\circ\bullet}, 7, 5^{\bullet}, 2, 3^{\bullet}, 1^{\circ})$}
\label{fig:growth-double}
\end{figure}

\paragraph{Haiman's shifted mixed insertion}

Haiman defines shifted mixed insertion \cite{Haim1989a}*{sec.~3 and Def.~6.7}
\begin{quotation}
We now describe a
shifted version of mixed insertion that has a symmetry relationship
(Proposition 6.8) to unshifted mixed insertion parallel to the
relationship between Worley-Sagan insertion and left-right insertion
given by Proposition 6.2.  [\ldots]

Let $w = w_1 \cdots w_n$, be a
word without circles. Construct a sequence $T_0, T_1, \ldots, T_n = T$ of shifted
tableaux with circles: [\ldots]

Insert $w_i$ into the first row, and insert bumped letters as in mixed
insertion: circled letters, bumped, insert into the next column;
uncircled letters, into the next row, with one exception. The
exception is that an uncircled letter, bumped from a diagonal cell,
acquires a circle and inserts into the next column.
\end{quotation}

The generic insertion diagrams are shown in
figure~\ref{fig:psi-shifted-mixed}. Note that all the nontrivial
labels are for $G_1$.
\begin{figure}
\hbox to \linewidth{\hfill
\begin{tikzcd}[cramped, sep=small, execute at end picture={
    \horizlinex{1}{2}{7}
    \vertline{2}{1}
    \horizline{2}{2}
    \vertline{3}{2}
    \horizline{3}{3}
    \vertline{4}{3}
    \horizline{4}{4}
    \vertline{5}{4}
    \horizline{5}{5}
    \vertline{6}{5}
    \horizline{6}{6}
    \vertline{2}{6}
    \horizline{2}{6}
    \vertline{3}{5}
    \horizline{3}{5}
    \vertline{4}{4}
}]
\null & \null & \null  & \null  & \null  & \null & \hskip 1em \ar[d, "\alpha=1/U-" near start, bend left=50] & \null \\
\null & \null & \null  & \null  & \null & - \ar[d, bend right=50, "U-/U-"' near start] \ar[r, "C-/C-"' near end, bend right=50] & + \\
\null & \null & \null  & \null  & - \ar[d, bend right=50, "U-/U-"' near start] \ar[r, "C-/C-"' near end, bend right=50] & + &  \\
\null & \null & \null  & - \ar[r, "{--/C-}"' near end, bend right=50] & + & \null &   \\
\null & \null & \null & \null & \null & \null \\
\null & \null & \null & \null & \null & \null & \null \\
\null & \null & \null & \null & \null & \null &   \\
\null & \null & \null & \null & \null & \null &
\end{tikzcd}
\hfill
\begin{tikzcd}[cramped, sep=small, execute at end picture={
    \horizlinex{1}{2}{8}
    \vertline{2}{1}
    \horizline{2}{2}
    \vertline{3}{2}
    \horizline{3}{3}
    \vertline{4}{3}
    \horizline{4}{4}
    \vertline{5}{4}
    \horizline{5}{5}
    \vertline{6}{5}
    \horizline{6}{6}
    \vertline{2}{7}
    \horizline{2}{7}
    \vertline{3}{6}
    \horizline{3}{6}
    \vertline{4}{5}
    \horizline{4}{5}
}]
\null & \null & \null & \null  & \null  & \null  & \null & \hskip 1em \ar[d, "\alpha=1/U-" near start, bend left=50] & \null \\
\null & \null & \null & \null  & \null  & \null & - \ar[d, bend right=50, "U-/U-"' near start] \ar[r, "C-/C-"' near end, bend right=50] & + \\
\null & \null & \null & \null  & \null  & - \ar[d, bend right=50, "U-/U-"' near start] \ar[r, "C-/C-"' near end, bend right=50] & + &   \\
\null & \null & \null & \null  & - \ar[d, bend right=50, "{U-/--}"' near start] \ar[r, "C-/C-"' near end, bend right=50] & + & \null &   \\
\null & \null & \null & \null & + & \null & \null &   \\
\null & \null & \null & \null & \null & \null & \null & \null \\
\null & \null & \null & \null & \null & \null & \null &   \\
\null & \null & \null & \null & \null & \null & \null &
\end{tikzcd}
\hfill}
\begin{tikzcd}[cramped, sep=small, execute at end picture={
    \horizlinex{1}{2}{3}
    \vertline{2}{1}
    \horizline{2}{2}
    \vertline{3}{2}
    \horizline{3}{3}
}]
\null & \hskip 1em \ar[d, "\alpha=1/--" near start, bend left=50] & \null & \null \\
\null & + & \null & \null \\
\null & \null & \null & \null \\
\null & \null & \null & \null
\end{tikzcd}
\caption{Generic insertion diagrams for shifted mixed insertion}
\label{fig:psi-shifted-mixed}
\end{figure}
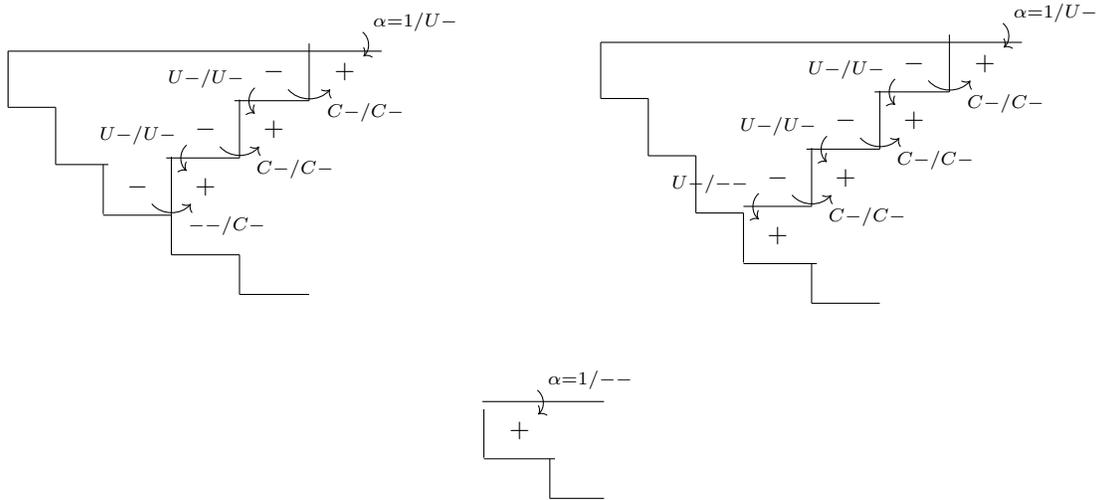

An example of the algorithm is shown in
figure~\ref{fig:growth-shifted-mixed}.  Since the permutation
$(1,2,5,4,3)$ is its own inverse, this growth is inversion-dual to the shifted
growth in figure~\ref{fig:growth-worley}.  That shifted mixed
insertion is inversion-dual to the second shifted insertion can be
seen directly
by comparing the insertion diagrams, figures~\ref{fig:psi-worley} and
\ref{fig:psi-shifted-mixed}.
\begin{figure}
\growthtableauxh{%
\begin{tikzcd}[ampersand replacement=\&,sep=small]
\emptyset \ar[rr,"-"]\ar[dd] \&\&
  1 \ar[rr,"U"]\ar[dd] \&\&
  2 \ar[rr,"U"]\ar[dd] \&\&
  3 \ar[rr,"-"]\ar[dd] \&\&
  31 \ar[rr,"C"]\ar[dd] \&\&
  32 \ar[dd] \\
\&  \&\&  \&\& X \&\&  \&\&  \& \\
\emptyset \ar[rr,"-"]\ar[dd] \&\&
  1 \ar[rr,"U"]\ar[dd] \&\&
  2 \ar[rr]\ar[dd] \&\&
  2 \ar[rr,"U"]\ar[dd] \&\&
  3 \ar[rr,"-"]\ar[dd] \&\&
  31 \ar[dd] \\
\&  \&\&  \&\&  \&\& X \&\&  \& \\
\emptyset \ar[rr,"-"]\ar[dd] \&\&
  1 \ar[rr,"U"]\ar[dd] \&\&
  2 \ar[rr]\ar[dd] \&\&
  2 \ar[rr]\ar[dd] \&\&
  2 \ar[rr,"U"]\ar[dd] \&\&
  3 \ar[dd] \\
\&  \&\&  \&\&  \&\&  \&\& X \& \\
\emptyset \ar[rr,"-"]\ar[dd] \&\&
  1 \ar[rr,"U"]\ar[dd] \&\&
  2 \ar[rr]\ar[dd] \&\&
  2 \ar[rr]\ar[dd] \&\&
  2 \ar[rr]\ar[dd] \&\&
  2 \ar[dd] \\
\&  \&\& X \&\&  \&\&  \&\&  \& \\
\emptyset \ar[rr,"-"]\ar[dd] \&\&
  1 \ar[rr]\ar[dd] \&\&
  1 \ar[rr]\ar[dd] \&\&
  1 \ar[rr]\ar[dd] \&\&
  1 \ar[rr]\ar[dd] \&\&
  1 \ar[dd] \\
\& X \&\&  \&\&  \&\&  \&\&  \& \\
\emptyset \ar[rr] \&\&
  \emptyset \ar[rr] \&\&
  \emptyset \ar[rr] \&\&
  \emptyset \ar[rr] \&\&
  \emptyset \ar[rr] \&\&
  \emptyset 
\end{tikzcd}
}{%
\begin{ytableau}
1 & 2 & 3 \\
\none & 4 & 5^\circ
\end{ytableau}
}{%
\begin{ytableau}
1 & 2 & 3 \\
\none & 4 & 5
\end{ytableau}
}
\caption[Growth for the permutation $(1,2,5,4,3)$ for the shifted mixed 
insertion]%
{Growth for the permutation $(1,2,5,4,3)$ for the shifted mixed 
insertion:
(a) growth diagram, (b) $P$ tableau, and (c) $Q$ tableau.
This is inversion-dual to \protect\ref{fig:growth-worley} because the
permutation is an involution.}
\label{fig:growth-shifted-mixed}
\end{figure}
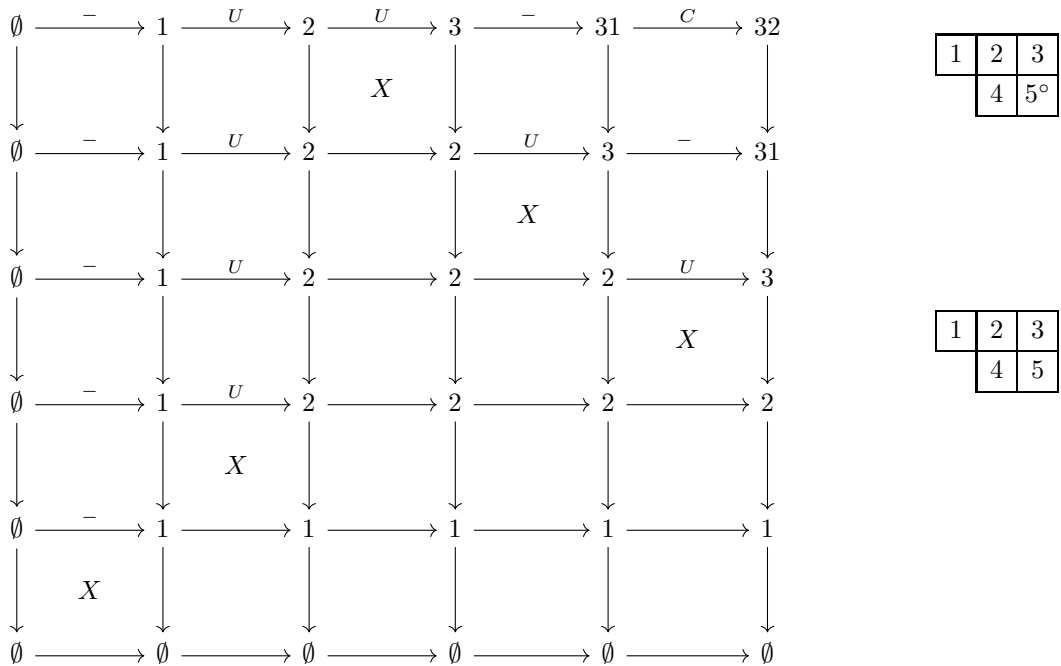

\paragraph{McLarnan's shifted column insertion}
\label{sec:shifted-column}
\cite{McLar1986}*{sec.~3, p.~49 et seq.} defines a shifted insertion algorithm with
an interesting set of properties.  The insertion diagrams for the
algorithm are figure~\ref{fig:psi-shifted-column}.
\begin{figure}
\hbox to \linewidth{\hfill
\begin{tikzcd}[cramped, sep=small, execute at end picture={
    \horizlinex{1}{2}{7}
    \vertline{2}{1}
    \horizline{2}{2}
    \vertline{3}{2}
    \horizline{3}{3}
    \vertline{4}{3}
    \horizline{4}{4}
    \vertline{5}{4}
    \horizline{5}{5}
    \vertline{6}{5}
    \horizline{6}{6}
    \vertline{2}{6}
    \horizline{2}{6}
    \vertline{3}{5}
    \horizline{3}{5}
    \vertline{4}{4}
}]
\null & \null & \null  & \null  & \null  & \null & \hskip 1em & \null \\
\null & \null & \null  & \null  & \null & - \ar[r, "{U-/U-,C-/C-}"' near end, bend right=50] & + \\
\null & \null & \null  & \null  & - \ar[r, "{U-/U-,C-/C-}"' near end, bend right=50] & + &  \\
\null & \null & \null  & - \ar[r, "U-/U-"' near end, bend right=50] & + & \raisebox{-0.5\baselineskip}{\makebox[1em][l]{\footnotesize $\alpha=1/C-$}} \ar[l] &   \\
\null & \null & \null & \null & \null & \null \\
\null & \null & \null & \null & \null & \null & \null \\
\null & \null & \null & \null & \null & \null &   \\
\null & \null & \null & \null & \null & \null &
\end{tikzcd}
\hfill
\begin{tikzcd}[cramped, sep=small, execute at end picture={
    \horizlinex{1}{2}{8}
    \vertline{2}{1}
    \horizline{2}{2}
    \vertline{3}{2}
    \horizline{3}{3}
    \vertline{4}{3}
    \horizline{4}{4}
    \vertline{5}{4}
    \horizline{5}{5}
    \vertline{6}{5}
    \horizline{6}{6}
    \vertline{2}{7}
    \horizline{2}{7}
    \vertline{3}{6}
    \horizline{3}{6}
    \vertline{4}{5}
    \horizline{4}{5}
}]
\null & \null & \null & \null  & \null  & \null  & \null & \null & \null \\
\null & \null & \null & \null  & \null  & \null & - \ar[r, "{U-/U-,C-/C-}"' near end, bend right=50] & + \\
\null & \null & \null & \null  & \null  & - \ar[r, "{U-/U-,C-/C-}"' near end, bend right=50] & + &   \\
\null & \null & \null & \null  & - \ar[r, "{U-/U-,C-/C-}"' near end, bend right=50] & + & \null &   \\
\null & \null & \null & \null & + & \raisebox{-0.5\baselineskip}{\makebox[1em][l]{\footnotesize $\alpha=1/U-$}} \ar[l] & \null & \null \\
\null & \null & \null & \null & \null & \null & \null & \null \\
\null & \null & \null & \null & \null & \null & \null &   \\
\null & \null & \null & \null & \null & \null & \null &
\end{tikzcd}
\hfill}
\begin{tikzcd}[cramped, sep=small, execute at end picture={
    \horizlinex{1}{2}{3}
    \vertline{2}{1}
    \horizline{2}{2}
    \vertline{3}{2}
    \horizline{3}{3}
}]
\null & \null & \null & \null \\
\null & + & \makebox[1em][l]{\footnotesize $\alpha=1/U-$} \ar[l] & \null \\
\null & \null & \null & \null \\
\null & \null & \null & \null
\end{tikzcd}
\caption{Generic insertion diagrams for shifted colun insertion}
\label{fig:psi-shifted-column}
\end{figure}
Expressed as an algorithm, it is straightforward:  insert the number
using column insertion,
starting in the first column into which it can be inserted,
circling the number in the $P$ tableau if it
was initially inserted into an off-diagonal cell.
An example growth diagram is shown in
figure~\ref{fig:growth-shifted-column}.
\begin{figure}
\growthtableauxh{%
\begin{tikzcd}[ampersand replacement=\&,sep=small]
\emptyset \ar[rr,"U"]\ar[dd] \&\&
  1 \ar[rr,"C"]\ar[dd] \&\&
  2 \ar[rr,"C"]\ar[dd] \&\&
  3 \ar[rr,"U"]\ar[dd] \&\&
  31 \ar[dd] \\
\&  \&\& X \&\&  \&\&  \& \\
\emptyset \ar[rr,"U"]\ar[dd] \&\&
  1 \ar[rr]\ar[dd] \&\&
  1 \ar[rr,"C"]\ar[dd] \&\&
  2 \ar[rr,"U"]\ar[dd] \&\&
  21 \ar[dd] \\
\&  \&\&  \&\&  \&\& X \& \\
\emptyset \ar[rr,"U"]\ar[dd] \&\&
  1 \ar[rr]\ar[dd] \&\&
  1 \ar[rr,"C"]\ar[dd] \&\&
  2 \ar[rr]\ar[dd] \&\&
  2 \ar[dd] \\
\&  \&\&  \&\& X \&\&  \& \\
\emptyset \ar[rr,"U"]\ar[dd] \&\&
  1 \ar[rr]\ar[dd] \&\&
  1 \ar[rr]\ar[dd] \&\&
  1 \ar[rr]\ar[dd] \&\&
  1 \ar[dd] \\
\& X \&\&  \&\&  \&\&  \& \\
\emptyset \ar[rr] \&\&
  \emptyset \ar[rr] \&\&
  \emptyset \ar[rr] \&\&
  \emptyset \ar[rr] \&\&
  \emptyset 
\end{tikzcd}
}{%
\begin{ytableau}
1 & 2^\circ & 3^\circ \\
\none & 4
\end{ytableau}
}{%
\begin{ytableau}
1 & 2 & 4 \\
\none & 3
\end{ytableau}
}
\caption[Growth for the permutation for McLarnan's shifted colun
insertion]%
{Growth for the permutation $(1, 3, 4, 2)$ for  McLarnan's shifted colun
insertion:
(a) growth diagram, (b) $P$ tableau, and (c) $Q$ tableau}
\label{fig:growth-shifted-column}
\end{figure}

McLarnan constructed the algorithm to be ``symmetric'', that is as
close as possible to being inversion self-dual.
Inversion self-duality can't be achieved, because shifted tableau do not admit the
weight $w(x) = 1$ and so the multiplicity of the $G_1$ edges and
$G_2$ edges cannot be the same.
This prevents the $P$ and $Q$ tableau from being interchanged in the
strict sense.
This is apparent in the generic insertion diagrams and the example
growth diagram.

However, inverting the permutation does swap the $P$ and $Q$ tableaux
as long as the circling is ignored.  Moreover, the circles in the
permutation's $P$ tableau are effectively moved to the elements in the
inverse's $P$ tableau (the permutation's $Q$ tableau) that are the
corresponding elements in the inverse permutation.
For example, the inverse of the permutation in
figure~\ref{fig:growth-shifted-column} is shown in
figure~\ref{fig:growth-shifted-column-1}.
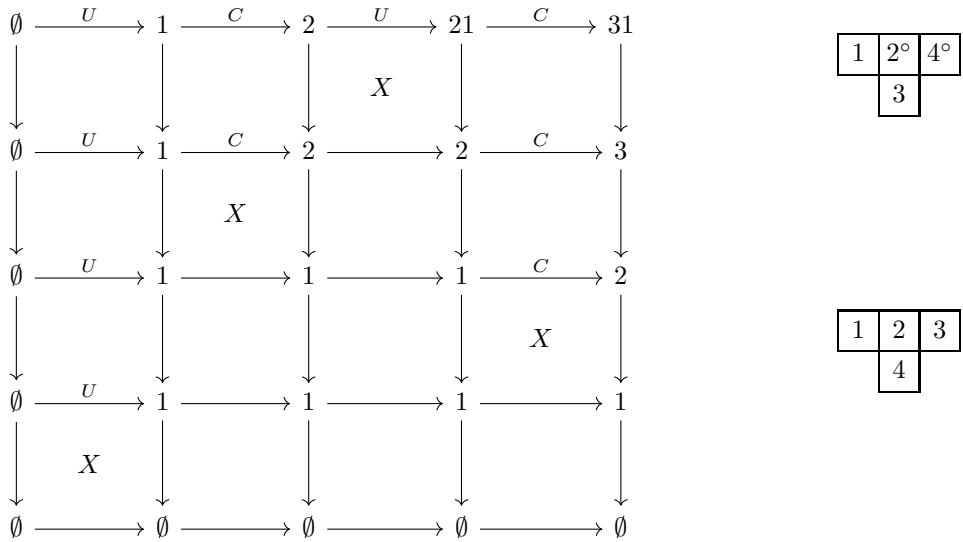
\begin{figure}
\growthtableauxh{%
\begin{tikzcd}[ampersand replacement=\&,sep=small]
\emptyset \ar[rr,"U"]\ar[dd] \&\&
  1 \ar[rr,"C"]\ar[dd] \&\&
  2 \ar[rr,"U"]\ar[dd] \&\&
  21 \ar[rr,"C"]\ar[dd] \&\&
  31 \ar[dd] \\
\&  \&\&  \&\& X \&\&  \& \\
\emptyset \ar[rr,"U"]\ar[dd] \&\&
  1 \ar[rr,"C"]\ar[dd] \&\&
  2 \ar[rr]\ar[dd] \&\&
  2 \ar[rr,"C"]\ar[dd] \&\&
  3 \ar[dd] \\
\&  \&\& X \&\&  \&\&  \& \\
\emptyset \ar[rr,"U"]\ar[dd] \&\&
  1 \ar[rr]\ar[dd] \&\&
  1 \ar[rr]\ar[dd] \&\&
  1 \ar[rr,"C"]\ar[dd] \&\&
  2 \ar[dd] \\
\&  \&\&  \&\&  \&\& X \& \\
\emptyset \ar[rr,"U"]\ar[dd] \&\&
  1 \ar[rr]\ar[dd] \&\&
  1 \ar[rr]\ar[dd] \&\&
  1 \ar[rr]\ar[dd] \&\&
  1 \ar[dd] \\
\& X \&\&  \&\&  \&\&  \& \\
\emptyset \ar[rr] \&\&
  \emptyset \ar[rr] \&\&
  \emptyset \ar[rr] \&\&
  \emptyset \ar[rr] \&\&
  \emptyset 
\end{tikzcd}
}{%
\begin{ytableau}
1 & 2^\circ & 4^\circ \\
\none & 3
\end{ytableau}
}{%
\begin{ytableau}
1 & 2 & 3 \\
\none & 4
\end{ytableau}
}
\caption[Growth for the permutation for McLarnan's shifted colun
insertion]%
{Growth for the permutation $(1, 4, 2, 3)$ for  McLarnan's shifted colun
insertion, which is the inverse of the permutation in
\protect\ref{fig:growth-shifted-column}}
\label{fig:growth-shifted-column-1}
\end{figure}

McLarnan arranged for this inversion near-duality by arranging
that the movement of an element upon bumping is not dependent on its
circled status.  In the context of a growth diagram, the labels on the
$G_1$ edges pass northward through the diagram, neither being changed
by more northern cells nor affecting the nodes of more northern cells.

\paragraph{Dual shifted column insertion}
For our purposes, the inversion-dual of the shifted colun algorithm, which
we call \textit{dual shifted column insertion}, is more interesting.  Indeed,
we believe that it deserves more study in the theory of tableau
insertion algorithms than it has received.

Because shifted colun insertion is nearly inversion self-dual, the only
difference from shifted colun insertion is that labels are
attached to the $G_2$ edges rather than the $G_1$ edges, and so
circles appear on the elements of the $Q$ tableau.

See figure~\ref{fig:psi-dual-shifted-column} for the generic insertion
diagrams, and figure~\ref{fig:growth-dual-shifted-column} for an example
growth diagram.
\begin{figure}
\hbox to \linewidth{\hfill
\begin{tikzcd}[cramped, sep=small, execute at end picture={
    \horizlinex{1}{2}{7}
    \vertline{2}{1}
    \horizline{2}{2}
    \vertline{3}{2}
    \horizline{3}{3}
    \vertline{4}{3}
    \horizline{4}{4}
    \vertline{5}{4}
    \horizline{5}{5}
    \vertline{6}{5}
    \horizline{6}{6}
    \vertline{2}{6}
    \horizline{2}{6}
    \vertline{3}{5}
    \horizline{3}{5}
    \vertline{4}{4}
}]
\null & \null & \null  & \null  & \null  & \null & \hskip 1em & \null \\
\null & \null & \null  & \null  & \null & - \ar[r, "{-U/-U,-C/-C}"' near end, bend right=50] & + \\
\null & \null & \null  & \null  & - \ar[r, "{-U/-U,-C/-C}"' near end, bend right=50] & + &  \\
\null & \null & \null  & - \ar[r, "-U/-U"' near end, bend right=50] & + & \raisebox{-0.5\baselineskip}{\makebox[1em][l]{\footnotesize $\alpha=1/-C$}} \ar[l] &   \\
\null & \null & \null & \null & \null & \null \\
\null & \null & \null & \null & \null & \null & \null \\
\null & \null & \null & \null & \null & \null &   \\
\null & \null & \null & \null & \null & \null &
\end{tikzcd}
\hfill
\begin{tikzcd}[cramped, sep=small, execute at end picture={
    \horizlinex{1}{2}{8}
    \vertline{2}{1}
    \horizline{2}{2}
    \vertline{3}{2}
    \horizline{3}{3}
    \vertline{4}{3}
    \horizline{4}{4}
    \vertline{5}{4}
    \horizline{5}{5}
    \vertline{6}{5}
    \horizline{6}{6}
    \vertline{2}{7}
    \horizline{2}{7}
    \vertline{3}{6}
    \horizline{3}{6}
    \vertline{4}{5}
    \horizline{4}{5}
}]
\null & \null & \null & \null  & \null  & \null  & \null & \null & \null \\
\null & \null & \null & \null  & \null  & \null & - \ar[r, "{-U/-U,-C/-C}"' near end, bend right=50] & + \\
\null & \null & \null & \null  & \null  & - \ar[r, "{-U/-U,-C/-C}"' near end, bend right=50] & + &   \\
\null & \null & \null & \null  & - \ar[r, "{-U/-U,-C/-C}"' near end, bend right=50] & + & \null &   \\
\null & \null & \null & \null & + & \raisebox{-0.5\baselineskip}{\makebox[1em][l]{\footnotesize $\alpha=1/-U$}} \ar[l] & \null & \null \\
\null & \null & \null & \null & \null & \null & \null & \null \\
\null & \null & \null & \null & \null & \null & \null &   \\
\null & \null & \null & \null & \null & \null & \null &
\end{tikzcd}
\hfill}
\begin{tikzcd}[cramped, sep=small, execute at end picture={
    \horizlinex{1}{2}{3}
    \vertline{2}{1}
    \horizline{2}{2}
    \vertline{3}{2}
    \horizline{3}{3}
}]
\null & \null & \null & \null \\
\null & + & \makebox[1em][l]{\footnotesize $\alpha=1/-U$} \ar[l] & \null \\
\null & \null & \null & \null \\
\null & \null & \null & \null
\end{tikzcd}
\caption{Generic insertion diagram for dual shifted column insertion}
\label{fig:psi-dual-shifted-column}
\end{figure}
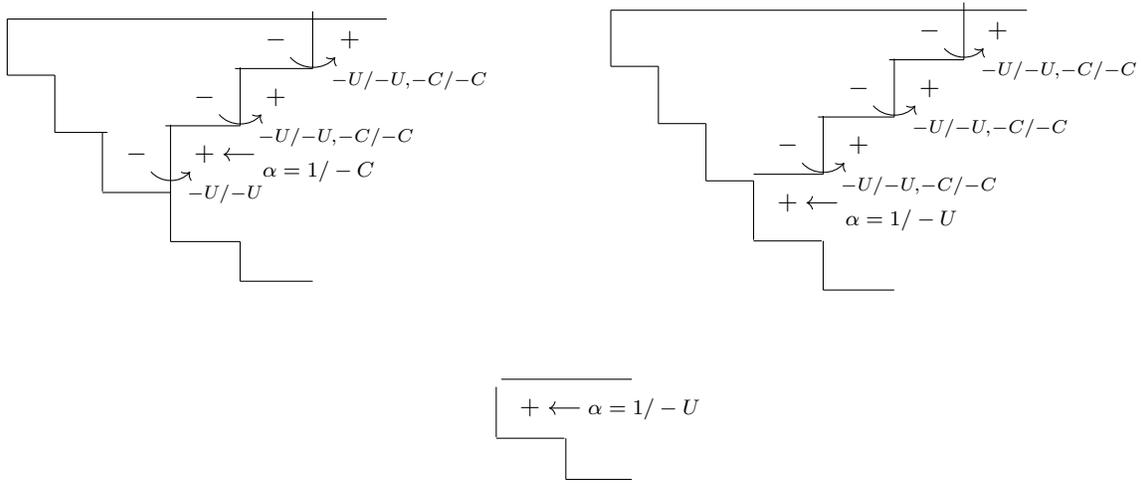
\begin{figure}
\growthtableauxh{%
\begin{tikzcd}[ampersand replacement=\&,sep=small]
\emptyset \ar[rr]\ar[dd] \&\&
  1 \ar[rr]\ar[dd] \&\&
  2 \ar[rr]\ar[dd,"C"] \&\&
  3 \ar[rr]\ar[dd,"C"] \&\&
  31 \ar[dd,"C"] \\
\&  \&\& X \&\&  \&\&  \& \\
\emptyset \ar[rr]\ar[dd] \&\&
  1 \ar[rr]\ar[dd] \&\&
  1 \ar[rr]\ar[dd] \&\&
  2 \ar[rr]\ar[dd] \&\&
  21 \ar[dd,"U"] \\
\&  \&\&  \&\&  \&\& X \& \\
\emptyset \ar[rr]\ar[dd] \&\&
  1 \ar[rr]\ar[dd] \&\&
  1 \ar[rr]\ar[dd] \&\&
  2 \ar[rr]\ar[dd,"C"] \&\&
  2 \ar[dd,"C"] \\
\&  \&\&  \&\& X \&\&  \& \\
\emptyset \ar[rr]\ar[dd] \&\&
  1 \ar[rr]\ar[dd,"U"] \&\&
  1 \ar[rr]\ar[dd,"U"] \&\&
  1 \ar[rr]\ar[dd,"U"] \&\&
  1 \ar[dd,"U"] \\
\& X \&\&  \&\&  \&\&  \& \\
\emptyset \ar[rr] \&\&
  \emptyset \ar[rr] \&\&
  \emptyset \ar[rr] \&\&
  \emptyset \ar[rr] \&\&
  \emptyset 
\end{tikzcd}
}{%
\begin{ytableau}
1 & 2 & 3 \\
\none & 4
\end{ytableau}
}{%
\begin{ytableau}
1 & 2^\circ & 4^\circ \\
\none & 3
\end{ytableau}
}
\caption[Growth for the permutation $(1, 3, 4, 2)$ for dual shifted
  column insertion]%
{Growth for the permutation $(1, 3, 4, 2)$ for dual shifted column insertion:
(a) growth diagram, (b) $P$ tableau, and (c) $Q$ tableau}
\label{fig:growth-dual-shifted-column}
\end{figure}
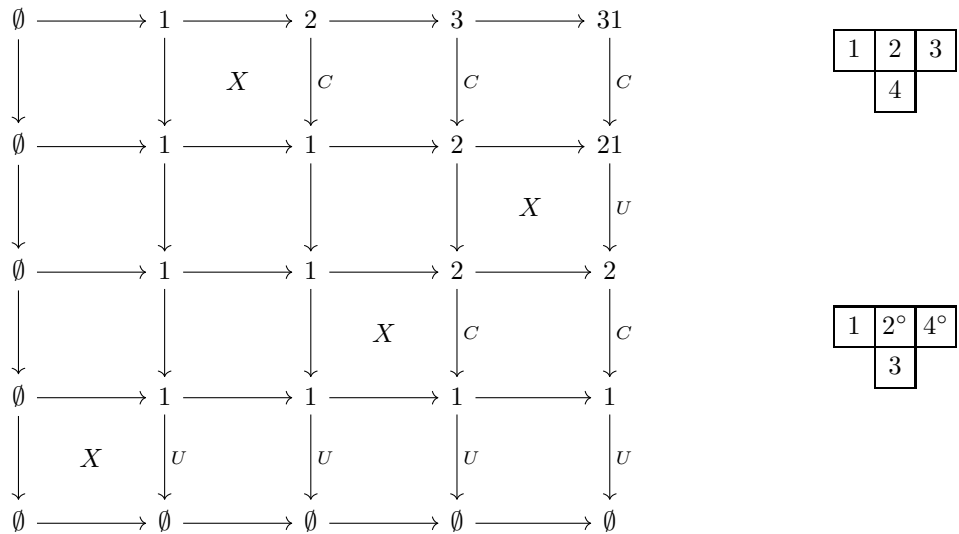

This algorithm is close in style to the first \ref{sec:sagan} and
second \ref{sec:shifted} shifted algorithms.  It is simpler and more
intuitive than either of them, and is
nearly inversion self-dual (in the same way as shifted column insertion).
And since it only applies labels to $G_2$ edges, it is compatible with
the original formulation of growth diagrams \ref{sec:growth-diagrams}.

It is a matter of interest why such a straightforward variant of a
described algorithm has apparently not been described before
now.  That appears to be due to the general neglect of 
McLarnan's shifted colun insertion.  One driver of that may be
that because shifted colun insertion puts circles on
elements of $P$ rather than $Q$ it is less
congruent with other early algorithms for shifted tableaux, which put
circles only on elements of $Q$, and consequently it is not compatible with
the original formulation of growth diagrams.  Putting circles on
elements of $P$ is natural in McLarnan's approach, but not so putting
circles on elements of $Q$.

Early work on other shifted algorithms tended to put circles on
elements of $Q$ as that allowed circles on elements of the generalized
permutation, which then propagate to the elements of $P$.

Conversely, Sagan's approaches via the Knuth relations
\cite{Sag1979b} and the ``lifting'' operation \cite{Sag1987a} and
Worley's approach via jeu de taquin \cite{Wor1984a} do not seem to
naturally lead to this algorithm.

Another factor is likely that Sagan and Worley both started searching
for algorithms based on row insertion, as that is the most
commonly-used unshifted algorithm, and both succeeded at that.  Neither
seems to have considered starting with column insertion.  In the
context of unshifted tableaux, row and column insertion are simply
transpose-dual, which suggests that attacking the problem of shifted
tableaux starting with either row or column insertion will give the
same results.  But with shifted tableaux row and column insertion are
distinctly different.

\section{Software}\label{sec:software}

There are two ancillary software files:
\begin{enumerate}
  \item[] \texttt{growth.py} -- (Python) A large number of classes containing
    machinery for executing tableau insertion algorithms.
  \item[] \texttt{growth-driver.py} -- (Python) A driver program that uses
    the machinery in \texttt{growth.py} for several dozen tests/examples,
    including many examples in this paper.
\end{enumerate}

\vspace{3em}

\end{document}